\documentclass[10pt]{amsart}
\usepackage{amsmath}
\usepackage{amsxtra}
\usepackage{amscd}
\usepackage{amsthm}
\usepackage{amsfonts}
\usepackage{amssymb}
\usepackage{eucal}
\usepackage[all]{xy}
\usepackage{graphicx}

\newtheorem{prop}[subsubsection]{Proposition}

\newtheorem{cor}[subsubsection]{Corollary}
\newtheorem{lem}[subsubsection]{Lemma}

\numberwithin{equation}{section}

\theoremstyle{remark}
\newtheorem{rem}[subsubsection]{Remark}
\newtheorem{ex}[subsubsection]{Example}

\newcommand{\lemref}[1]{Lemma~\ref{#1}}

\newcommand{\secref}[1]{Sect.~\ref{#1}}
\newcommand{\corref}[1]{Corollary~\ref{#1}}
\newcommand{\propref}[1]{Proposition~\ref{#1}}

\newcommand{\nc}{\newcommand}

\nc{\ssec}{\subsection}
\nc{\sssec}{\subsubsection}

\nc{\renc}{\renewcommand}
\nc{\on}{\operatorname}
\nc\ol{\overline}
\nc\wt{\widetilde}
\nc{\Loc}{\on{Loc}}
\nc{\Bun}{\on{Bun}}
\nc{\BQ}{{\mathbb{Q}}}
\nc{\BA}{{\mathbb{A}}}
\nc{\BC}{{\mathbb{C}}}
\nc{\BH}{{\mathbb{H}}}
\nc{\BG}{{\mathbb{G}}}
\nc{\BK}{{\mathbb{K}}}
\nc{\BN}{{\mathbb{N}}}
\nc{\BD}{{\mathbb{D}}}
\nc{\BV}{{\mathbb{V}}}
\nc{\BL}{{\mathbb{L}}}
\nc{\CA}{{\mathcal{A}}}
\nc{\CC}{{\mathcal{C}}}
\nc{\cD}{{\mathcal{D}}}
\nc{\CG}{{\mathcal{G}}}
\nc{\CI}{{\mathcal{I}}}
\nc{\CJ}{{\mathcal{J}}}
\nc{\CO}{{\mathcal{O}}}
\nc{\CP}{{\mathcal{P}}}
\nc{\CR}{{\mathcal{R}}}
\nc{\CV}{{\mathcal{V}}}
\nc{\CW}{{\mathcal{W}}}
\nc{\CK}{{\mathcal{K}}}
\nc{\CM}{{\mathcal{M}}}
\nc{\CN}{{\mathcal{N}}}
\nc{\CL}{{\mathcal{L}}}
\nc{\CQ}{{\mathcal{Q}}}
\nc{\CF}{{\mathcal{F}}}
\nc{\CX}{{\mathcal{X}}}
\nc{\CY}{{\mathcal{Y}}}
\nc{\CZ}{{\mathcal{Z}}}

\nc{\D}{{\mathcal{D}}}
\nc{\fd}{{\mathfrak{d}}}
\nc{\fg}{{\mathfrak{g}}}
\nc{\fD}{{\mathfrak{D}}}
\nc{\fh}{{\mathfrak{h}}}
\nc{\fl}{{\mathfrak{l}}}
\nc{\fn}{{\mathfrak{n}}}
\nc{\fp}{{\mathfrak{p}}}
\nc{\fs}{{\mathfrak{s}}}
\nc{\ft}{{\mathfrak{t}}}
\nc{\sM}{{\mathsf M}}
\nc{\ppart}{(\!(t)\!)}
\nc{\qqart}{[\![t]\!]}
\nc{\hg}{{\widehat\fg}}
\nc{\sA}{{\mathsf A}}
\nc{\sB}{{\mathsf B}}
\nc{\sF}{{\mathsf F}}
\nc{\sG}{{\mathsf G}}
\nc{\sk}{{\mathsf k}}
\nc{\sj}{{\mathsf j}}

\nc{\bC}{{\mathbf{C}}}
\nc{\bZ}{{\mathbf{Z}}}
\nc{\bD}{{\mathbf{D}}}
\nc{\bO}{{\mathbf{O}}}
\nc{\bU}{{\mathbf{U}}}
\nc{\be}{{\mathbf{e}}}
\nc{\br}{{\mathbf{r}}}
\nc{\bM}{{\mathbf{M}}}
\nc{\bA}{{\mathbf{A}}}
\nc{\bK}{{\mathbf{K}}}

\nc{\oY}{\overset{\circ}Y{}}
\nc{\oX}{\overset{\circ}X{}}
\nc{\oS}{\overset{\circ}S{}}
\nc{\of}{\overset{\circ}f{}}
\nc{\oCX}{\overset{\circ}\CX{}}
\nc{\oCY}{\overset{\circ}\CY{}}
\nc{\oCZ}{\overset{\circ}\CZ{}}

\nc{\fW}{{\mathfrak{W}}}
\nc{\reg}{{\text{\rm reg}}}
\nc{\nilp}{{\text{\rm nilp}}}
\nc{\cG}{{\check{G}}}
\nc{\cB}{{\check{B}}}
\nc{\cg}{{\check{\fg}}}
\nc{\cb}{{\check{\fb}}}
\nc{\cn}{{\check{\fn}}}
\nc{\mer}{{\on{mer}}}
\nc{\Const}{\mathsf{Const}}
\nc{\Whit}{\on{Whit}}
\nc{\KL}{\on{KL}}
\nc{\FS}{\on{FS}}
\nc{\LocSys}{\on{LocSys}}
\nc{\QCoh}{\on{QCoh}}
\nc{\Coh}{\on{Coh}}
\nc{\IndCoh}{\on{IndCoh}}
\nc{\Cat}{\on{Cat}}
\nc{\Op}{\on{Op}}
\nc{\Gr}{\on{Gr}}
\nc{\Fl}{\on{Fl}}
\nc{\Rep}{\on{Rep}}
\renc{\mod}{{\on{-mod}}}
\nc{\Conn}{\on{Conn}}
\nc{\unit}{{\mathbf{1}}}

\nc{\Hom}{\on{Hom}}
\nc{\End}{\on{End}}
\nc{\Vect}{\on{Vect}}
\nc{\Av}{\on{Av}}
\nc{\Ind}{\on{Ind}}
\nc{\Spec}{\on{Spec}}
\nc{\KG}{K\backslash G}
\nc{\comult}{{co\text{-}mult}}
\nc{\counit}{{co\text{-}unit}}
\nc{\uHom}{{\underline{\Hom}}}
\nc{\dgSch}{\on{DGSch}}
\nc{\dgindSch}{\on{DGindSch}}
\nc{\indSch}{\on{indSch}}
\nc{\Sch}{\on{Sch}}
\nc{\affdgSch}{\on{DGSch}^{\on{aff}}}
\nc{\affSch}{\on{Sch}^{\on{aff}}}

\nc{\Groupoids}{\on{Grpd}}
\nc{\inftygroup}{\infty\on{-Grpd}}
\nc{\inftypic}{\infty\on{-PicGrpd}}
\nc{\inftyCat}{\infty\on{-Cat}}
\nc{\StinftyCat}{\on{DGCat}}
\nc{\MoninftyCat}{\infty\on{-Cat}^{Mon}}
\nc{\SymMoninftyCat}{\infty\on{-Cat}^{SymMon}}
\nc{\SymMonStinftyCat}{\on{DGCat}^{SymMon}}
\nc{\MonStinftyCat}{\on{DGCat}^{Mon}}
\nc{\inftystack}{\on{Stk}}
\nc{\inftystackalg}{Stk^{1\text{-}alg}}
\nc{\inftyprestack}{\on{PreStk}}
\nc{\inftydgnearstack}{\on{NearStk}}
\nc{\inftydgstack}{\on{Stk}}
\nc{\inftydgstackalg}{DGStk^{1\text{-}alg}}
\nc{\inftydgprestack}{\on{PreStk}}

\nc{\mmod}{{\on{-}{\mathbf{mod}}}}
\nc{\wh}{\widehat}

\nc{\nDG}{^{\leq n}\!\on{DG}}

\nc{\Maps}{\on{Maps}}
\nc{\CMaps}{{\mathcal Maps}}

\nc{\Lie}{\on{Lie}}
\nc{\triv}{{\mathbf{triv}}}

\nc{\dr}{{\on{dR}}}

\nc{\Crys}{\on{Crys}}
\nc{\oblv}{{\mathbf{oblv}}}
\nc{\ind}{{\mathbf{ind}}}
\nc{\Dmod}{\on{D-mod}}
\nc{\Pic}{\on{Pic}}
\nc{\Ge}{\on{Ge}}
\nc{\Tw}{\on{Tw}}
\nc{\sotimes}{\overset{!}\otimes}
\nc{\bDelta}{{\mathbf\Delta}}

\begin{document}

\title[Crystals and D-modules]{Crystals and D-modules}

\author{Dennis Gaitsgory and Nick Rozenblyum}

\date{\today}

\dedicatory{To the memory of A.~Todorov.}

\begin{abstract}
The goal of this paper is to develop the notion of crystal in the context
of derived algebraic geometry, and to connect crystals to more classical
objects such as D-modules.
\end{abstract}

\maketitle

\tableofcontents

\section*{Introduction}

\ssec{Flat connections, D-modules and crystals}

\sssec{}

Let $M$ be a smooth manifold with a vector bundle $V$. Recall that a
flat connection on $V$ is a map
$$ \nabla: V \rightarrow V\otimes \Omega^1_M$$
satisfying the Leibniz rule, and such that the curvature $[\nabla,\nabla]=0$. Dualizing the
connection map, we obtain a map
$$ T_M \otimes V \rightarrow V.$$
The flatness of the connection implies that this makes $V$ into a module over the Lie algebra
of vector fields. Equivalently, we obtain that $V$ is module over the algebra $\on{Diff}_M$ of
differential operators on $M$.

\medskip

This notion generalizes immediately to smooth algebraic varieties in characteristic
zero. On such a variety a D-module is defined as a module over the sheaf of differential
operators which is quasi-coherent as an $\CO$-module. The notion of D-module on an algebraic variety thus
generalizes the notion of vector bundle with a flat connection, and encodes the data of a system of linear
differential equations with polynomial coefficients. The study of D-modules on smooth
algebraic varieties is a very rich theory, with applications to numerous fields such as
representation theory. Many of the ideas from the differential geometry of vector bundles with a flat connection
carry over to this setting.

\medskip

However, the above approach to D-modules presents a number of difficulties. For example, 
one needs to consider sheaves with a flat connection on singular schemes in addition 
to smooth ones. While the algebra of differential operators is well-defined on a singular variety, 
the category of modules over it is not \emph{the} category that we are interested in
(e.g., the algebra in question is not in general Noetherian). In another direction, even for a smooth 
algebraic variety, it is not clear how to define connections on objects that are not linear, 
e.g., sheaves of categories. 

\sssec{Parallel transport}

The idea of a better definition comes from another interpretation of the notion of flat connection on
a vector bundle in the context of differential geometry, namely, that of parallel transport:

\medskip

Given a vector bundle with a flat connection $V$ on a smooth manifold $M$, and a path
$\gamma: [0,1]\rightarrow M$, we obtain an isomorphism $$ \Pi_{\gamma}: V_{\gamma(0)}\simeq
V_{\gamma(1)} $$ of the fibers of $V$ at the endpoints, which only depends on the homotopy
class of the path. 
We can rephrase this construction as follows. Let $B \subset M$ be a small ball inside $M$.
Since the parallel transport isomorphism only depends on the homotopy class of the path, and $B$ is
contractible, we obtain a coherent identification of fibers of $V$ $$ V_x \simeq V_y $$ for points
$x,y\in B$ by considering paths inside $B$. So, roughly, the data of a connection gives an identification of fibers at ``nearby''
points of the manifold.


\medskip

Building on this idea, Grothendieck \cite{Groth} gave a purely algebraic analogue of the notion of parallel 
transport, using the theory of schemes (rather than just varieties) in an essential
way: he introduced the relation of infinitesimal closeness for $R$-points
of a scheme $X$. Namely, two $R$-points $x, y: \Spec(R)\rightarrow X$ are infinitesimally close if the
restrictions to $\Spec({}^{red}\!R)$ agree, where $^{red}\!R$ is the quotient of $R$ by its nilradical. 

\medskip

A \emph{crystal} on $X$ is by definition a quasi-coherent sheaf on $X$ which is equivariant with respect to the relation
of infinitesimal closeness. More preciesly, a crystal on $X$ is a quasi-coherent sheaf $\CF$ with the
additional data of isomorphisms
$$x^*( \CF) \simeq y^* (\CF)$$
for any two infinitesimally close points $x,y: \Spec(R)\rightarrow X$ satisfying a cocycle condition.

\medskip

Grothendieck showed that on a smooth algebraic variety, the abelian category of crystals
is equivalent to that of left modules over the ring of differential operators. In this way, crystals
give a more fundamental definition of sheaves with a flat connection.

\medskip

A salient feature of the category of crystals is that Kashiwara's lemma is built into its definition: 
for a closed embedding of schemes $i: Z\rightarrow X$, the category of crystals on $Z$ is
equivalent to the category of crystals on $X$, which are set-theoretically supported on $Z$.
This observation allows us to reduce the study of crystals on schemes to the case of smooth
schemes, by (locally) embedding a given scheme into a smooth one. 

\sssec{}  \label{sss:simpson}
In this paper, we develop the theory of crystals in the context of derived algebraic geometry, where
instead of ordinary rings one considers derived rings, i.e., $E_\infty$ ring spectra. Since we work
over a field $k$ of characteristic zero, we shall use connective commutative DG $k$-algebras as our
model of derived rings (accordingly, we shall use the term ``DG scheme" rather than ``derived
scheme"). The key idea is that one should regard higher homotopy groups of a derived ring as a
generalization of nilpotent elements.

\medskip

Thus, following Simpson \cite{Simpson}, 
for a DG scheme $X$, we define its de Rham prestack 
$X_\dr$ to be the functor $$ X_\dr:R \mapsto X(^{red,cl}\!R)$$
on the category of derived rings $R$, where $$^{red,cl}\!R:={}^{red}(\pi_0(R))$$ is the reduced ring corresponding to the
underlying classical ring of $R$. I.e., $X_\dr$ is a \emph{prestack} in the terminology of \cite{Stacks}. 

\medskip 

We define crystals on $X$ as quasi-coherent sheaves on the prestack $X_\dr$. See, \cite[Sect. 2]{Lu1} for the theory
of quasi-coherent sheaves in prestacks, or \cite[Sect. 1.1]{QCoh} for a brief review. 

\medskip

The above definition does not
coincide with one of Grothendieck mentioned earlier: the latter specifies a map $\Spec(R)\to X$ up to an equivalence
relation, and the former only a map $\Spec({}^{red,cl}\!R)\to X$. However, we will show that
for $X$ which is \emph{eventually coconnective}, i.e., if its structure ring has only finitely many
non-zero homotopy groups, the two definitions of a crystal are equivalent.
\footnote{When $X$ is not eventually coconnective, the two notions are different,
and the correct one is the one via $X_\dr$.}

\sssec{}
Even though the category of crystals is equivalent to that of D-modules, it offers a more flexible
framework in which to develop the theory.  The definition immediately extends to non-smooth
schemes, and the corresponding category is well-behaved (for instance, the category of crystals on any
scheme is locally Noetherian).

\medskip

Let $f:X\to Y$ be a map of DG schemes. We will construct the natural pullback functor
$$f^\dagger:\Crys(Y)\to \Crys(X).$$
In fact, we shall extend the assignment $X\mapsto \Crys(X)$ to a functor from the category $\dgSch^{\on{op}}$
to that of stable $\infty$-categories. The latter will enable us to define crystals not just on DG schemes,
but on arbitrary prestacks.  

\medskip

Furthermore, the notion of crystal immediately extends to a non-linear and categorified setting. Namely, 
we can just as well define a crystal of schemes or a crystal of categories over $X$.

\ssec{Left crystals vs. right crystals}

\sssec{}

Recall that on a smooth algebraic variety $X$, in addition to
usual (i.e., left) D-modules, one can also consider the category of right D-modules. The
two categories are equivalent: the corresponding functor is given by tensoring with
the dualizing line bundle $\omega_X$ over the ring of functions. However, this equivalence does 
not preserve the forgetful functor to quasi-coherent sheaves. For this reason, we can consider an 
abstract category of D-modules, with two different realization functors to quasi-coherent sheaves. 
In the left realization, the D-module pullback functor becomes the $*$-pullback functor on 
quasi-coherent sheaves, and in the right realization, it becomes the !-pullback functor.

\medskip

It turns out that the ``right" realization has several advantages over the ``left" one. Perhaps
the main advantage is that the ``right" realization endows the category of D-modules with a t-structure 
with very favorable functorial properties. In particular, this t-structure becomes the perverse
t-structure under the Riemann-Hilbert correspondence. 

\sssec{}

One can then ask whether there are also ``left'' and ``right'' crystals on arbitrary DG schemes.  
It turns out that indeed both categories are defined very generally.

\medskip

Left crystals are what we defined in \secref{sss:simpson}. However, in order to define right crystals,
we need to replace the usual category of quasi-coherent sheaves by its renormalized version, 
the category of ind-coherent sheaves introduced in \cite{IndCoh}. 

\medskip

The category
$\IndCoh(X)$ is well-behaved for (derived) schemes that are (almost) locally of finite type, so
right crystals will only be defined on DG schemes, and subsequently, on prestacks with this property.

\medskip

Let us recall from \cite[Sect. 5]{IndCoh} that for a map $f:X\to Y$ between DG schemes,
we have the !-pullback functor
$$f^!:\IndCoh(Y)\to \IndCoh(X).$$

The assignment $X\mapsto \IndCoh(X)$ is a functor from the category $\dgSch^{\on{op}}$ to
that of stable $\infty$-categories and thus can be extended to a functor out of the category
of prestacks. 

\medskip

For a DG scheme $X$, we define the category of right crystals $\Crys^r(X)$ as $\IndCoh(X_\dr)$. We can also reformulate
this definition \`a la Grothendieck by saying that a right crystal on $X$ is an object $\CF\in \IndCoh(X)$, 
together with an identification
\begin{equation}\label{e:right crystal}
x^!(\CF) \simeq y^!(\CF)
\end{equation}
for every pair of infinitesimally close points $x,y: \Spec(R)\rightarrow X$ satisfying (the
$\infty$-category version of) the cocycle condition. It can be shown that, unlike in the case of left crystals,
this does give an equivalent definition of right crystals without any coconnectivity assumptions.

\sssec{}

Now that the category of right crystals is defined, we can ask whether it is equivalent to
that of left crystals. The answer also turns out to be ``yes." Namely, for any DG scheme $X$
almost of finite type, tensoring by the dualizing complex $\omega_X$ defines a functor
$$\Upsilon_X:\QCoh(X)\to \IndCoh(X)$$
that intertwines the $*$-pullback on quasi-coherent sheaves and the !-pullback on ind-coherent sheaves. 

\medskip

Although the functor $\Upsilon_S$ is not an equivalence for an individual $S$ unless $S$ is smooth,
the totality of such maps for DG schemes mapping to the de Rham prestack of $X$ define an equivalence
between left and right crystals.

\medskip

Thus, just as in the case of smooth varieties, we can think that to each DG scheme $X$ we attach 
the category $\Crys(X)$ equipped with two ``realization" functors
$$
\xy
(-20,0)*+{\QCoh(X)}="A";
(20,0)*+{\IndCoh(X)}="B";
(0,20)*+{\Crys(X).}="C";
{\ar@{->}_{\oblv^l_X} "C";"A"};
{\ar@{->}^{\oblv^r_X} "C";"B"};
{\ar@{->}^{\Upsilon_X} "A";"B"};
\endxy
$$

However, in the case of non-smooth schemes, the advantages of the t-structure on $\Crys(X)$ that is associated
with the ``right'' realization become even more pronounced. 

\sssec{Historical remark}

To the best of our knowledge, the approach to D-modules via right crystals was first suggested by A.~Beilinson 
in the early 90's, at the level of abelian categories.

\medskip

For some time after that it was mistakenly believed that one cannot use left crystals
to define D-modules, because of the incompatibility of the t-structures. However, it was explained by J.~Lurie, that if
one forgoes the t-structure and defines the corresponding stable $\infty$-category right away, 
left crystals work just as well.

\ssec{The theory of crystals/D-modules}

Let us explain the formal structure of the theory, as developed in this paper, and its sequel 
\cite{Funct}.  

\sssec{}
To each prestack (locally almost of finite type) $\CY$, we assign
a stable $\infty$-category
$$ \CY \rightsquigarrow \Crys(\CY) .$$

This category has two realization functors: a left realization functor to $\QCoh(\CY)$, and a right realization functor to 
$\IndCoh(\CY)$ which are related via the following commutative diagram

$$
\xy
(-20,0)*+{\QCoh(\CY)}="A";
(20,0)*+{\IndCoh(\CY)}="B";
(0,20)*+{\Crys(\CY).}="C";
{\ar@{->}_{\oblv^l_\CY} "C";"A"};
{\ar@{->}^{\oblv^r_\CY} "C";"B"};
{\ar@{->}^{\Upsilon_\CY} "A";"B"};
\endxy
$$
where $\Upsilon_\CY$ is the functor $\QCoh(\CY)\rightarrow \IndCoh(\CY)$ given by tensoring by the dualizing complex $\omega_\CY$.

\sssec{}
The assignment of $\Crys(\CY)$ to $\CY$ is functorial in a number of ways. For a map $f: \CY_1\rightarrow \CY_2$, there is
a pullback functor
$$ f^\dagger: \Crys(\CY_2) \rightarrow \Crys(\CY_1) $$
which is functorial in $f$; i.e., this assignment gives a functor
$$\Crys^\dagger_{\on{PreStk}}:(\on{PreStk})^{\on{op}}\to \StinftyCat_{\on{cont}}.$$

The pullback functor on D-modules is compatible with the realization functors. Namely, we have
commutative diagrams

$$
\CD
\Crys(\CY_1)  @<{f^\dagger}<< \Crys(\CY_2)  \\
@V{\oblv^l_{\CY_1}}VV    @VV{\oblv^l_{\CY_2}}V   \\
\QCoh(\CY_1)  @<{f^*}<<  \QCoh(\CY_2)
\endCD
$$
and
$$
\CD
\Crys(\CY_1)  @<{f^\dagger}<< \Crys(\CY_2)  \\
@V{\oblv^r_{\CY_1}}VV    @VV{\oblv^r_{\CY_2}}V   \\
\IndCoh(\CY_1)  @<{f^!}<<  \IndCoh(\CY_2).
\endCD
$$

Furthermore, this compatibility is itself functorial in $f$; i.e. we have a naturally commutative diagram of functors
$$ \xymatrix{ & \Crys^\dagger_{\on{PreStk}} \ar[rd]^{\oblv^r}\ar[ld]_{\oblv^l} & \\ 
\QCoh^*_{\on{PreStk}} \ar[rr]^{\Upsilon}&& \IndCoh^!_{\on{PreStk}}}.$$

\sssec{}

The above portion of the theory is constructed in the present paper. I.e., this paper is concerned with the assignment
$$\CY\rightsquigarrow \Crys(\CY)$$
and the operation of pullback. Thus, in this paper, we develop the local theory of crystals/D-modules.

\medskip

However, in addition to the functor $f^\dagger$, we expect to also have a pushforward functor $f_{\dr,*}$, and the two must satisfy
various compatibility relations. The latter will be carried out in \cite{Funct}. However, let us indicate the main ingredients
of the combined theory:

\sssec{}

For a schematic quasi-compact map between prestacks $f: \CY_1\rightarrow \CY_2$, there is the de Rham pushforward functor
$$ f_{\dr,*}: \Crys(\CY_1)\rightarrow \Crys(\CY_2) $$
which is functorial in $f$.  This assignment gives another functor
$$(\Crys_{\dr,*})_{\on{PreStk}_{\on{sch-qc}}}: \on{PreStk}_{\on{sch-qc}}\to \StinftyCat_{\on{cont}},$$
where $\on{PreStk}_{\on{sch-qc}}$ is the non-full subcategory of $\on{PreStk}$ obtained by restricting
$1$-morpisms to schematic quasi-compact maps.

\medskip

Let $\CY=X$ be a DG scheme\footnote{More generally, we can let $\CY$ be a prestack that \emph{admits deformation theory}.}. 
In this case, the forgetful functor
$$\oblv^r_\CY:\Crys(\CY)\to \IndCoh(\CY)$$ admits a left adjoint, denoted
$$ \ind^r_{\CY} : \IndCoh(\CY) \to \Crys(\CY),$$
and called the induction functor.

\medskip

The induction functor is compatible with de Rham pushforward.  Namely, we have a commutative diagram
$$ \xymatrix{ \IndCoh(\CY_1) \ar[r]^{f^{\IndCoh}_*}\ar[d]_{\ind^r_{\CY_1}} & 
\IndCoh(\CY_2) \ar[d]^{\ind^r_{\CY_2}} \\ \Crys(\CY_1) \ar[r]^{f_{\dr,*}} & \Crys(\CY_2).} $$
This compatibility is itself functorial, i.e. we have a natural transformation of functors
$$ \xymatrix{ (\IndCoh_{*})_{\on{PreStk}_{\on{sch-qc}}} \ar[r]^{\ind^r} & (\Crys_{\dr,*})_{\on{PreStk}_{\on{sch-qc}}} .} $$

\sssec{}

In the case when $f$ is proper, the functors $(f_{\dr,*}, f^\dagger)$ form an adjoint pair, and if $f$ is
smooth, the functors $(f^\dagger[-2n], f_{\dr,*})$ form an adjoint pair for $n$ the relative dimension of
$f$.

\medskip

In general, the two functors are not adjoint, but they satisfy a base change formula. As explained in
\cite[Sect. 5.1]{IndCoh}, a way to encode the functoriality of the base change formula is to consider a category
of correspondences. Namely, let $(\on{PreStk})_{\on{corr:all,sch-qc}}$ be the
$\infty$-category whose objects are prestacks locally of finite type and morphisms from $\CY_1$ to $\CY_2$ 
are given by correspondences
$$ \xymatrix{\CZ \ar[d]_f \ar[r]^g & \CY_1 \\ \CY_2 & } $$
such that $f$ is schematic and quasi-compact, and $g$ arbitrary. Composition in this category is given by taking
Cartesian products of correspondences. A coherent base change formula for the functors $\Crys^\dagger$ and $\Crys_{\dr,*}$ is then a
functor
$$ \Crys_{(\on{PreStk})_{\on{corr:all,sch-qc}}}: (\on{PreStk})_{\on{corr:all,sch-qc}} \to \StinftyCat_{\on{cont}} $$
and an identification of the restriction to $(\on{PreStk})^{\on{op}}$ with $\Crys^\dagger_{\on{PreStk}}$, and the restriction to 
$\on{PreStk}_{\on{sch-qc}}$ with $(\Crys_{\dr,*})_{\on{PreStk}_{\on{sch-qc}}}$.

\ssec{Twistings}

\sssec{}

In addition to D-modules, it is often important to consider twisted D-modules. For instance, in
representation theory, the localization theorem of Beilinson and Bernstein identifies the category of
representations of a reductive Lie algebra $\mathfrak{g}$ with fixed central character $\chi$ with the
category of twisted D-modules on the flag variety $G/B$, with the twisting determined by $\chi$.

\medskip

In the case of smooth varieties, the theory of twistings and twisted D-modules was introduced by
Beilinson and Bernstein \cite{BB}. Important examples of twistings are given by complex tensor powers
of line bundles. For a smooth variety $X$, twistings form a Picard groupoid, which can be described as
follows. Let $\mathcal{T}$ be the complex of sheaves, in degrees 1 and 2, given by
$$ \mathcal{T} := \Omega^1 \rightarrow \Omega^{2,cl} $$
where $\Omega^1$ is the sheaf of 1-forms on $X$, $\Omega^{2,cl}$ is the sheaf of closed 2-forms and
the map is the de Rham differential. Then the space of objects of the Picard groupoid of twistings is
given by $H^2(X,\mathcal{T})$ and, for a given object, the space of isomorphisms is $H^1(X,\mathcal{T})$.

\sssec{}
The last two sections of this paper are concerned with developing the theory of twistings and twisted
crystals in the derived (and, in particular, non-smooth) context. We give several equivalent
reformulations of the notion of twisting and show that they are equivalent to that defined in
\cite{BB} in the case of smooth varieties.

\medskip

For a prestack (almost locally of finite type) $\CY$, we define a twisting to be a $\BG_m$-gerbe on
the de Rham prestack $\CY_\dr$ with a trivialization of its pullback to $\CY$. A line bundle $\CL$ on
$\CY$ gives a twisting which is the trivial gerbe on $\CY_\dr$, but the trivialization on $\CY$ is
given by $\CL$.

\medskip

Given a twisting $T$, the category of $T$-twisted crystals on $\CY$ is defined as the category of sheaves 
(ind-coherent or quasi-coherent) on $\CY_\dr$ twisted by the $\BG_m$-gerbe given by $T$.

\ssec{Contents}

We now describe the contents of the paper, section-by-section. 

\sssec{}
In Section 1, for a prestack $\CY$, we define the de Rham prestack $\CY_\dr$ and establish
some of its basic properties. Most importantly, we show that if $\CY$ is locally almost of finite type
then so is $\CY_\dr$.

\sssec{}
In Section 2, we define left crystals as quasi-coherent sheaves on the de Rham prestack and, 
in the locally almost of finite type case, right crystals as ind-coherent sheaves on the de Rham prestack.  
The latter is well-defined because, as established in Section 1, for a prestack locally almost of finite 
type its de Rham prestack is also locally almost of finite type.  In this case, we show that the 
categories of left and right crystals are equivalent.  Furthermore, we prove a version of Kashiwara's 
lemma in this setting.

\sssec{}
In Section 3, we show that the category of crystals satisfies h-descent (and in particular, fppf descent). We also introduce the
infinitesimal groupoid of a prestack $\CY$ as the \v{C}ech nerve of the natural map $\CY \rightarrow
\CY_\dr$.  Specifically, the infinitesimal groupoid of $\CY$ is given by
$$ (\CY \times \CY)^\wedge_{\CY} \rightrightarrows \CY$$
where $(\CY \times \CY)^\wedge_{\CY}$ is the formal completion of $\CY\times \CY$ along the diagonal.

\medskip

In much of Section 3, we specialize to the case that $\CY$ is an indscheme.  
Sheaves on the infinitesimal groupoid of $\CY$ are sheaves on $\CY$ which are equivariant with respect
to the equivalence relation of infinitesimal closeness. In the case of ind-coherent sheaves, this
category is equivalent to right crystals. However, quasi-coherent sheaves on the infinitesimal
groupoid are, in general, not equivalent to left crystals. We show that quasi-coherent sheaves on the
infinitesimal groupoid of $\CY$ are equivalent to left crystals if $\CY$ is an eventually coconnective DG scheme or
a classically formally smooth prestack.  Thus, in particular, this equivalence holds in the case of classical schemes.

\medskip

We also define induction functors from $\QCoh(\CY)$ and $\IndCoh(\CY)$ to crystals on $\CY$.  
In the case of ind-coherent sheaves the induction functor is left adjoint to the forgetful functor, and 
we have that the category of right crystals is equivalent to the category of modules over the corresponding monad.  
The analogous result is true for $\QCoh$ and left crystals in the case that $\CY$ is an eventually coconnective
DG scheme.  

\sssec{}
In Section 4, we show that the category of crystals has two natural t-structures: one compatible with 
the left realization to $\QCoh$ and another comaptible with the right realization to $\IndCoh$.  
In the case of a quasi-compact DG scheme, the two t-structures differ by a bounded amplitude.

\medskip

We also show that for an affine DG scheme, the category of crystals is
equivalent to the derived category of its heart with respect to the right t-structure. 

\sssec{}

In Section 5 we relate the monad acting on $\IndCoh$ (resp., $\QCoh$) on a DG scheme, responsible for
the category of right (resp., left) crystals, to the sheaf of differential operators. 

\medskip

As a result, we relate the category of crystals to the derived category of D-modules.

\sssec{}
In Section 6, we define the Picard groupoid of twistings on a prestack $\CY$ as that of $\BG_m$-gerbes
on the de Rham prestack $\CY_\dr$ which are trivialized on $\CY$. We then give several equivalent
reformulations of this definition. For instance, using a version of the exponential map, we show that
the Picard groupoid of twistings is equivalent to that of $\BG_a$-gerbes on the de Rham prestack
$\CY_\dr$ which are trivialized on $\CY$. In particular, this makes twistings naturally a $k$-linear
Picard groupoid.

\medskip

Furthermore, using the description of twistings in terms of $\BG_a$-gerbes, we identify the $\infty$-groupoid
of twistings as

$$ \tau^{\leq 2}\left( H_{\dr}(\CY) \underset{H(\CY)}{\times} \{*\} \right)[2] $$
where $H_{\dr}(Y)$ is the de Rham cohomology of $\CY$, and $H(\CY)$ is the coherent cohomology of $Y$.
In particular, for a smooth classical scheme, this shows that this notion of twisting agrees with that
defined in \cite{BB}.

\medskip

Finally, we show that the category of twistings on a DG (ind)scheme $\CX$ locally of finite type can be 
identified with that of central extensions of its infinitesimal groupoid.

\sssec{}
In Section 7, we define the category of twisted crystals and establish its basic properties.  
In particular, we show that most results about crystals carry over to the twisted setting.

\ssec{Conventions and notation}

Our conventions follow closely those of \cite{IndSch}. Let us recall the most essential ones. 

\sssec{The ground field}

Throughout the paper we will work over a fixed ground field $k$ of characteristic $0$. 

\sssec{$\infty$-categories}

By an $\infty$-category we shall always mean an $(\infty,1)$-category. By a slight abuse of language,
we will sometimes refer to ``categories" when we actually mean $\infty$-categories. Our usage of
$\infty$-categories is model independent, but we have in mind their realization as quasi-categories. The
basic reference for $\infty$-categories as quasi-categories is \cite{Lu0}.

\medskip

We denote by $\inftygroup$ the $\infty$-category of $\infty$-groupoids, which is the
same as the category ${\mathcal S}$ of spaces in the notation of \cite{Lu0}.

\medskip

For an $\infty$-category $\bC$, and $x,y\in \bC$, we shall denote by
$\on{Maps}_\bC(x,y)\in \inftygroup$ the corresponding mapping space. By
$\Hom_\bC(x,y)$ we denote the set $\pi_0(\on{Maps}_\bC(x,y))$, i.e.,
what is denoted $\Hom_{h\bC}(x,y)$ in \cite{Lu0}.

\medskip

A stable $\infty$-category $\bC$ is naturally enriched in spectra. In this case, for $x,y\in \bC$, we
shall denote by ${\mathcal Maps}_\bC(x,y)$ the spectrum of maps from $x$ to $y$. In particular, we
have that $\on{Maps}_\bC(x,y) = \Omega^{\infty} {\mathcal Maps}_\bC(x,y)$.

\medskip

When working in a fixed $\infty$-category $\bC$, for two objects $x,y\in \bC$,
we shall call a point of $\on{Maps}_\bC(x,y)$ an \emph{isomorphism} what is in
\cite{Lu0} is called an \emph{equivalence}. I.e., an isomorphism is a map that admits a homotopy 
inverse. We reserve the word ``equivalence" to mean a (homotopy) equivalence
between $\infty$-categories.

\sssec{DG categories}

Our conventions regarding DG categories follow \cite[Sect. 0.6.4]{IndCoh}.
By a DG category we shall understand a presentable DG category over $k$;
in particular, all our DG categories will be assumed cocomplete. 
Unless specified otherwise, we will only consider continuous
functors between DG categories (i.e., exact functors that commute
with direct sums, or equivalently, with all colimits). In other words,
we will be working in the category $\StinftyCat_{\on{cont}}$ in the
notation of \cite{DG}. \footnote{One can replace $\StinftyCat_{\on{cont}}$
by (the equivalent) $(\infty,1)$-category of stable presentable 
$\infty$-categories tensored over $\Vect$, with colimit-preserving functors.}

\medskip

We let $\Vect$ denote the DG category of complexes of $k$-vector spaces. 
The category $\StinftyCat_{\on{cont}}$ has a natural symmetric monoidal
structure, for which $\Vect$ is the unit. 

\medskip

For a DG category $\bC$ equipped with a t-structure, we denote by $\bC^{\leq n}$
(resp., $\bC^{\geq m}$, $\bC^{\leq n,\geq m}$) the corresponding full subcategory
of $\bC$ spanned by objects $x$, such that $H^i(x)=0$ for $i>n$ (resp., $i<m$,
$(i>n)\wedge (i<m)$). The inclusion $\bC^{\leq n}\hookrightarrow \bC$ admits
a right adjoint denoted by $\tau^{\leq n}$, and similarly, for the other
categories.

\medskip

There is a fully faithful functor from $\StinftyCat_{\on{cont}}$ to that of stable $\infty$-categories
and continuous exact functors. A stable $\infty$-category obtained in this way is enriched
over the category $\Vect$. Thus, we shall often think of
the spectrum ${\mathcal Maps}_\bC(x,y)$ as an object of $\Vect$; the former is obtained
from the latter by the Dold-Kan correspondence. 

\sssec{(Pre)stacks and DG schemes}

Our conventions regarding (pre)stacks and DG schemes follow \cite{Stacks}:

\medskip

Let $\affdgSch$ denote the $\infty$-category opposite to that of \emph{connective}
commutative DG algebras over $k$.

\medskip

The category $\inftydgprestack$ of prestacks is by definition that of all functors
$$(\affdgSch)^{\on{op}}\to \inftygroup.$$

\medskip

Let ${}^{< \infty}\!\affdgSch$ be the full subcategory of $\affdgSch$ given by eventually coconnective objects.

\medskip

Recall that an eventually coconnective affine DG scheme $S=\Spec(A)$ is \emph{almost of finite type} if
\begin{itemize}
\item $H^0(A)$ is finite type over $k$.
\item Each $H^i(A)$ is finitely generated as a module over $H^0(A)$.
\end{itemize}

Let ${}^{<\infty}\!\affdgSch_{\on{aft}}$ denote the full subcategory of ${}^{<\infty}\!\affdgSch$
consisting of schemes almost of finite type, and let $\inftydgprestack_{\on{laft}}$ be the category of all
functors
$$ {}^{<\infty}\!(\affdgSch_{\on{aft}})^{\on{op}} \to \inftygroup.$$
As explained in \cite[Sect. 1.3.11]{Stacks}, $\inftydgprestack_{\on{laft}}$ is naturally a subcategory
of $\inftydgprestack$ via a suitable Kan extension.

\medskip

In order to apply the formalism of ind-coherent sheaves developed in \cite{IndCoh}, we assume that the
prestacks we consider are locally almost of finite type for most of this paper. We will explicitly
indicate when this is not the case.

\sssec{Reduced rings}
Let $({^{red}\!\affSch})^{\on{op}} \subset (\affdgSch)^{\on{op}}$ denote the category of reduced discrete rings.  
The inclusion functor has a natural left adjoint
$$ ^{cl,red}(-): (\affdgSch)^{\on{op}} \rightarrow ({^{red}\!\affSch})^{\on{op}} $$
given by
$$ S \mapsto H^0(S)/\on{nilp}(H^0(S))$$
where $\on{nilp}(H^0(S))$ is the ideal of nilpotent elements in $H^0(S)$.

\ssec{Acknowledgments}

We are grateful to Jacob Lurie for numerous helpful discussions.  His ideas have so strongly influenced this paper that it is 
even difficult to pinpoint specific statements that
we learned directly from him.

\medskip

The research of D.G. is supported by NSF grant DMS-1063470.


\section{The de Rham prestack}
For a prestack $\CY$, crystals are defined as sheaves (quasi-coherent or ind-coherent) on the de Rham prestack $\CY_\dr$ of $\CY$.  
In this section, we define the functor $\CY \mapsto \CY_\dr$ and establish a number of its basic properties.

\medskip

Most importantly, we will show that if $\CY$ is locally almost of finite type, then so is $\CY_\dr$.  In this case, we will also show that 
$\CY_\dr$ is classical, i.e., it can be studied entirely within the realm of ``classical'' algebraic geometry without reference to derived rings.

\medskip

As the reader might find this section particularly abstract, it might be a good
strategy to skip it on  first pass, and return to it when necessary when assertions established here are applied to crystals.

\ssec{Definition and basic properties}

\sssec{}

Let $\CY$ be an object of $\on{PreStk}$. We define the de Rham prestack of $\CY$, $\CY_\dr\in \on{PreStk}$ as
\begin{equation} \label{e:dR on all}
\CY_\dr(S) := \CY({}^{cl,red}S)
\end{equation}
for $S\in \affdgSch$.

\sssec{}

More abstractly, we can rewrite
$$\CY_\dr:=\on{RKE}_{^{red}\!\affSch\hookrightarrow \affdgSch}({}^{cl,red}\CY),$$
where $^{cl,red}\CY:=\CY|_{^{red}\!\affSch}$ is the restriction of $\CY$ to reduced classical affine schemes, 
and $$\on{RKE}_{^{red}\!\affSch\hookrightarrow \affdgSch}$$
is the right Kan extension of a functor along the inclusion $^{red}\!\affSch\hookrightarrow \affdgSch$.

\sssec{}

The following (obvious) observation will be useful in the sequel.

\begin{lem} \label{l:dr and colimits}
The functor $\dr:\on{PreStk}\to \on{PreStk}$ commutes
with limits and colimits. 
\end{lem}

\begin{proof}
Follows from the fact that limits and colimits in
$$\on{PreStk}=\on{Funct}((\affdgSch)^{\on{op}},\inftygroup)$$
are computed object-wise.
\end{proof} 

As a consequence, we obtain: 

\begin{cor} \label{c:dr from Sch}
The functor $\dr:\on{PreStk}\to \on{PreStk}$ is the left Kan extension of the functor
$$\dr|_{\affdgSch}:\affdgSch\to \on{PreStk}$$
along $\affdgSch\hookrightarrow \on{PreStk}$.
\end{cor}

\begin{proof}
This is true for any colimit-preserving functor out of $\on{PreStk}$ to an aribitrary
$\infty$-category.
\end{proof}

\sssec{}

Furthermore, we have:

\begin{lem} \label{l:LKE from red}
The functor $\dr|_{\affdgSch}:\affdgSch\to \on{PreStk}$ is isomorphic to the left Kan
of the functor
$$\dr|_{^{red}\!\affSch}:{}^{red}\!\affSch\to \on{PreStk}$$
along $^{red}\!\affSch\hookrightarrow \affdgSch$.
\end{lem}

\begin{proof}
For any target category $\bD$ and any functor $\Phi:\affdgSch\to \bD$, the map
$$\on{LKE}_{^{red}\!\affSch\hookrightarrow \affdgSch}(\Phi|_{^{red}\!\affSch})\to \Phi$$
is an isomorphism if and only if the natural transformation
$$\Phi({}^{cl,red}S)\to \Phi(S),\quad S\in \affdgSch$$
is an isomorphism. The latter is the case for $\bD=\on{PreStk}$ and $\Phi$ the functor 
$S\mapsto S_\dr$ by definition. 
\end{proof}

\sssec{}  \label{sss:abstract LKE}

Let $\bC_1\subset \bC_2$ be a pair of categories from the following
list of full subcategories of $\on{PreStk}$:
$$^{red}\!\affSch,\, \affSch,\, \affdgSch, \,\dgSch_{\on{qs-qc}},\,\dgSch,\,\on{PreStk}$$
(here the subscript ``$\on{qs-qc}$" means ``quasi-separated and quasi-compact").

\medskip

From \lemref{l:LKE from red} and \corref{c:dr from Sch} we obtain:

\begin{cor}  \label{c:abstract LKE}
The functor $\bC_2\to \on{PreStk}$ given by $\dr|_{\bC_2}$
is isomorphic to the left Kan extension along $\bC_1\hookrightarrow \bC_2$
of the functor $\dr|_{\bC_1}:\bC_1\to \on{PreStk}$.
\end{cor}

\ssec{Relation between $\CY$ and $\CY_\dr$}

\sssec{}

The functor $\dr:\on{PreStk}\to \on{PreStk}$ comes equipped with a natural transformation
$$p_\dr:\on{Id}\to \dr,$$
i.e., for every $\CY\in \on{PreStk}$ we have a canonical map 
$$p_{\dr,\CY}:\CY\to \CY_\dr.$$

\sssec{}  \label{sss:inf groupoid}

Let $\CY^\bullet/\CY_\dr$ be the \v{C}ech nerve of $p_{\dr,\CY}$, regarded as a simplicial object of $\on{PreStk}$. 
It is augmented by $\CY_\dr$.

\medskip

Note that each $\CY^i/\CY_\dr$ is the formal completion of 
$\CY^i$ along the main diagonal. (We refer the reader to \cite[Sect. 6.1.1]{IndSch}, for our conventions 
regarding formal completions). 

\medskip

We have a canonical map
\begin{equation} \label{e:comp with inf groupoid}
|\CY^\bullet/\CY_\dr|\to \CY_\dr.
\end{equation}

\sssec{Classically formally smooth prestacks}  \label{sss:cl formally smooth}

We shall say that a prestack $\CY$ is classically formally smooth, if for $S\in \affdgSch$, the map
$$\on{Maps}(S,\CY)\to \on{Maps}({}^{cl,red}S,\CY)$$
induces a surjection on $\pi_0$.

\medskip

The following results from the definitions:

\begin{lem}  \label{l:cl formally smooth}
If $\CY$ is classically formally smooth, the map
$$|\CY^\bullet/\CY_\dr|\to \CY_\dr$$
is an isomorphism in $\on{PreStk}$. 
\end{lem}

\ssec{The locally almost of finite type case}

\sssec{}  \label{sss:convergence}

Recall that $\on{PreStk}$ contains a full subcategory $\on{PreStk}_{\on{laft}}$ of prestacks locally almost of finite type, 
see \cite[Sect. 1.3.9]{Stacks}. By definition, an object $\CY\in \on{PreStk}$ belongs to $\on{PreStk}_{\on{laft}}$ if:

\begin{itemize}

\item $\CY$ is \emph{convergent}, i.e., for $S\in \affdgSch$, the natural map
$$\Maps(S,\CY) \to \underset{n\geq 0}{lim}\, \Maps({}^{\leq n}\!S,\CY)$$
is an isomorphism, where $^{\leq n}\!S$ denotes th $n$-coconnective truncation of $S$. 

\medskip

\item For every $n$, the restriction $^{\leq n}\CY:=\CY|_{^{\leq n}\!\affdgSch}$ belongs to
$^{\leq n}\!\on{PreStk}_{\on{lft}}$, i.e., the functor
$$S\mapsto \Maps(S,\CY),\quad ({}^{\leq n}\!\affdgSch)^{\on{op}}\to \inftygroup$$
commutes with filtered colimits (equivaently, is a left Kan extension form the full
subcategory $^{\leq n}\!\affdgSch_{\on{ft}}\hookrightarrow {}^{\leq n}\!\affdgSch$). 

\end{itemize}

\sssec{}

The following observation will play an important role in this paper.

\begin{prop}  \label{p:dr laft}
Assume that $\CY\in \on{PreStk}_{\on{laft}}$. Then:

\smallskip

\noindent{\em(a)} $\CY_\dr\in \on{PreStk}_{\on{laft}}$.

\smallskip

\noindent{\em(b)} $\CY_\dr$ is classical, i.e., belongs to the full
subcategory $^{cl}\!\on{PreStk}\subset {}\on{PreStk}$. 

\end{prop}

\sssec{Proof of point (a)} \hfill

\medskip

We need to verify two properties:

\medskip

\noindent(i) $\CY_{\dr}$ is convergent; 

\medskip

\noindent(ii) Each truncation $^{\leq n}(\CY_\dr)$ is locally of finite type.

\medskip

Property (i) follows tautologically; it is true for any $\CY\in \on{PreStk}$.
To establish property (ii), we need to show that 
the functor $\CY_\dr$
takes filtered limits in $^{\leq n}\!\affdgSch$ to colimits in $\inftygroup$. Since $\CY$ itself has this
property, it suffices to show that the functor 
$$S\mapsto {}^{cl,red}S:\affdgSch\to \affdgSch$$
preserves filtered limits, which is evident.

\qed

\sssec{Proof of point (b)} \hfill

\medskip

By \corref{c:abstract LKE}, we need to prove that the colimit
$$\underset{S\in (\affSch)_{/\CY}}{colim}\, S_\dr\in \on{PreStk}$$
is classical. By part (a), the functor
$$(\affSch_{\on{ft}})_{/\CY}\to (\affSch)_{/\CY}$$
is cofinal; hence,
$$\underset{S\in (\affSch_{\on{ft}})_{/\CY}}{colim}\, S_\dr\to \underset{S\in (\affSch)_{/\CY}}{colim}\, S_\dr$$
is an isomorphism. 

\medskip

Therefore, since the full subcategory $^{cl}\!\on{PreStk}\subset {}\on{PreStk}$ is closed under colimits, we can 
assume without loss of generality that $\CY$ is a classical affine scheme of finite type.

\medskip

More generally, we will show that
for $X\in \affdgSch_{\on{aft}}$, the prestack $X_\dr$ is classical.
Let $i:X\hookrightarrow Z$ be a closed embedding, where $Z$ is a \emph{smooth} classical affine scheme
of finite type. Let $Y$ denote the formal completion $Z^\wedge_X$ of $Z$ along $X$ (see \cite[Sects. 6.1.1 or 6.5]{IndSch}). 
The map $X\to Y$ induces an isomorphism 
$X_\dr\to Y_\dr$. Hence, it suffices to show that $Y_\dr$ is classical. 

\medskip

Consider $Y^\bullet/Y_\dr$ (see \secref{sss:inf groupoid} above). Note that $Y$ is formally smooth, 
since $Z$ is (see \cite[Sect. 8.1]{IndSch}). In particular, $Y$ is classically
formally smooth. Since 
the subcategory $^{cl}\!\on{PreStk}\subset {}\on{PreStk}$ is closed
under colimits and by \lemref{l:cl formally smooth}, it suffices to show that each term $Y^i/Y_\dr$ is classical as a prestack. 

\medskip

Note that $Y^i/Y_\dr$ is isomorphic to the formal completion of $Z^i$ along the diagonally embedded
copy of $X$. Hence, $Y^i/Y_\dr$ is classical by \cite[Proposition 6.8.2]{IndSch}.

\qed

\sssec{}

From \propref{p:dr laft} we obtain:

\begin{cor} \label{c:dr laft abstract}
Let $\bC_1\subset \bC_2$ be any of the following full subcategories of $\affdgSch$:
$$\affSch_{\on{ft}},\,{}^{<\infty}\!\affdgSch_{\on{aft}},\,\affdgSch_{\on{aft}},\,\affSch,\,\affdgSch.$$
Then for $\CY\in \on{PreStk}_{\on{laft}}$, the functor
$$(\bC_1)_{/\CY_\dr}\to (\bC_2)_{/\CY_\dr}$$
is cofinal.
\end{cor}

\begin{proof}
It suffices to prove the assertion for the inclusions
$$\affSch_{\on{ft}}\hookrightarrow \affSch\hookrightarrow \affdgSch.$$
For right arrow, the assertion follows from point (b) of \propref{p:dr laft},
and for the left arrow from point (a).
\end{proof}

\sssec{}

Now, consider the following full subcategories 
\begin{equation} \label{e:ft categories}
^{red}\!\affSch_{\on{ft}},\, \affSch_{\on{ft}},\, \affdgSch_{\on{aft}}, \,\dgSch_{\on{aft}},\, \dgSch_{\on{laft}},\, \on{PreStk}_{\on{laft}}.
\end{equation}
of the categories appearing in \secref{sss:abstract LKE}.

\begin{cor}  \label{c:abstract LKE laft}
The restriction of the functor $\dr$ to $\on{PreStk}_{\on{laft}}$ is isomorphic to the left Kan extension
of this functor to $\bC$, where $\bC$ is one of the subcategories in \eqref{e:ft categories}.
\end{cor}

\begin{proof}
It suffices to prove the corollary for $\bC={}^{red}\!\affSch_{\on{ft}}$. By \corref{c:abstract LKE}, it is enough to
show that for $\CY\in \on{PreStk}_{\on{laft}}$, the functor
$$({}^{red}\!\affSch_{\on{ft}})_{/\CY}\to (\affSch)_{/\CY}$$
is cofinal. 

\medskip

By \propref{p:dr laft}(a), the functor
$$(\affSch_{\on{ft}})_{/\CY}\to (\affSch)_{/\CY}$$
is cofinal. Now, the assertion follows from the fact that the inclusion $^{red}\!\affSch_{\on{ft}}\hookrightarrow \affSch_{\on{ft}}$
admits a right adjoint. 
\end{proof}

\section{Definition of crystals}

In this section we will define left crystals (for arbitrary objects of $\on{PreStk}$), and right crystals for
objects of $\on{PreStk}_{\on{laft}}$. We will show that in the latter case, the two theories are equivalent.

\ssec{Left crystals}

\sssec{}

For $\CY\in \on{PreStk}$ we define
$$\Crys^l(\CY):=\QCoh(\CY_\dr).$$

I.e.,
$$\Crys^l(\CY)=\underset{S\in (\affdgSch_{/\CY_\dr})^{\on{op}}}{lim}\, \QCoh(S).$$

Informally, an object $\CM\in \Crys^l(\CY)$ is an assignment of a quasi-coherent sheaf $\CF_S\in \QCoh(S)$ for every affine DG 
scheme $S\in \affdgSch$ with a map ${}^{red,cl}S\to\CY$, as well as an isomorphism $$f^*(\CF_S)\simeq \CF_{S'}\in \QCoh(S')$$ 
for every morphism $f:S'\to S$ of affine DG schemes.

\sssec{}

More functorially, let $\Crys^l_{\on{PreStk}}$ denote the functor $(\on{PreStk})^{\on{op}}\to \StinftyCat_{\on{cont}}$
defined as
$$\Crys^l_{\on{PreStk}}:=\QCoh^*_{\on{PreStk}}\circ \dr,$$
where
$$\QCoh^*_{\on{PreStk}}:(\on{PreStk})^{\on{op}}\to \StinftyCat_{\on{cont}}$$
is the functor which assigns to a prestack the corresponding category of quasi-coherent sheaves \cite[Sect. 1.1.5]{QCoh}.

\medskip

For a map $f:\CY_1\to \CY_2$ in $\on{PreStk}$, let $f^{\dagger,l}$ denote the corresponding pullback functor
$$\Crys^l(\CY_2)\to \Crys^l(\CY_1).$$

\medskip

By construction, if $f$ induces an isomorphism of the underlying reduced classical prestacks
$^{cl,red}\CY_1\to {}^{cl,red}\CY_2$, then it induces an isomorphism of de Rham prestacks $\CY_{1,\dr} \to \CY_{2,\dr}$ 
and in particular $f^{\dagger,l}$ is an equivalence.

\sssec{}

Recall that the functor $\QCoh^*_{\on{PreStk}}:(\on{PreStk})^{\on{op}}\to \StinftyCat_{\on{cont}}$ is by definition the right
Kan extension of the functor 
$$\QCoh^*_{\affdgSch}:(\affdgSch)^{\on{op}}\to \StinftyCat_{\on{cont}}$$
along $(\affdgSch)^{\on{op}}\hookrightarrow (\on{PreStk})^{\on{op}}$.

\medskip

In particular, it takes colimits in $\on{PreStk}$ to limits in $\StinftyCat_{\on{cont}}$. Therefore, by
\corref{c:abstract LKE}, for $\CY\in \on{PreStk}$ we obtain:

\begin{cor}  \label{c:RKE for left}
Let $\bC$ be any of the categories from the list of \secref{sss:abstract LKE}. Then for $\CY\in \on{PreStk}$,
the functor
$$\Crys^l(\CY)\to \underset{X\in (\bC_{/\CY})^{\on{op}}}{lim}\, \Crys^l(X)$$
is an equivalence.
\end{cor}

Informally, this corollary says that
the data of an object $\CM\in \Crys^l(\CY)$ is
equivalent to that of $\CM_S\in \Crys^l(S)$ for every $S\in \bC_{/\CY}$, and for every
$f:S'\to S$, an isomorphism
$$f^{\dagger,l}(\CM_S)\simeq \CM_{S'}\in \Crys^l(S').$$

\sssec{}

Recall the natural transformation $p_\dr:\on{Id}\to \dr$. It induces a natural transformation
$$\oblv^l:\Crys^l_{\on{PreStk}}\to \QCoh^*_{\on{PreStk}}.$$

I.e., for every $\CY\in \on{PreStk}$, we have a functor 
\begin{equation} \label{e:oblv left}
\oblv^l_\CY:\Crys^l(\CY)\to \QCoh(\CY),
\end{equation}
and for every morphism $f:\CY_1\to \CY_2$, a commutative diagram:
\begin{equation} \label{e:pullback left}
\CD
\Crys^l(\CY_1)   @>{\oblv^l_{\CY_1}}>>  \QCoh(\CY_1)   \\
@A{f^{\dagger,l}}AA    @AA{f^*}A  \\
\Crys^l(\CY_2)    @>{\oblv^l_{\CY_2}}>> \QCoh(\CY_2).
\endCD
\end{equation}

\sssec{}

Recall the simplicial object $\CY^\bullet/\CY_\dr$ of \secref{sss:inf groupoid}. 

\medskip

From \lemref{l:cl formally smooth} we obtain:

\begin{lem} \label{l:left on cl formally smooth}
If $\CY$ is classically formally smooth, then the functor
$$\Crys^l(\CY)\to \on{Tot}(\QCoh(\CY^\bullet/\CY_\dr))$$
is an equivalence.
\end{lem}

\begin{rem}
Our definition of left crystals on $\CY$ is what in Grothendieck's terminology is \emph{quasi-coherent sheaves 
on the infinitesimal site of $\CY$}. The category $\on{Tot}(\QCoh(\CY^\bullet/\CY_\dr))$ is what in Grothendieck's terminology is 
\emph{quasi-coherent sheaves on the stratifying site of $\CY$}. Thus, \lemref{l:left on cl formally smooth} says that
the two are equivalent for classically formally smooth prestacks. We shall see in \secref{ss:ind left} that the same is
also true when $\CY$ is an eventually coconnective DG scheme locally almost of finite type. However,
the equivalence fails for DG schemes that are not eventually coconnective (even ones that are locally almost of finite type).
\end{rem}

\ssec{Left crystals on prestacks locally almost of finite type}

For the rest of this section, unless specified otherwise,
we will restrict ourselves to the subcategory
$\on{PreStk}_{\on{laft}}\subset\on{PreStk}$.

\medskip

So, unless explicitly stated otherwise, by a prestack/DG scheme/affine DG scheme, we shall
mean one which is locally almost of finite type.

\medskip

Let $\Crys^l_{\on{PreStk}_{\on{laft}}}$ denote the restriction of $\Crys^l_{\on{PreStk}}$ to
$\on{PreStk}_{\on{laft}}\subset\on{PreStk}$.

\sssec{}  \label{sss:only laft left}

The next corollary says that we ``do not need to know" about schemes of infinite type or derived algebraic geometry
 in order to
define $\Crys^l(\CY)$ for $\CY\in \on{PreStk}_{\on{laft}}$.  In other words, to define crystals on a prestack locally almost of finite 
type, we can stay within the world of classical affine schemes of finite type.

\medskip

Indeed, from \corref{c:dr laft abstract} we obtain:

\begin{cor}  \label{c:dr laft}
Let $\bC$ be one of the full subcategories 
$$\affSch_{\on{ft}},\,{}^{<\infty}\!\affdgSch_{\on{aft}},\,\affdgSch_{\on{aft}},\,\affSch$$
of $\affdgSch$. Then for $\CY\in \on{PreStk}_{\on{laft}}$ the natural functor
$$\Crys^l(\CY)\to \underset{S\in (\bC_{/\CY_\dr})^{\on{op}}}{lim}\, \QCoh(S)$$
is an equivalence.
\end{cor}

\sssec{}

Recall that according to \corref{c:RKE for left}, the category $\Crys^l(\CY)$ can be recovered from the functor
\[ \Crys^l: \bC_{/\CY} \rightarrow \StinftyCat_{\on{cont}} \]
where $\bC$ is any one of the categories
$$^{red}\!\affSch,\, \affSch,\, \affdgSch, \,\dgSch_{\on{qs-qc}},\, \dgSch\subset \on{PreStk}.$$

We now claim that the above categories can be also replaced by their full subcategories
in the list \eqref{e:ft categories}:
$$^{red}\!\affSch_{\on{ft}},\, \affSch_{\on{ft}},\, \affdgSch_{\on{aft}}, \,\dgSch_{\on{aft}},\, \dgSch_{\on{laft}},\, \on{PreStk}_{\on{laft}}.$$

\begin{cor}   \label{c:RKE for left aft}
For $\CY\in \on{PreStk}_{\on{laft}}$ and $\bC$ being one of the categories in \eqref{e:ft categories},
the functor
$$\Crys^l(\CY)\to \underset{X\in (\bC_{/\CY})^{\on{op}}}{lim}\, \Crys^l(X)$$
is an equivalence.
\end{cor}

\begin{proof}

Follows from \corref{c:abstract LKE laft}.

\end{proof}

Informally, the above corollary says that an object $\CM\in \Crys^l(\CY)$ can be recovered from an 
assignment of $\CM_S\in \Crys^l(S)$ for every $S\in \bC_{/\CY}$, and for every $f:S'\to S$
of an isomorphism 
$$f^{\dagger,l}(\CM_S)\simeq \CM_{S'}\in \Crys^l(S').$$

\sssec{}

Consider again the functor 
$$\oblv^l_\CY:\Crys^l(\CY)\to \QCoh(\CY)$$
of \eqref{e:oblv left}. We have:

\begin{lem}  \label{l:cons}
For $\CY \in \on{PreStk}_{\on{laft}}$, the functor $\oblv^l_\CY$ is conservative.
\end{lem}

The proof is deferred until \secref{sss:proof of l cons}.

\ssec{Right crystals}

\sssec{}

Recall that $\on{PreStk}_{\on{laft}}$ can be alternatively viewed as the category of all functors
$$({}^{<\infty}\!\affdgSch_{\on{aft}})^{\on{op}}\to \inftygroup,$$ 
see \cite[Sect. 1.3.11]{Stacks}.

\medskip

Furthermore, we have the functor 
$$\IndCoh^!_{\on{PreStk}_{\on{laft}}}:(\on{PreStk}_{\on{laft}})^{\on{op}}\to \StinftyCat_{\on{cont}}$$
of \cite[Sect. 10.1.2]{IndCoh}, which is defined as the right Kan extension of the corresponding functor
$$\IndCoh^!_{^{<\infty}\!\affdgSch_{\on{aft}}}:({}^{<\infty}\!\affdgSch_{\on{aft}})^{\on{op}}\to \StinftyCat_{\on{cont}}$$
along
$$({}^{<\infty}\!\affdgSch_{\on{aft}})^{\on{op}}\to (\on{PreStk}_{\on{laft}})^{\on{op}}.$$

In particular, the functor $\IndCoh^!_{\on{PreStk}_{\on{laft}}}$ takes colimits in $\on{PreStk}_{\on{laft}}$ to
limits in $\StinftyCat_{\on{cont}}$.

\sssec{}

We define the functor 
$$\Crys^r_{\on{PreStk}_{\on{laft}}}:(\on{PreStk}_{\on{laft}})^{\on{op}}\to \StinftyCat_{\on{cont}}$$
as the composite
$$\Crys^r_{\on{PreStk}_{\on{laft}}}:=\IndCoh^!_{\on{PreStk}_{\on{laft}}}\circ \dr.$$

\medskip

In the above formula, \propref{p:dr laft}(a) is used to make sure that $\dr$ is defined as a functor
$\on{PreStk}_{\on{laft}}\to \on{PreStk}_{\on{laft}}$.

\begin{rem}
In defining $\Crys^r_{\on{PreStk}_{\on{laft}}}$ we ``do not need to know'' about schemes of infinite type:
we can define the endo-functor $\dr:\on{PreStk}_{\on{laft}}\to \on{PreStk}_{\on{laft}}$ directly by the formula
$$\on{Maps}(S,\CY_\dr)=\on{Maps}({}^{red,cl}S,\CY)$$
for $S\in {}^{<\infty}\!\affdgSch_{\on{aft}}$.
\end{rem}

\sssec{}

For a map $f:\CY_1\to \CY_2$ in $\on{PreStk}_{\on{laft}}$, we shall denote by $f^{\dagger,r}$ the corresponding
functor $\Crys^r(\CY_2)\to \Crys^r(\CY_1)$.

\medskip

If $f$ induces an equivalence $^{cl,red}\CY_1\to {}^{cl,red}\CY_2$, then the map $\CY_{1,\dr}\rightarrow \CY_{2,\dr}$ is 
an equivalence, and in particular, so is $f^{\dagger,r}$.

\sssec{}

By definition, for $\CY\in \on{PreStk}_{\on{laft}}$, we have:
$$\Crys^r(\CY)=\underset{S\in (({}^{<\infty}\!\affdgSch_{\on{aft}})_{/\CY_\dr})^{\on{op}}}{lim}\, \IndCoh(S).$$

Informally, an object $\CM\in \Crys^r(\CY)$ is an assignment for every $S\in {}^{<\infty}\!\affdgSch_{\on{aft}}$
and a map $^{red,cl}S\to \CY$ of an object $\CF_S\in \IndCoh(S)$, and for every $f:S'\to S$ of an isomorphism
$$f^!(\CF_S)\simeq \CF_{S'}\in \IndCoh(S').$$

\sssec{}

As in \secref{sss:only laft left}, we ``do not need to know" about DG schemes
in order to recover $\Crys^r(\CY)$:

\begin{cor}  \label{c:only aft right}
For $\CY\in \on{PreStk}_{\on{laft}}$, the functor
$$\Crys^r(\CY)\to \underset{S\in ((\affSch_{\on{ft}})_{/\CY_\dr})^{\on{op}}}{lim}\, \IndCoh(S)$$
is an equivalence.
\end{cor}

\begin{proof}
Follows readily from \corref{c:dr laft abstract}.
\end{proof}

Informally, the above corollary says that an $\CM\in \Crys^r(\CY)$ can be recovered from 
an assignment for every $S\in \affSch_{\on{ft}}$
and a map $^{red,cl}S\to \CY$ of an object $\CF_S\in \IndCoh(S)$, and for every $f:S'\to S$ of an isomorphism
$$f^!(\CF_S)\simeq \CF_{S'}\in \IndCoh(S').$$

\sssec{}

Furthermore, the analogue of \corref{c:RKE for left aft} holds for right crystals as well:

\begin{cor} \label{c:RKE for right aft}
Let $\bC$ be any of the categories from \eqref{e:ft categories}. Then the functor
$$\Crys^r(\CY)\to \underset{X\in (\bC_{/\CY})^{\on{op}}}{lim}\, \Crys^r(X)$$
is an equivalence.
\end{cor}

\begin{proof}
Follows from \corref{c:abstract LKE laft}.
\end{proof}

Informally, this corollary says that
we can recover an object $\CM\in \Crys^r(\CY)$ from an 
assignment of $\CM_S\in \Crys^r(S)$ for every $S\in \bC_{/\CY}$, and for every $f:S'\to S$
of an isomorphism 
$$f^{\dagger,r}(\CM_S)\simeq \CM_{S'}\in \Crys^r(S').$$

\sssec{}

The natural transformation $p_\dr:\on{Id}\to \dr$ induces a natural transformation
$$\oblv^r:\Crys^r_{\on{PreStk}_{\on{laft}}}\to \IndCoh_{\on{PreStk}_{\on{laft}}}.$$

I.e., for every $\CY\in \on{PreStk}_{\on{laft}}$, we have a functor $$\oblv^r_\CY:\Crys^r(\CY)\to \IndCoh(\CY),$$
and for every morphism $f:\CY_1\to \CY_2$, a commutative diagram:
\begin{equation} \label{e:pullback right}
\CD
\Crys^r(\CY_1)   @>{\oblv^r_{\CY_1}}>>  \IndCoh(\CY_1)   \\
@A{f^{\dagger,r}}AA    @AA{f^!}A  \\
\Crys^r(\CY_2)    @>{\oblv^r_{\CY_2}}>> \IndCoh(\CY_2).
\endCD
\end{equation}

We have:

\begin{lem} \label{l:right on cl formally smooth}
If $\CY$ is classically formally smooth, then the functor
$$\Crys^r(\CY)\to \on{Tot}(\IndCoh(\CY^\bullet/\CY_\dr))$$
is an equivalence.
\end{lem}

\begin{proof}
Same as that of \lemref{l:left on cl formally smooth}, i.e. follows from \lemref{l:cl formally smooth}.
\end{proof}

\begin{lem}  \label{l:cons for r}
For any $\CY$, the functor $\oblv^r_\CY$ is conservative.
\end{lem}

\begin{proof}

By Corollary \ref{c:RKE for right aft} and the commutativity of \eqref{e:pullback right}, we can assume without loss of generality 
that $\CY=X$ is an affine DG scheme locally almost of finite type.  Let $i: X \rightarrow Z$ be a closed embedding of $X$ into a 
smooth classical finite type scheme $Z$, and let $Y$ be the formal completion of $Z$ along $X$.  Let $'i$ denote the resulting map
$X\to Y$. 

\medskip

Consider the commutative diagram
$$
\CD
\Crys^r(Y)  @>{\oblv^r_Y}>>  \IndCoh(Y)  \\
@V{'i^{\dagger,r}}VV   @VV{'i^!}V    \\
\Crys^r(X)  @>{\oblv^r_X}>> \IndCoh(X).
\endCD
$$
In this diagram the left vertical arrow is an equivalence since $'i_\dr:X_\dr\to Y_\dr$ is an isomorphism.
The top horizontal arrow is conservative by \lemref{l:right on cl formally smooth}, 
since $Y$ is formally smooth (and, in particular, classically formally smooth).  

\medskip

Hence, it remains to show that the functor $'i^!$ is conservative. This follows, e.g., by
combining \cite[Proposition 7.4.5]{IndSch} and \cite[Proposition 4.1.7(a)]{IndCoh}. 

\end{proof}

\ssec{Comparison of left and right crystals}

We remind the reader that we assume that all prestacks and DG schemes are locally almost of finite type.

\sssec{}
Recall (see \cite[Sect. 5.7.5]{IndCoh}) that
for any $S\in \dgSch_{\on{aft}}$ there is a canonically defined functor
$$\Upsilon_S:\QCoh(S)\to \IndCoh(S),$$
given by tensoring with the duaizing sheaf $\omega_S\in \IndCoh(S)$, such that for $f:S_1\to S_2$,
the diagram
$$
\CD
\QCoh(S_1)   @>{\Upsilon_{S_1}}>>  \IndCoh(S_1)   \\
@A{f^*}AA    @AA{f^!}A  \\
\QCoh(S_2)    @>{\Upsilon_{S_2}}>> \IndCoh(S_2)
\endCD
$$
canonically commutes. In fact, the above data upgrades to 
a natural transformation of functors 
$$\Upsilon_{\dgSch_{\on{aft}}}:\QCoh^*_{\dgSch_{\on{aft}}}\to \IndCoh^!_{\dgSch_{\on{aft}}},$$
and hence gives rise to a natural transformation
$$\Upsilon_{\on{PreStk}_{\on{laft}}}:\QCoh^*_{\on{PreStk}_{\on{laft}}}\to \IndCoh^!_{\on{PreStk}_{\on{laft}}},$$
\cite[Sect. 10.3.3]{IndCoh}.

\medskip

For an individual object $\CY\in \on{PreStk}_{\on{laft}}$, we obtain a functor
$$\Upsilon_\CY:\QCoh(\CY)\to \IndCoh(\CY).$$

\sssec{}

Applying $\Upsilon$ to $\CY_\dr$ for $\CY\in \on{PreStk}_{\on{laft}}$, we obtain 
a canonically defined functor
\begin{equation}  \label{e:from left to right}
\Upsilon_{\CY_\dr}:\Crys^l(\CY)\to \Crys^r(\CY),
\end{equation}
making the diagram
\begin{equation} \label{e:oblv left and right}
\CD
\Crys^l(\CY)   @>{\Upsilon_{\CY_\dr}}>>  \Crys^r(\CY) \\
@V{\oblv^l_\CY}VV   @VV{\oblv^r_\CY}V  \\
\QCoh(\CY)  @>{\Upsilon_\CY}>> \IndCoh(\CY)
\endCD
\end{equation}
commute.

\medskip

In fact we obtain a natural transformation
$$\Upsilon_{\on{PreStk}_{\on{laft}}}\circ \dr:\Crys^l_{\on{PreStk}_{\on{laft}}}\to \Crys^r_{\on{PreStk}_{\on{laft}}}.$$
In particular, for $f:\CY_1\to \CY_2$ the diagram
$$
\CD
\Crys^l(\CY_1)   @>{\Upsilon_{{\CY_1}_\dr}}>>  \Crys^r(\CY_1)   \\
@A{f^{\dagger,l}}AA    @AA{f^{\dagger,r}}A  \\
\Crys^l(\CY_2)    @>{\Upsilon_{{\CY_2}_\dr}}>> \Crys^r(\CY_2)
\endCD
$$
commutes.

\sssec{}

We claim:

\begin{prop} \label{p:left to right}
For $\CY\in \on{PreStk}_{\on{laft}}$, the functor \eqref{e:from left to right} is an equivalence.
\end{prop}

\begin{proof}

By Corollaries \ref{c:RKE for left aft} and \ref{c:RKE for right aft}, the statement reduces to one saying that
$$\Upsilon_{X_\dr}:\Crys^l(X)\to \Crys^r(X)$$ is an equivalence for an affine DG scheme $X$ almost of finite type. 

\medskip

Let $i:X\hookrightarrow Z$ be a closed embedding, where $Z$ is a smooth classical scheme, and let
$Y$ be the formal completion of $Z$ along $X$. Since $X_\dr\to Y_\dr$ is an isomorphism, the
functors
$$f^{\dagger,l}:\Crys^l(Y)\to \Crys^l(X) \text{ and } f^{\dagger,r}:\Crys^r(Y)\to \Crys^r(X)$$
are both equivalences. Hence, it is enough to prove the assertion for $Y$. 

\medskip

Let $Y^\bullet/Y_\dr$ be the \v{C}ech nerve of $\on{PreStk}_{\on{laft}}$ corresponding to
the map $$p_{\dr,Y}:Y\to Y_\dr.$$

\medskip

Consider the commutative diagram
$$
\CD
\Crys^l(Y) @>{\Upsilon_{Y_\dr}}>>   \Crys^r(Y)  \\
@VVV   @VVV  \\
\on{Tot}(\QCoh(Y^\bullet/Y_\dr))  @>{\on{Tot}(\Upsilon_{Y^\bullet/Y_\dr})}>>\on{Tot}(\IndCoh(Y^\bullet/Y_\dr)).
\endCD
$$

By Lemmas \ref{l:left on cl formally smooth} and \ref{l:right on cl formally smooth},
the vertical arrows in the diagram are equivalences. Therefore, it
suffices to show that for every $i$,
$$\Upsilon_{Y^i/Y_\dr}:\QCoh(Y^i/Y_\dr)\to \IndCoh(Y^i/Y_\dr)$$
is an equivalence. 

\medskip

Recall (also from the proof of \propref{p:dr laft}) that $Y^i/Y_\dr$ is the completion of the smooth
classical scheme $Z^i$ along the diagonal copy of $X$. Let us denote by $U_i\subset Z^i$ the complementary 
open substack. 

\medskip

From \cite[Propositions 7.1.3 and 7.4.5 and Diagram (7.16)]{IndSch}, we obtain that we have a map
of ``short exact sequences" of DG categories
$$
\CD
0 @>>> \QCoh(Y^i/Y_\dr)  @>>>  \QCoh(Z^i)  @>>>  \QCoh(U_i)  @>>>  0 \\
& & @V{\Upsilon_{Y^i/Y_\dr}}VV  @VV{\Upsilon_{Z^i}}V  @VV{\Upsilon_{U_i}}V  & & \\
0 @>>> \IndCoh(Y^i/Y_\dr)  @>>>  \IndCoh(Z^i)  @>>>  \IndCoh(U_i)  @>>>  0.
\endCD
$$

Now, the functors 
$$\Upsilon_{Z^i}:\QCoh(Z^i)\to \IndCoh(Z^i) \text{ and } \Upsilon_{U_i}:\QCoh(U_i) \to \IndCoh(U_i)$$
are both equivalences, since $Z^i$ and $U_i$ are smooth: 

\medskip

Indeed, by \cite[Proposition 9.3.3]{IndCoh}, for any $S\in \dgSch_{\on{aft}}$, the functor $\Upsilon_S$
is the dual of $\Psi_S:\IndCoh(S)\to \QCoh(S)$, and the latter is an equivalence if $S$ is smooth by
\cite[Lemma 1.1.6]{IndCoh}.

\end{proof}

\sssec{}

\propref{p:left to right} allows us to identify left and right crystals for objects
$\CY\in \on{PreStk}_{\on{laft}}$.

\medskip

In other words, we can consider the category $\Crys(\CY)$ equipped with two realizations: ``left" and ``right", which incarnate themselves
as forgetful functors 
$\oblv^l_\CY$ and $\oblv^r_\CY$ from $\Crys(\CY)$ to $\QCoh(\CY)$
and $\IndCoh(\CY)$, respectively. 

\medskip

The two forgetful functors are related by the commutative diagram
\begin{gather} \label{e:2categ diag}
\xy
(-20,0)*+{\QCoh(\CY)}="A";
(20,0)*+{\IndCoh(\CY)}="B";
(0,20)*+{\Crys(\CY).}="C";
{\ar@{->}_{\oblv^l_\CY} "C";"A"};
{\ar@{->}^{\oblv^r_\CY} "C";"B"};
{\ar@{->}_{\Upsilon_\CY} "A";"B"};
\endxy
\end{gather}

For a morphism $f:\CY_1\to \CY_2$ we have a naturally defined functor
$$f^{\dagger}:\Crys(\CY_2)\to \Crys(\CY_1),$$
which makes the following diagrams commute
$$
\CD
\Crys(\CY_1)  @<{f^\dagger}<< \Crys(\CY_2)  \\
@V{\oblv^l_{\CY_1}}VV    @VV{\oblv^l_{\CY_2}}V   \\
\QCoh(\CY_1)  @<{f^*}<<  \QCoh(\CY_2)
\endCD
$$
and
$$
\CD
\Crys(\CY_1)  @<{f^\dagger}<< \Crys(\CY_2)  \\
@V{\oblv^r_{\CY_1}}VV    @VV{\oblv^r_{\CY_2}}V   \\
\IndCoh(\CY_1)  @<{f^!}<<  \IndCoh(\CY_2).
\endCD
$$

\sssec{}

In the sequel, we shall use symbols $\Crys(\CY)$, $\Crys^r(\CY)$ and $\Crys^l(\CY)$ interchangeably
with the former emphasizing that the statement is independent of realization (left or right)
we choose, and the latter two, when a choice of the realization is important.

\sssec{Proof of \lemref{l:cons}}   \label{sss:proof of l cons}

Follows by combining \lemref{l:cons for r} and \propref{p:left to right}.

\qed

\ssec{Kashiwara's lemma}

A feature of the assignment $\CY\mapsto \Crys(\CY)$ is that Kashiwara's lemma becomes nearly tautological.

\medskip

We will formulate and prove it for the incarnation of crystals as right crystals. By \propref{p:left to right},
this implies the corresponding assertion for left crystals. However, one could easily write the same proof in the language of left crystals instead.

\sssec{}

Recall that a map $i:\CX\to \CZ$ in $\on{PreStk}$
is called a closed embedding if it is such at the level
of the underlying classical prestacks. I.e., if for every $S\in (\affSch)_{/\CZ}$
the base-changed map
$$^{cl}(S\underset{\CZ}\times \CX)\to S$$
is a closed embedding; in particular, $^{cl}(S\underset{\CZ}\times \CX)$ is a classical affine scheme.

\medskip

If $\CX,\CZ\in \on{PreStk}_{\on{laft}}$, it suffices to check the above condition for $S\in (\affSch_{\on{ft}})_{/\CZ}$.

\sssec{}

For $i:\CX\hookrightarrow \CZ$ a closed embedding of objects of $\on{PreStk}_{\on{laft}}$,
let $j:\oCZ\hookrightarrow \CZ$ be the complementary open embedding. The induced map
$$j:\oCZ_\dr\to \CZ_\dr$$
is also an open embedding of prestacks. Consider the restriction functor
$$j^{\dagger,r}:\Crys^r(\CZ)\to \Crys^r(\oCZ).$$

It follows from \cite[Lemma 4.1.1]{IndCoh},
that the above functor admits a \emph{fully faithful} right adjoint, denoted $j_{\dr,*}$, 
such that for every $S\in (\dgSch_{\on{aft}})_{/\CZ_\dr}$
and 
$$\oS:=S\underset{\CZ_\dr}\times \oCZ_\dr\overset{j_S}\hookrightarrow S,$$
the natural transformation in the diagram
$$
\CD
\IndCoh(S)  @<{\oblv_\CZ^r}<<  \Crys^r(\CZ)  \\
@A{(j_S)_*^{\IndCoh}}AA    @AA{j_{\dr,*}}A  \\
\IndCoh(\oS)  @<{\oblv_{\oS}^r}<<  \Crys^r(\oCZ) 
\endCD
$$
arising by adjunction from the diagram
$$
\CD
\IndCoh(S)  @<{\oblv_\CZ^r}<<  \Crys^r(\CZ)  \\
@V{(j_S)^!}VV    @VV{j^{\dagger,r}}V  \\
\IndCoh(\oS)  @<{\oblv_{\oS}^r}<<  \Crys^r(\oCZ),
\endCD
$$
is an isomorphism.

\medskip

In particular, the natural transformation
$$\oblv^r_\CZ\circ j_{\dr,*}\to j^{\IndCoh}_*\circ \oblv^r_{\oCZ}$$ is an isomorphism. 

\sssec{}

Let $\Crys^r(\CZ)_{\CX}$ denote the full subcategory of
$\Crys^r(\CZ)$ equal to $\on{ker}(j^{\dagger,r})$. 

\medskip

Clearly, an object $\CM\in \Crys^r(\CZ)$ belongs to
$\Crys^r(\CZ)_{\CX}$ if and only if for every $S\in \dgSch_{\on{aft}}$, equipped with a map 
$^{cl,red}S\to \CZ$, the corresponding object $\CF_S\in \IndCoh(S)$ lies in 
$$\IndCoh(S)_{S-\oS}:=\on{ker}\left(j^!_S:\IndCoh(S)\to \IndCoh(\oS)\right).$$

\sssec{}

The functor $\Crys^r(\CZ)_{\CX}\hookrightarrow  \Crys^r(\CX)$
admits a right adjoint, given by
$$\CM\mapsto \on{Cone}(\CM\to j_{\dr,*}\circ j^{\dagger,r}(\CM))[-1].$$
Hence, we can think of $\Crys^r(\CZ)_{\CX}$ as a co-localization of $\Crys^r(\CZ)$. 

\sssec{}

Since the composition $i^{\dagger,r}\circ j_{\dr,*}$ is zero, the functor $i^{\dagger,r}:\Crys^r(\CZ)\to \Crys^r(\CX)$ factors through the
above co-localization: 
$$\Crys^r(\CZ)\to \Crys^r(\CZ)_{\CX}\overset{'i^{\dagger,r}}\longrightarrow \Crys^r(\CX).$$

\medskip

Kashiwara's lemma says:

\begin{prop} \label{p:Kashiwara}
The above functor 
$$'i^{\dagger,r}:\Crys^r(\CZ)_{\CX}\to \Crys^r(\CX)$$
is an equivalence.
\end{prop}

\begin{proof}

Note that we have an isomorphism in $\on{PreStk}_{\on{laft}}$:
$$\underset{S\in (\dgSch_{\on{aft}})_{/\CZ_\dr}}{colim}\, S\underset{\CZ_\dr}\times \CX_\dr\simeq \CX_\dr.$$
Furthermore, $S^\wedge:=S\underset{\CZ_\dr}\times \CX_\dr$ identifies with the formal completion of $S$
along $$^{red,cl}S\underset{^{cl,red}\CZ}\times {}^{cl,red}\CX.$$

\medskip

Hence, the category $\Crys^r(\CX)$ can be described as
$$\underset{S\in ((\dgSch_{\on{aft}})_{/\CZ_\dr})^{\on{op}}}{lim}\, \IndCoh(S^\wedge).$$

\medskip

By definition, the category $\Crys^r(\CZ)_{\CX}$ is given by
$$\underset{S\in ((\dgSch_{\on{aft}})_{/\CZ_\dr})^{\on{op}}}{lim}\, \on{ker}\left(j_S^!:\IndCoh(S)\to \IndCoh(\oS)\right).$$

\medskip

Now, \cite[Proposition 7.4.5]{IndSch} says that for any $S$ as above, !-pullback gives an equivalence
$$\on{ker}\left(j^!_S:\IndCoh(S)\to \IndCoh(\oS)\right)\to \IndCoh(S^\wedge),$$
as desired.

\end{proof}

\begin{rem}
If we phrased the above proof in terms of left crystals instead of right crystals, we would
have used \cite[Proposition 7.1.3]{IndSch} instead of \cite[Proposition 7.4.5]{IndSch}.
\end{rem}

\section{Descent properties of crystals}  \label{s:descent for crystals}

In this section all prestacks, including DG schemes and DG indschemes are assumed
locally almost of finite type, unless explicitly stated otherwise. 

\medskip

The goal of this section is to establish a number of properties concerning 
the behavior of crystals on DG schemes and DG indschemes. These properties
include: an interpretation of crystals (right and left) via the infinitesimal groupoid;
h-descent; a monadic description of the category of crystals; induction functors
for right and left crystals. 

\ssec{The infinitesimal groupoid}

In this subsection, we let $\CX$ be a DG indscheme locally almost of finite, see 
\cite[Sect. 1.7.1]{IndSch}.

\sssec{}

Consider the simplicial prestack $\CX^\bullet/\CX_\dr$, i.e., the 
\v{C}ech nerve of the map $\CX\to \CX_\dr$. As was
remarked already, each $\CX^i/\CX_\dr$ is the formal completion of $\CX^i$ along the main diagonal.
In particular, all $\CX^i/\CX_\dr$ also belong to $\dgindSch_{\on{laft}}$. 

\medskip

We shall refer to 
$$\CX\underset{\CX_\dr}\times \CX\rightrightarrows \CX$$
as the infinitesimal groupoid of $\CX$. 

\sssec{}

Consider the cosimplicial category $\IndCoh(\CX^\bullet/\CX_\dr)$.

\begin{prop}  \label{p:inf descent for right}
The functor
$$\Crys^r(\CX)\to \on{Tot}(\IndCoh(\CX^\bullet/\CX_\dr)),$$
defined by the augmentation, is an equivalence.
\end{prop}

\begin{rem}
Note that by \lemref{l:right on cl formally smooth}, the assertion of the proposition holds
also for $\CX$ replaced any classically formally smooth object $\CY\in \on{PreStk}_{\on{laft}}$.
\end{rem}

\begin{proof}

It suffices to show that for any $S\in \dgSch_{\on{aft}}$ and a map $S\to \CX_\dr$, the functor
$$\IndCoh(S)\to \on{Tot}\left(\IndCoh(S\underset{\CX_\dr}\times \CX^\bullet/\CX_\dr)\right)$$
is an equivalence. 

\medskip

Note that the simplicial prestack $S\underset{\CX_\dr}\times (\CX^\bullet/\CX_\dr)$ is the 
\v{C}ech nerve of the map
\begin{equation} \label{e:map to S}
S\underset{\CX_\dr}\times \CX\to S.
\end{equation}

Note that $S\underset{\CX_\dr}\times \CX$ identifies with the formal completion of $S\times \CX$ along
the map $^{red,cl}S\to S\times \CX$, where $^{red,cl}S\to \CX$ is the map corresponding to $S\to \CX_\dr$.
In particular, we obtain that the map in \eqref{e:map to S} is \emph{ind-proper} (see \cite[Sect. 2.7.4]{IndSch},
where the notion of ind-properness is introduced) and surjective. 
 
\medskip

Hence, our assertion follows from \cite[Lemma 2.10.3]{IndSch}.

\end{proof}

\ssec{Fppf and h-descent for crystals}

\sssec{}

Recall the h-topology on the category $\affdgSch_{\on{aft}}$, \cite[Sect. 8.2]{IndCoh}.  It is generated by Zariski covers and proper-surjective covers.


\medskip
 
Consider the functor
$$\Crys^r_{\affdgSch_{\on{aft}}}:=\Crys^r_{\on{PreStk}_{\on{laft}}}|_{\affdgSch_{\on{aft}}}:
(\affdgSch_{\on{aft}})^{\on{op}}\to \StinftyCat.$$

\medskip

We will prove:

\begin{prop}  \label{p:descent for crystals}
The functor $\Crys^r_{\affdgSch_{\on{aft}}}$ satisfies h-descent.
\end{prop}

\begin{proof}

We will show that $\Crys^r_{\affdgSch_{\on{aft}}}$ satisfies \'etale
descent and proper-surjective descent.

\medskip

The \'etale descent statement is clear: if $S'\to S$ is an \'etale cover in $\affdgSch_{\on{aft}}$
then the corresponding map $S'_{\dr}\to S_\dr$ is a schematic, \'etale and surjective map in $\on{PreStk}_{\on{laft}}$.
In particular, it is a cover for the fppf topology, and the statement follows from the fppf descent for 
$\IndCoh$, see \cite[Corollary 10.4.5]{IndCoh}.

\medskip

Thus, let $S'\to S$ be a proper surjective map. Consider the bi-simplicial object of $\on{PreStk}_{\on{laft}}$
equal to
$$(S'{}^\bullet/S)^\star/(S_\dr'{}^\bullet/S_\dr),$$
i.e., the term-wise infinitesimal groupoid of the \v{C}ech nerve of $S'\to S$.  
Namely, it is the bi-simplicial object whose $(p,q)$ simplices are given by the $q$-simplices of 
Cech nerve of the map $S'{}^p/S \rightarrow S_\dr'{}^p/S_\dr$; so $\star$ stands for the index $q$,
and $\bullet$ for the index $p$. 

\medskip

By \propref{p:inf descent for right}, it is enough to show that the composite functor
\begin{multline} \label{e:map to double Cech}
\Crys^r(S):=\IndCoh(S_\dr)\to \on{Tot}\left(\IndCoh((S'{}^\bullet/S)_\dr)\right)\to \\
\to \on{Tot}\left(\IndCoh((S'{}^\bullet/S)^\star/(S'{}^\bullet/S)_\dr)\right).
\end{multline}
is an equivalence.

\medskip

Note, however, that we have a canonical isomorphism of bi-simplicial objects of $\on{PreStk}_{\on{laft}}$
$$(S'{}^\bullet/S)^\star/(S'_\dr{}^\bullet/S_\dr)\simeq (S'{}^\star/S'_\dr)^\bullet/(S^\star/S_\dr),$$
where the latter is the term-wise \v{C}ech nerve of the map of cosimplicial objects
$$(S'{}^\star/S'_\dr)\to (S^\star/S_\dr).$$

\medskip

The map in \eqref{e:map to double Cech} can be rewritten as
$$\Crys^r(S):=\IndCoh(S_\dr)\to \on{Tot}\left(\IndCoh(S^\bullet/S_\dr)\right)\to
\on{Tot}\left(\IndCoh((S'{}^\star/S'_\dr)^\bullet/(S^\star/S_\dr))\right).$$

Applying \propref{p:inf descent for right} again, we obtain that 
it suffices to show that for every $i$, the map
$$\IndCoh(S^i/S_\dr)\to \on{Tot}\left(\IndCoh((S'{}^i/S'_\dr)^\bullet/(S^i/S_\dr))\right)$$
is an equivalence.

\medskip

However, we note that the map
$$S'{}^i/S'_\dr\to S^i/S_\dr$$
is ind-proper and surjective. Hence, the assertion follows from \cite[Lemma 2.10.3]{IndSch}.

\end{proof}

\sssec{}
Consider the fppf topology on the category $\affdgSch_{\on{aft}}$, induced from the fppf topology
on $\affdgSch$ (see \cite[Sect. 2.2]{Stacks}).  Note that every fppf covering is in particular an h-covering.  Therefore, we obtain,

\begin{cor}
The functor $\Crys^r_{\affdgSch_{\on{aft}}}$ satisfies fppf descent.
\end{cor}

As in \cite[Theorem 8.3.2]{IndCoh}, fppf descent is a combination of Nisnevich descent and 
finite-flat descent\footnote{This observation was explained to us by J.~Lurie.}.  In particular, we established fppf descent in the proof 
of \propref{p:descent for crystals} without appealing to the fact that every fppf covering is also an h-covering.

\sssec{}

Fppf (resp. h-) topology on $\affdgSch_{\on{aft}}$ induces the fppf (resp. h-) topology on the full
subcategory $$^{<\infty}\!\affdgSch_{\on{aft}}\subset \affdgSch_{\on{aft}}.$$

\propref{p:descent for crystals} implies:

\begin{cor}
The functor 
$$\Crys^r_{^{<\infty}\!\affdgSch_{\on{aft}}}:=\Crys^r_{\on{PreStk}_{\on{laft}}}|_{^{<\infty}\!\affdgSch_{\on{aft}}}:
({}^{<\infty}\!\affdgSch_{\on{aft}})^{\on{op}}\to \StinftyCat$$
on $^{<\infty}\!\affdgSch_{\on{aft}}$ satisfies h-descent and, in particular, fppf descent.
\end{cor}

Thus by \cite[Corollary 6.2.3.5]{Lu0}, we obtain:

\begin{cor}
Let $\CY_1\to \CY_2$ be a map in $\on{PreStk}_{\on{laft}}$ which is a surjection in the h-topology. Then the natural map
$$\Crys(\CY_2)\to \on{Tot}(\Crys(\CY_1^\bullet/\CY_2))$$
is an equivalence.
\end{cor}

\ssec{The induction functor for right crystals}

\sssec{}

Let $p_s,p_t$ denote the two projections 
$$\CX\underset{\CX_\dr}\times \CX\rightrightarrows \CX.$$

Note that the maps $p_i$, $i=s,t$ are ind-proper. Hence, the functors $p_i^!$ admit left adjoints,
$(p_i)^{\IndCoh}_*$, see \cite[Corollary 2.8.3]{IndSch}.

\medskip


\begin{prop}  \label{p:right ind}  \hfill

\smallskip

\noindent{\em(a)}
The forgetful functor $$\oblv_\CX^r:\Crys(\CX)\to \IndCoh(\CX)$$ admits
a left adjoint, to be denoted $\ind_\CX^r$.

\smallskip

\noindent{\em(b)} 
We have a canonical isomorphism of functors
$$\oblv_\CX^r\circ \ind_\CX^r\simeq (p_t)_*^{\IndCoh}\circ (p_s)^!.$$

\smallskip

\noindent{\em(c)} The adjoint pair
$$\ind_\CX^r:\IndCoh(\CX)\rightleftarrows \Crys^r(\CX):\oblv^r_\CX$$ is monadic, i.e.,
the natural functor from $\Crys^r(\CX)$ to the category of modules in
$\IndCoh(\CX)$ over the monad $\oblv_\CX^r\circ \ind_\CX^r$ is an equivalence. 

\end{prop}

\begin{proof}
By \propref{p:inf descent for right} and \cite[Theorem 6.2.4.2]{Lu2}, it suffices to show
that the co-simplicial category
$$\IndCoh(\CX^\bullet/\CX_\dr)$$
satisfies the Beck-Chevalley condition, i.e. for each $n$, the coface map $$ d^0:
\IndCoh(\CX^n/\CX_\dr) \rightarrow \IndCoh(\CX^{n+1}/\CX_\dr) $$ admits a left adjoint,
to be denoted by $\mathfrak{t}^0$, and for every map $[m]\rightarrow [n]$ in $\Delta$,
the diagram
$$ \xymatrix{\IndCoh(\CX^m/\CX_\dr) \ar[d] & \ar[l]_{\mathfrak{t}^0}
\IndCoh(\CX^{m+1}/\CX_\dr) \ar[d] \\ \IndCoh(\CX^n/\CX_\dr) & \ar[l]_{\mathfrak{t}^0}
\IndCoh(\CX^{n+1}/\CX_\dr)}
$$
which, a priori, commutes up to a natural transformation, actually commutes.

\medskip

In this case, the Beck-Chevalley condition amounts to the adjunction and base change
between $*$-pushforwards and !-pullbacks for ind-proper morphisms between DG indschemes,
and is given by \cite[Proposition 2.9.2]{IndSch}.

\end{proof}


\begin{cor}
The category $\Crys^r(\CX)$ is compactly generated.
\end{cor}

\begin{proof}
The set of compact generators is obtained by applying $\ind^r_\CX$ to
the compact generators $\Coh(\CX)\subset \IndCoh(\CX)$ (see \cite[Corollary 2.4.4]{IndSch}).
\end{proof}

\ssec{The induction functor and infinitesimal groupoid for left crystals}  \label{ss:ind left}

\sssec{}

It follows from \lemref{l:cl formally smooth} that for a smooth classical scheme $X$,
the analogue of \propref{p:inf descent for right} holds for left crystals, i.e., the functor
\begin{equation} \label{e:left to Tor}
\Crys^l(X)=\QCoh(X_\dr)\to \on{Tot}\left(\QCoh(X^\bullet/X_\dr)\right)
\end{equation}
is an equivalence. 

\medskip

By \propref{p:inf descent for right}, the analogous statement for right crystals is true for any DG scheme $X$ (and even a DG indscheme).
However, this is not the case for left crystals.

\sssec{}

We claim:

\begin{prop}  \label{p:groupoud for left}
If a DG scheme $X$ is eventually coconnective, then 
the functor \eqref{e:left to Tor} is an equivalence. 
\end{prop}

\begin{rem}
One can show that the statement of the proposition holds for any DG scheme $X$ locally almost of finite type. 
But the proof is more involved. In addition, 
\lemref{l:left on cl formally smooth}, the statement of the 
proposition holds for any prestack which is classically formally smooth.
\end{rem}

\begin{ex}  \label{ex:deg 2}
Consider the DG scheme $X = \Spec(k[\alpha])$, where $\alpha$ is in degree -2.  This is a good case to have in mind
to produce counterexamples for assertions involving $\Crys^l(X)$.
\end{ex}

\begin{proof}[Proof of \propref{p:groupoud for left}]
We have a commutative diagram of functors
$$
\CD
\Crys^l(X)   @>>> \on{Tot}\left(\QCoh(X^\bullet/X_\dr)\right) \\
@V{\Upsilon_X}VV    @VV{\on{Tot}(\Upsilon_{X^\bullet/X_\dr})}V  \\
\Crys^r(X)   @>>> \on{Tot}\left(\IndCoh(X^\bullet/X_\dr)\right).
\endCD
$$
with the left vertical map and the bottom horizontal map being equivalences.
Hence, we obtain that $\Crys^l(X)$ is a retract of $\on{Tot}\left(\QCoh(X^\bullet/X_\dr)\right)$.

\medskip

Recall that if $Z$ is an eventually coconnective DG scheme, the functor
$$\Upsilon_Z:\QCoh(Z)\to \IndCoh(Z)$$ is fully faithful (see \cite[Corollary 9.6.3]{IndCoh}. 
Hence, by \cite[Propositions 7.1.3 and 7.4.5]{IndSch}, the same is true for the completion
of an eventually coconnective DG scheme along a Zariski-closed subset. Hence, the
functors
$$\Upsilon_{X^i/X_\dr}:\QCoh(X^i/X_\dr)\to \IndCoh(X^i/X_\dr)$$
are fully faithful. Thus, the functor $\on{Tot}(\Upsilon_{X^\bullet/X_\dr})$ in the above commutative
diagram is also fully faithful. But it is also essentially surjective since the identity functor is its retract.

\end{proof}

\sssec{} \label{sss:induction for left}

For a DG scheme $X$, define the functor
$$\ind^l_X:\QCoh(X)\to \Crys^l(X)$$
as 
$$\ind^l_X:=(\Upsilon_{X_\dr})^{-1}\circ \ind^r_X\circ \Upsilon_X.$$

We claim:

\begin{lem} \label{l:ind l}
If $X$ is an eventually coconnective DG scheme, the functors $(\ind^l_X,\oblv^l_X)$
are mutually adjoint.
\end{lem}

\begin{rem}
The assertion of the lemma would be false if we dropped the assumption that $X$
be eventually coconnective. Indeed, in this case the functor $\ind^l_X$ fails to
preserve compactness. 
\end{rem}

\begin{proof}[Proof of \lemref{l:ind l}]

Recall (see \cite[Sect. 9.6.6]{IndCoh}) that for $X$ eventually coconnective, the functor
$\Upsilon_X$ admits a right adjoint, denoted $\Xi^\vee_X$; moreover, the functor
$\Upsilon_X$ itself is fully faithful.

\medskip

We obtain that the right adjoint to $\ind^l_X$ is given by 
$$\Xi^\vee_X\circ \oblv^r_X \circ \Upsilon_{X_\dr}\simeq \Xi^\vee_X\circ \Upsilon_X\circ \oblv^l_X\simeq \oblv^l_X,$$
as required.

\end{proof}

In the course of the proof of \lemref{l:ind l} we have also seen:

\begin{lem} \label{l:oblv l}
The functor $\oblv^l_X$ is canonically isomorphic to
$$\Xi_X^\vee\circ  \oblv^r_X \circ \Upsilon_{X_\dr}.$$
\end{lem}

\sssec{}

We now claim:

\begin{prop}
Let $X$ be an eventually coconnective DG scheme. Then the adjoint pair
$$\ind_X^l:\QCoh(X)\rightleftarrows \Crys^l(X):\oblv_X^l$$ is monadic, i.e.,
the natural functor from $\Crys^l(X)$ to the category of modules in
$\QCoh(X)$ over the monad $\oblv_X^l\circ \ind_X^l$ is an equivalence. 
\end{prop}

\begin{proof}
We need to show that the conditions of the Barr-Beck-Lurie theorem hold.
The functor $\oblv_X^l$ is continuous, and hence commutes with all colimits.
The fact that $\oblv_X^l$ is conservative is given by \lemref{l:cons}.
\end{proof}

\section{$\on{t}$-structures on crystals}
The category of crystals has two natural t-structures, which are compatible with the left and right realizations respectively.  
One of the main advantages of the right realization is that the t-structure compatible with it is much better behaved.  

\medskip

In this section, we will define the two t-structures and prove some of their basic properties. These include:
results on left/right t-exactness and boundedness of cohomological amplitude of the induction/forgetful functors;
the left-completness property of $\Crys$ of a DG scheme; relation to the derived category of the heart
of the t-structure. 

\ssec{The left t-structure}
In this subsection, we do not make the assumption that prestacks be locally almost of finite type.

\sssec{}  \label{sss:can t-str}
Recall \cite[Sec. 1.2.3]{QCoh} that for any prestack $\CZ$, the category $\QCoh(\CZ)$ has a canonical t-structure characterized by 
the following condition: an object $\CF\in \QCoh(\CZ)$
belongs to $\QCoh(\CZ)^{\leq 0}$ if and only if for every $S \in \affdgSch$ and a map $\phi: S\rightarrow \CZ$, we have 
$$\phi^{*}(\CF)\in \QCoh(S)^{\leq 0}.$$

\medskip

In particular, taking $\CZ = \CY_{\dr}$ for some prestack $\CY$, we obtain a canonical t-structure on $\Crys^{l}(\CY)$, which we shall call the 
``left t-structure.''

\medskip

By definition, the functor
$$ \oblv^{l}: \Crys^{l}(\CX)\rightarrow \QCoh(\CX) $$
is right t-exact for the left t-structure.

\sssec{} In general, the left t-structure is quite poorly behaved.  However, we have the following assertion:

\begin{prop}\label{p:left t-struct}
Let $\CY$ be a classically formally smooth prestack. Then
$$\CM\in\Crys^{l}(\CY)^{\leq 0} \Leftrightarrow \oblv^{l}_{\CY}(\CM)\in \QCoh(\CY)^{\leq 0}.$$
\end{prop}

\begin{proof}
We need to show that if $\CM\in \Crys^{l}(\CY)$ is such that $\oblv^{l}_{X}(\CM) \in \QCoh(\CY)^{\leq 0}$ 
then $\CM\in \Crys^{l}(\CY)^{\leq 0}$. 
I.e., we need to show that for every $S\in \affdgSch$ and $\phi: S\rightarrow \CY_{\dr}$, $\phi^{*}(\CM)\in \QCoh(S)^{\leq 0}$.

\medskip

Let $\CY^{\bullet}/\CY_{\dr}$ be the \v{C}ech nerve of the map $p_{\dr,\CY}: \CY\rightarrow \CY_{\dr}$.  By \lemref{l:cl formally smooth},
there exists a map $\phi': S\rightarrow \CY$ and an isomorphism $\phi \simeq p_{\dr,\CY} \circ \phi'$.  
The assertion now follows from the fact that 
$\phi'^{*}$ is right t-exact.
\end{proof}

\ssec{The right t-structure}
From this point until the end of this section we reinstate the assumption that all prestacks are locally 
almost of finite type, unless explicitly stated otherwise.

\medskip

In this subsection we shall specialize to the case of DG schemes.

\sssec{}

Let $X$ be a DG scheme. Recall that the category $\IndCoh(X)$ has a natural t-structure, compatible with
filtered colimits, see \cite[Sect. 1.2]{IndCoh}.

\medskip

It is characterized by the property that an object of $\IndCoh(X)$ is connective (i.e., lies in $\IndCoh(X)^{\leq 0}$)
if and only if its image under the functor $\Psi_X:\IndCoh(X)\to \QCoh(X)$ is connective.

\sssec{}

We define the right t-structure on $\Crys^r(X)$ by declaring that 
$$\CM\in \Crys^r(X)^{\geq 0}\, \Leftrightarrow \, \oblv^r_X(\CM)\in \IndCoh(X)^{\geq 0}.$$

In what follows, we shall refer to the right t-structure on $\Crys^{r}(X)$ as ``the'' t-structure on crystals.  In other words, by 
default the t-structure we shall consider will be the right one.  By construction, this t-structure is also compatible with 
filtered colimits, since $\oblv^r_X$ is continuous.

\sssec{}  \label{sss:t and Zariski}

We claim that the right t-structure on $\Crys^r(X)$ is Zariski-local,
i.e., an object is connective/coconnective if and only if its restriction to a Zariski
cover has this property. Indeed, this follows from the corresponding
property of the t-structure on $\IndCoh(X)$, see \cite[Corollaty 4.2.3]{IndCoh}.

\sssec{The right $\on{t}$-structure and Kashiwara's lemma} \label{sss:cl pushforward}

Let $i:X\to Z$ be a closed embedding of DG schemes. Let $i_{\dr,*}$ denote the
functor $\Crys^r(X)\to \Crys^r(Z)$ equal to the composition
$$\Crys^r(X)\overset{('i^{\dagger,r})^{-1}}\longrightarrow \Crys^r(Z)_{X}\hookrightarrow \Crys^r(Z),$$
which, by construction, is the left adjoint of $i^{\dagger,r}$.

\medskip

We have:

\begin{prop} \label{p:cl embed exact}
The functor $i_{\dr,*}:\Crys^r(X)\to \Crys^r(Z)$ is t-exact.
\end{prop}

\begin{proof}

Note that the full subcategory
$$\Crys^r(Z)_{X}\subset \Crys^r(Z)$$
is compatible with the t-structure, since it is the kernel of the functor $j^{\dagger,r}$,
which is t-exact (here $j$ denotes the open embedding $Z-X\hookrightarrow Z$). 

\medskip

Hence, it remains to show that the functor
$$'i^{\dagger,r}:\Crys^r(Z)_{X}\to \Crys^r(X)$$
is t-exact.

\medskip

Thus, we need to show that for $\CM\in \Crys^r(Z)_{X}$ we have:
$$\oblv^r_Z(\CM)\in \IndCoh(Z)^{>0}\, \Leftrightarrow\, \oblv^r_X(i^{\dagger,r}(\CM))\in \IndCoh(X)^{>0}.$$

\medskip

Recall the notation
$$ \IndCoh(Z)_X:=\on{ker}\left(j^!:\IndCoh(Z)\to \IndCoh(Z-X)\right).$$

It suffices to show the following:

\begin{lem} \label{l:ineq r}
Let $i:X\to Z$ be a closed embedding. Then for $\CF\in \IndCoh(Z)_X$ we have:
$$\CF\in \IndCoh(Z)^{>0} \, \Leftrightarrow\, i^!(\CF)\in \IndCoh(X)^{>0}.$$
\end{lem}

\end{proof}

\begin{proof}[Proof of \lemref{l:ineq r}]

The $\Rightarrow$ implication follows from the fact that $i^!$ is left t-exact
(being the right adjoint of the t-exact functor, namely, $i^{\IndCoh}_*$).

\medskip

For the converse implication, note that the full subcategory 
$$\IndCoh(Z)_X\subset \IndCoh(Z)$$
is also compatible with the t-structure, since it is the kernel of the t-exact functor $j^!$.
Furthermore, it follows from \cite[Proposition 4.1.7(b)]{IndCoh} that the t-structure on $\IndCoh(Z)_X$ is generated
by the t-structure on
$$\Coh(Z)_X:=\on{ker}\left(j^*:\Coh(Z)\to \Coh(Z-X)\right).$$

\medskip

Let $\CF\in \IndCoh(Z)_X$ be such that $i^!(\CF)\in \IndCoh(X)^{>0}$. We need to show that
$\CF$ is right-orthogonal to $(\Coh(Z)_X)^{\leq 0}$. By assumption, $\CF$ is right-orthogonal
to the essential image of $\Coh(X)^{\leq 0}$ under 
$$\Coh(X)\overset{i_*}\longrightarrow \Coh(Z)_X\to \IndCoh(Z)_X.$$
However, it is easy to see that every object of $(\Coh(Z)_X)^{\leq 0}$ can be obtained
as a finite successive extension of objects in the essential image of $\Coh(X)^{\leq 0}$,
which implies the required assertion.

\end{proof}

\begin{cor}  \label{c:t structure indep}
If a map $X_1\to X_2$ of DG schemes indices an isomorphism $$^{cl,red}\!X_1\to {}^{cl,red}\!X_2,$$
then the corresponding t-structures on $\Crys^r(X_1)\simeq \Crys^r(X_2)$ coincide.
\end{cor}

\sssec{}

Let $X$ be a DG scheme. By construction, the forgetful functor $\oblv^r_X$ is left t-exact. Hence, by adjunction,
the functor $\ind^r_X$ is right t-exact. 

\medskip

We now claim:

\begin{prop} \label{p:t and crys gen}
The functor $\ind^r_X$ is t-exact.
\end{prop}

\begin{proof}

It suffices to show that the composition $\oblv^r_X\circ \ind^r_X$
is left t-exact. We deduce this from \propref{p:right ind}(b):

\medskip

The functor $p_s^!$ is left t-exact (e.g., by \lemref{l:ineq r} applied
to $\Delta_X:X\to X\times X$). The functor $(p_t)^{\IndCoh}_*$ is left t-exact (in fact, $t$-exact)
by \cite[Lemma 2.7.11]{IndSch}.

\end{proof}

\sssec{}

We now claim:

\begin{prop} \label{p:t and crys sm} \hfill

\smallskip

\noindent{\em(a)} If $X$ is a smooth classical scheme, then $\oblv^r_X$ is t-exact.

\smallskip

\noindent{\em(b)} For a quasi-compact DG scheme $X$, the functor $\oblv^r_X$ is of
bounded cohomological amplitude. 

\end{prop}

\begin{proof}

Let $X$ be a smooth classical scheme. By the definition of the t-structure on $\Crys^r(X)$, the essential
image of $\IndCoh(X)^{\leq 0}$ under $\ind^r_X$ generates $\Crys^r(X)^{\leq 0}$
by taking colimits. Hence, in order to show that $\oblv^r_X$ is right t-exact,
it suffices to show the same for the functor $\oblv^r_X\circ \ind^r_X$. We will deduce this
from \propref{p:right ind}(b): 

\medskip

We can write
$$(X\times X)^\wedge_X\simeq \underset{n}{colim}\, X_n,$$
where $X_n\overset{i_n}\to X\times X$ is the $n$-th infinitesimal neighborhood of the diagonal. 
Hence, by \cite[Equation (2.2)]{IndSch},
$$(p_t)^{\IndCoh}_*\circ p_s^!\simeq \underset{n}{colim}\,(p_t\circ i_n)^{\IndCoh}_*\circ (p_s\circ i_n)^!.$$
Now, each of the functors $(p_t\circ i_n)^{\IndCoh}_*$ is t-exact by \cite[Lemma 2.7.11]{IndSch}, and each of
the functors $(p_s\circ i_n)^!$ is t-exact because $p_s\circ i_n:X_n\to X$ is finite and flat.

\medskip

Now, let $X$ be a quasi-compact DG scheme, and let us show that $\oblv^r_X$ is of bounded 
cohomological amplitude. The question readily reduces to the case when $X$ is affine,
and let $i:X\hookrightarrow Z$ be a closed embedding, where $Z$ is smooth. By 
\propref{p:cl embed exact} and point (a), it suffices to show that the functor
$i^!:\IndCoh(Z)\to \IndCoh(X)$ is of bounded cohomological amplitude, but the
latter follows easily from the fact that $Z$ is regular.

\end{proof}

\ssec{Right t-structure on crystals on indschemes}

\sssec{}

Let $\CX$ be a DG indscheme.  Fix a presentation of $\CX$
\begin{equation}\label{e:pres indscheme}
\CX = \underset{\alpha}{colim}\  X_\alpha
\end{equation}
as in \cite[Prop. 1.7.6]{IndSch}.  For each $\alpha$, let $i_\alpha$ denote the corresponding closed embedding 
$X_\alpha\rightarrow \CX$, and for each $\alpha_1\to \alpha_2$ let $i_{\alpha_1,\alpha_2}$ denote the closed
embedding $X_{\alpha_1}\to X_{\alpha_2}$. 

\medskip

We have:
$$\Crys^r(\CX)\simeq  \underset{\alpha}{lim}\ \Crys^r(X_\alpha),$$
where for $\alpha_1\to \alpha_2$, the functor $\Crys^r(X_{\alpha_2})\rightarrow \Crys^r(X_{\alpha_1})$ is 
given by $i_{\alpha_1,\alpha_2}^{\dagger,r}$.

\medskip

Hence, by \cite[Sect. 1.3.3]{DG}, we have that
$$ \Crys^r(\CX) \simeq \underset{\alpha}{colim}\ \Crys^r(X_\alpha), $$
where for $\alpha_1\to \alpha_2$, the functor $\Crys^r(X_{\alpha_1})\rightarrow \Crys^r(X_{\alpha_2})$ is 
given by $(i_{\alpha_1,\alpha_2})_{\dr,*}$.

\medskip

In particular, for each $\alpha$, we obtain a pair of adjoint functors
$$(i_{\alpha})_{\dr,*}: \Crys^r(X) \rightleftarrows \Crys^r(\CX) : i_\alpha^{\dagger,r}.$$

\sssec{}
Recall from \cite[Sect. 2.5]{IndSch} that $\IndCoh(\CX)$ has a natural t-structure compatible with
filtered colimits.

\medskip

Using this t-structure on $\IndCoh(\CX)$, we can define the right t-structure on $\Crys^r(\CX)$.  Namely, we have
$$\CM \in \Crys^r(\CX)^{\geq 0} \ \Leftrightarrow\  \oblv^r_{\CX}(\CM) \in \IndCoh(\CX)^{\geq 0} $$
Since $\oblv^r_\CX$ preserves colimits, this t-structure is compatible with filtered colimits.
We can describe this t-structure more explicitly using the presentation \eqref{e:pres indscheme}, in a 
way analogous to \cite[Lemma 2.5.3]{IndSch} for the t-structure on $\IndCoh(\CX)$:

\begin{lem}
Under the above circumstances, we have:

\smallskip

\noindent{\em(a)} An object $\CF\in \Crys^r(\CX)$ belongs to $\Crys^r(\CX)^{\geq 0}$ if and only if for
every $\alpha$, the object $i_\alpha^{\dagger,r}(\CF)\in \Crys^r(X_\alpha)$ belongs to $\Crys^r(X_\alpha)^{\geq 0}$.

\smallskip

\noindent{\em(b)} The category $\Crys^r(\CX)^{\leq 0}$ is generated under colimits by the essential images of 
the functors $(i_\alpha)_{\dr,*}\left(\Crys^r(X_\alpha)^{\leq 0}\right)$.

\end{lem}
\begin{proof}
Point (a) follows from the definition and \cite[Lemma 2.5.3(a)]{IndSch}.  Point (b) follows formally from point (a).
\end{proof}

\sssec{}
Suppose that $i: X \to \CX$ is a closed embedding of a DG scheme into a DG indscheme. By the exact same argument as in 
\cite[Lemma 2.5.5]{IndSch}, we have:

\begin{lem}
The functor $i_{\dr,*}$ is t-exact.
\end{lem}

\sssec{}

As an illustration of the behavior of the above t-structure on right crystals over a DG indscheme, let us consider the following
situation. Let $i:X\to Z$ be a closed embedding of quasi-compact DG schemes. Let $Y$ denote the formal completion
of $X$ in $Z$, considered as an object of $\dgindSch_{\on{laft}}$; let $'i$ denote the resulting map $X\to Y$.

\medskip

We claim:

\begin{lem}   \label{l:t structure on form compl}
The equivalence $\Crys^r(X)\simeq \Crys^r(Y)$, induced by the isomorphism $'i_\dr:X_\dr\to Y_\dr$,
is compatible with the t-structures.
\end{lem}

\begin{proof}[Proof 1]
Follows from \propref{p:cl embed exact} and the fact that the equivalence
$$\IndCoh(Y)\simeq \IndCoh(Z)_Y$$
of \cite[Proposition 7.4.5]{IndSch} is compatible with the t-structures
(see \cite[Lemma 7.4.8]{IndSch}).
\end{proof}

\begin{proof}[Proof 2]
From the commutative diagram
$$
\CD
\Crys^r(X)  @<{'i^{\dagger,r}}<< \Crys^r(Y)  \\
@V{\oblv^r_X}VV   @VV{\oblv^r_Y}V   \\
\IndCoh(X)   @<{'i^!}<<  \IndCoh(Y).
\endCD
$$
we obtain that it suffices to show that for $\CF\in \IndCoh(Y)$ we have
$$\CF\in \IndCoh(Y)^{>0} \, \Leftrightarrow\, {}'i^!(\CF)\in \IndCoh(X)^{>0},$$
which follows formally from \cite[Lemma 7.4.8]{IndSch} and 
\lemref{l:ineq r} (or can be easily proved directly).
\end{proof}

\ssec{Further properties of the left t-structure}  \label{ss:more on left t}

\sssec{}

First, let us describe the relation between the left and the right t-structures on crystals in the case of a smooth classical scheme.

\begin{prop}\label{p:left to right t-struct}
Let $X$ be a smooth classical scheme of dimension $n$.  Then
$$\CF\in \Crys^{l}(X)^{\leq 0} \Leftrightarrow \CF\in \Crys^{r}(X)^{\leq -n}.$$
I.e., the left t-structure agrees with the right t-structure up to a shift by the dimension of $X$.
\end{prop}

\begin{proof}
Recall that the two forgetful functors are related by the commutative diagram
$$
\xy
(-20,0)*+{\QCoh(X)}="A";
(20,0)*+{\IndCoh(X)}="B";
(0,20)*+{\Crys(X).}="C";
{\ar@{->}_{\oblv^l_X} "C";"A"};
{\ar@{->}^{\oblv^r_X} "C";"B"};
{\ar@{->}_{\Upsilon_X} "A";"B"};
\endxy
$$
In the case that $X$ is a smooth classical scheme of dimension $n$, the functior $\Upsilon_X$ is an equivalence and maps 
$\QCoh(X)^{\leq 0}$ isomorphically to $\IndCoh(X)^{\leq -n}$.  
The assertion now follows from \propref{p:left t-struct} and \propref{p:t and crys sm}(a), combined with the fact that
$\oblv^r_X$ is conservative. 

\end{proof}

\sssec{}

The next proposition compares the ``left" and ``right" t-structures on $\Crys(X)$ for an arbitrary DG scheme $X$.

\begin{prop} \label{p:left t vs right t}
Let $X$ be quasi-compact. Then the identity functor
$$ \Crys^l(X)\rightarrow \Crys^r(X) $$
has bounded amplitude, i.e. the difference between the left and right t-structures is bounded.
\end{prop}

\begin{proof}

Without loss of generality, we can assume that $X$ is affine. Let $Z$ be a smooth classical scheme of dimension $n$;
$i:X\hookrightarrow Z$ a closed embedding. We claim that for $\CM^l\in \Crys^l(X)$ and the
corresponding object $\CM^r\in \Crys^r(X)$ we have
\begin{multline} \label{e:estimates}
\CM^l\in (\Crys^l(X))^{\leq 0}\, \Rightarrow\, \CM^r\in (\Crys^r(X))^{\leq 0} \text{ and } \\
\CM^r\in (\Crys^r(X))^{\leq 0}\, \Rightarrow\, \CM^l\in (\Crys^l(X))^{\leq n}.
\end{multline}

\medskip

Let $U\overset{j}\hookrightarrow Z$ denote the 
complementary open embedding. Let $Y$ denote the formal completion of
$X$ in $Z$; let $\wh{i}$ denote map $Y\to Z$. 

\medskip

The map $X\to Y$ defines an isomorphism $X_\dr\to Y_\dr$, which allows to identify
$\Crys^l(X)\simeq \Crys^l(Y)$. Applying \propref{p:left t-struct}, we have:
\begin{equation} \label{e:est 1}
\CM^l\in (\Crys^l(X))^{\leq 0}\, \Leftrightarrow\, \oblv_Y^l(\CM^l)\in \QCoh(Y)^{\leq 0},
\end{equation}
where the t-structure on $\QCoh(Y)$ is that of \secref{sss:can t-str}.

\medskip

Consider the subcategory $\QCoh(Z)_X\subset \QCoh(X)$ which is by definition equal to
$$\on{ket}(j^*:\QCoh(Z)\to \QCoh(U)).$$
This subcategory is compatible with the t-structure on $\QCoh(Z)$, since the functor $j^*$
is t-exact.

\medskip

Recall (see \cite[Proposition 7.1.3]{IndSch}) that the functor $\wh{i}^*$ defines an equivalence
$$\QCoh(Z)_X\to \QCoh(Y).$$

\medskip

Let $\CF$ be the object of $\QCoh(Z)_X$ corresponding to $\oblv_Y^l(\CM^l)\in \QCoh(Y)$.
We have:
$$\Upsilon_Z(\CF)\simeq \oblv^r_Z(i_{\dr,*}(\CM^r)).$$

Since the functor $i_{\dr,*}$ is t-exact (\propref{p:cl embed exact}), and since $\Upsilon_Z$ shifts 
cohomological degrees by $[-n]$, we have:
\begin{equation} \label{e:est 2}
\CM^r\in (\Crys^r(X))^{\leq 0} \, \Leftrightarrow \CF\in (\QCoh(Z)_X)^{\leq n}.
\end{equation}

\medskip

Combining \eqref{e:est 1} and \eqref{e:est 2}, 
the implications in \eqref{e:estimates} follow from the next assertion:

\begin{lem}  \label{l:comparing t}
The equivalence $\wh{i}^*:\QCoh(Z)_X\simeq \QCoh(Y)$ has the following properties with respect to the
t-structure on $\QCoh(Z)_X$ inherited from $\QCoh(Z)$ and the t-structure on $\QCoh(Y)$ of
\secref{sss:can t-str}:

\smallskip

\noindent{\em(a)} If $\CF\in (\QCoh(Z)_X)^{\leq 0}$ then $\wh{i}^*(\CF)\in \QCoh(Y)^{\leq 0}$. 

\smallskip

\noindent{\em(b)} If $\wh{i}^*(\CF)\in \QCoh(Y)^{\leq 0}$, then $\CF\in (\QCoh(Z)_X)^{\leq n}$.
\end{lem}

\end{proof}

\begin{proof}[Proof of \lemref{l:comparing t}]

Point (a) follows from the fact that the functor $\wh{i}^*$ is right t-exact. 

\medskip

To prove point (b),
we note that the category $\QCoh(Y)^{\leq 0}$ is generated under taking colimits by
the essential image of $\QCoh(Z)^{\leq 0}$ under the functor $\wh{i}^*$, see
\cite[Proposition 7.3.5]{IndSch}. Hence, it is sufficient to show the the functor
$$\QCoh(Z)\overset{\wh{i}^*}\longrightarrow \QCoh(Y)\simeq \QCoh(Z)_X$$
has cohomological amplitude bounded by $n$. However, the above functor is
the right adjoint to the embedding
$$\QCoh(Z)_X\hookrightarrow \QCoh(Z),$$
and is given by
$$\CF'\mapsto \on{Cone}(\CF'\to j_*\circ j^*(\CF'))[-1].$$
Now, $j^*$ is t-exact, and $j_*$ is of cohomological amplitude bounded by $n-1$.
This implies the required assertion.

\end{proof}

\sssec{}

Let $X$ be an arbitrary quasi-compact DG scheme. We have:

\begin{prop}  \label{p:amplitude for oblv l} \hfill

\smallskip

\noindent{\em(a)}
The functor $\oblv^l_X:\Crys(X)\to \QCoh(X)$ has bounded cohomological amplitude.

\smallskip

\noindent{\em(b)}
If $X$ is eventually coconnective, the functor $\ind^l_X:\QCoh(X)\to \Crys(X)$
has cohomological amplitude bounded from above. 

\end{prop}

\begin{proof}

For point (a) we can assume that $X$ is affine and find a closed embedding
$i:X\hookrightarrow Z$, where $Z$ is a smooth classical scheme. In this
case, the assertion follows from \propref{p:cl embed exact} and the fact that
the functor
$$i^*:\QCoh(Z)\to \QCoh(X)$$
has a bounded cohomological amplitude.

\medskip

Point (b) follows from point (a) by the $(\ind^l_X,\oblv^l_X)$-adjunction.

\end{proof}

\begin{rem}
The assumption that $X$ be eventually coconnective in point (b) is essential; otherwise
a counterexample can be provided by the DG scheme from Example \ref{ex:deg 2}. In addition,
is it easy to show that $\ind^l_X$ has a cohomological amplitude bounded from
below if and only if $X$ is Gorenstein (see \lemref{l:Drinfeld}).
\end{rem}

\ssec{Left completeness}

\sssec{}

Let $X$ be an affine smooth classical scheme. We observe that in this case 
the category $\Crys^r(X)$ contains a canonical object
$$\ind_X^r(\CO_X),$$
which lies in the heart of the t-structure (see \propref{p:t and crys gen}), and is \emph{projective}, i.e.,
$$H^0(\CN)=0 \, \Rightarrow \Hom_{\Crys^r(X)}(\ind_X^r(\CO_X),\CN)=0.$$
Moreover, $\ind_X^r(\CO_X)$ is a compact generator of $\Crys^r(X)$. This implies:

\begin{cor}  \label{c:derived category of heart smooth a}
Let $X$ be an affine smooth classical scheme. Then the 
category $\Crys^r(X)$ is left-complete in its t-structure. 
\end{cor}

\sssec{}

The above corollary implies left-completeness for any DG scheme $X$:

\begin{cor}  \label{c:left complete}
For any DG scheme $X$, the category $\Crys^r(X)$ is left-complete in the ``right" t-structure.
\end{cor}

\begin{proof}

First, we note that the property of being left-complete is Zariski-local (proved by the same argument
as \cite[Proposition 5.2.4]{QCoh}). Hence, we can assume without loss of generality that $X$ is affine. 
Choose a closed embedding
$i:X\hookrightarrow Z$, where $Z$ is a smooth classical scheme. Now the assertion
follows formally from the fact that the functor $i_{\dr,*}$ is continuous, fully faithful
(by \propref{p:Kashiwara}), t-exact (by \propref{p:cl embed exact}),
and the fact that $\Crys^r(Z)$ is left-complete (by the previous corollary). 

\medskip

Here is an alternative argument:

\medskip

By \corref{c:t structure indep}, we can assume that $X$ is eventually coconnective. In this case,
the functor $\oblv^l_X$ commutes with limits, as it admits a left adjoint. Moreover, by \lemref{l:cons},
$\oblv^l_X$ is conservative, and by \propref{p:amplitude for oblv l} it has bounded cohomological
amplitude. Therefore, the fact that $\QCoh(X)$ is left-complete in its t-structure 
implies the corresponding fact for $\Crys^r(X)$.

\end{proof}

\begin{rem}
The question of right completeness is not an issue: since our t-structures are compatible with filtered
colimits, right completeness is equivalent to the t-structure being separated on the coconnective
subcategory, which is evident since $\oblv^r_X$ is left t-exact and conservative, and the t-structure
on $\IndCoh(X)^+\overset{\Psi_X}\simeq \QCoh(X)^+$ has this property.
\end{rem}

\sssec{}

Combining \corref{c:left complete} with \propref{p:left t vs right t}, we obtain:

\begin{cor}
For a quasi-compact DG scheme $X$, the category $\Crys(X)$ is also left-complete 
in the ``left" t-structure.
\end{cor}

\ssec{The ``coarse" induction and forgetful functors}   \label{ss:coarse}

\sssec{}

Let $X$ be a DG scheme. Recall that the functor $\Psi_X$ identifies the 
category $\QCoh(X)$ with the left-completion of
$\IndCoh(X)$ (see \cite[Proposition 1.3.4]{IndCoh}). 

\medskip

Since the category $\Crys^r(X)$ is left-complete in its t-structure, and the functor $\ind^r_X$ is t-exact,
by the universal property of left completions, we obtain:

\begin{cor}
The functor $\ind^r_X$ canonically factors as
$$\IndCoh(X)\overset{\Psi_X}\longrightarrow \QCoh(X)\overset{'\ind^r_X}\longrightarrow \Crys^r(X).$$
\end{cor}

\sssec{}

We can also consider the functor 
$$'\oblv_X^r:\Crys^r(X)\to \QCoh(X),$$
given by $\Psi_X\circ \oblv_X^r$, where $\Psi_X:\IndCoh(X)\to \QCoh(X)$ is the functor
of \cite[Sect. 1.1.5]{IndCoh}.

\medskip

It is clear that the functor $'\oblv_X^r$ has a finite cohomological amplitude. Indeed, the follows 
from the corresponding fact for $\oblv_X^r$ and the fact that $\Psi_X$ is t-exact
(see \cite[Lemma 1.2.2]{IndCoh}).

\begin{prop}  \label{p:coarse cons}
The functor $'\oblv_X^r$ is conservative.
\end{prop}

\begin{proof}
The assertion is Zariski-local, so we can assume that $X$ is affine. Choose a closed
embedding $i:X\to Z$, where $Z$ is a smooth classical affine scheme.

\medskip

Let $i^{\QCoh,!}:\QCoh(Z)\to \QCoh(X)$ denote the right adjoint 
of $i_*:\QCoh(X)\to \QCoh(Z)$.\footnote{Although this is irrelevant for us, we note that
the $i^{\QCoh,!}$ is continuous. This is because the functor $i_*:\QCoh(X)\to \QCoh(Z)$ sends 
compact objects to compacts (since $Z$ is regular, any coherent sheaf on it is perfect).}
It is easy to see that we have a canonical isomorphism of functors
$$\Psi_X\circ i^!\simeq i^{\QCoh,!}\circ \Psi_Z.$$

Hence, for $\CM\in \Crys^r(Z)$, we have
$$'\oblv_X^r(i^{\dagger,r}(\CM))\simeq i^{\QCoh,!}({}'\oblv^r_Z(\CM)).$$

Applying Kashiwara's lemma, the assertion of the proposition follows from the next lemma:

\begin{lem} \label{l:ugly cons}
The functor $i^{\QCoh,!}:\QCoh(Z)\to \QCoh(X)$ is conservative when restricted to $\QCoh(Z)_X$.
\end{lem}

\end{proof}

\begin{proof}[Proof of \lemref{l:ugly cons}]
We need to show that the essential image of the functor $i_*:\QCoh(X)\to \QCoh(Z)$ generates
$\QCoh(Z)_X$. 

\medskip

First, we claim that $\QCoh(Z)_X$ is generated by the subcategory of bounded objects, denoted
$(\QCoh(Z)_X)^b$. This
follows from the corresponding fact for $\QCoh(Z)$ and the fact that the inclusion
$\QCoh(Z)_X\hookrightarrow \QCoh(Z)$ has a right adjoint of bounded cohomological 
amplitude. By devissage, we obtain that $\QCoh(Z)_X$ is generated by 
$(\QCoh(Z)_X)^\heartsuit$, and further by $(\QCoh(Z)_X)^\heartsuit\cap \Coh(Z)$.

\medskip

However, it is clear that every object of $(\QCoh(Z)_X)^\heartsuit\cap \Coh(Z)$ is a finite
extension of objects lying in the essential image of $\Coh(X)^\heartsuit$.

\end{proof}

\begin{rem}
In the case when $X$ is eventually coconnective we will give a cleaner proof of \propref{p:coarse cons}, below. 
\end{rem}

\sssec{}

Assume now that $X$ is eventually coconnective. Recall that in this case the functor $\Psi_X$ 
admits a fully faithful left adjoint $\Xi_X$ (see \cite[Proposition 1.5.3]{IndCoh}).

\medskip

We observe:

\begin{lem} \label{l:coarse adj}
There exists a canonical isomorphism $'\ind^r_X\simeq \ind^r_X\circ \Xi_X$.
\end{lem}

\begin{proof}
Follows from the isomorphisms $\ind^r_X\simeq {}'\ind^r_X\circ \Psi_X$ and 
$\Psi_X\circ \Xi_X\simeq \on{Id}_{\QCoh(X)}$.
\end{proof}

\begin{cor} \label{c:coarse adj}
The functors $('\ind^r_X,{}'\oblv_X^r)$ form an adjoint pair.
\end{cor}

\begin{proof}
Follows formally from \lemref{l:coarse adj} by adjunction.
\end{proof}

\begin{rem}
The functors $('\ind^r_X,{}'\oblv_X^r)$ are \emph{not} adjoint unless $X$ is eventually coconnective.
Indeed, if $X$ is not eventually coconnective, the functor $'\ind^r_X$ does not preserve compact objects: it sends $\CO_X\in \QCoh(X)$
to a non-compact object of $\Crys^r(X)$.
\end{rem}

\begin{proof}[Alternate proof of \propref{p:coarse cons}]

By \corref{c:coarse adj}, the assertion of \propref{p:coarse cons} (in the eventually coconnective case)
is equivalent to the fact that the essential image of the functor $'\ind_X^r$ generates $\Crys^r(X)$.
However, the latter is tautological from the corresponding fact for $\ind^r_X$.

\end{proof}

\sssec{}

Let $X$ be an eventually coconnective DG scheme, and consider the pair of adjoint functors
$$\Xi_X:\QCoh(X)\rightleftarrows \IndCoh(X):\Psi_X$$
with $\Xi_X$ being fully faithful (see \cite[Sect. 1.4]{IndCoh}).

\medskip

We have seen that the functor $\ind^r_X$ factors through the colocalization functor
$\Psi_X$. However, it is \emph{not} true in general that the functor $\oblv^r_X$
factors through $\Xi_X$, i.e., that it takes values in $\QCoh(X)$, considered as a 
full subcategory of $\IndCoh(X)$ via $\Xi_X$.

\medskip

In fact, the following holds:

\begin{lem}[Drinfeld]  \label{l:Drinfeld}
The functor $\oblv^r_X$ factors through the essential image of $\QCoh(X)$ under $\Xi_X$
if and only if $X$ is Gorenstein.
\end{lem}

Recall that a DG scheme $X$ is said to be Gorenstein if:

\medskip 

\noindent (a) $\omega_X\in \Coh(X)$
(which is equivalent to $X$ being eventually coconnective, see \cite[Proposition 9.6.11]{IndCoh}); 

\medskip

\noindent(b) When considered as a coherent sheaf, $\omega_X$ is a graded line bundle
(which is equivalent to $\omega_X\in \QCoh(X)^{\on{perf}}$, see \cite[Corollary 7.4.3]{IndCoh}).

\medskip

\begin{proof}

Suppose that $\oblv^r_X$ factors through $\QCoh(X)$. In
particular, we obtain that $\omega_X\in \Coh(X)$ lies in the essential image of
$\Xi_X$. Now the assertion follows from \cite[Lemma 1.5.8]{IndCoh}.

\medskip

For the opposite implication, we write $\oblv^r_X(\CM)$ as 
$$\Upsilon_X(\CM)=\oblv^l_X(\CM)\otimes \omega_X,$$
where the tensor product is understood in the sense of the action of
$\QCoh(X)$ on $\IndCoh(X)$, see \cite[Sect. 1.4]{IndCoh}. Recall also
that the functor $\Xi_X$ is tautologically compatible with the above action of
$\QCoh(X)$. Hence, if $\omega_X$, being perfect, lies in the essential image of 
$\Xi_X$, then so does $\oblv^l_X(\CM)\otimes \omega_X$

\end{proof}

\ssec{Relation to the abelian category}

In this subsection we let $X$ be an affine DG scheme.  We will relate the category $\Crys^r(X)$ to a more 
familiar object. 

\sssec{}

Since the t-structure on $\Crys^r(X)$ is compatible with filtered colimits, we obtain that 
$\Crys^r(X)^\heartsuit$ is a Grothendieck abelian category.

\medskip

Using the fact that $\Crys^r(X)$ is \emph{right-complete} in its structure, by reversal of arrows in 
\cite[Theorem 1.3.2.2]{Lu2}, we obtain a canonically
defined t-exact functor
\begin{equation} \label{e:from derived +}
D\left(\Crys^r(X)^\heartsuit\right)^+\to \Crys^r(X)^+,
\end{equation}
where $D(-)^+$ denotes the eventually coconnective part of the derived category of the abelian category. 
 
\sssec{}

We are going to prove:

\begin{prop} \label{p:relation to abelian}
The functor \eqref{e:from derived +} uniquely extends to an equivalence of categories
$$D\left(\Crys^r(X)^\heartsuit\right)\to \Crys^r(X).$$
\end{prop}

The rest of this subsection is devoted to the proof of \propref{p:relation to abelian}. Without
loss of generality, we can assume that $X$ is classical. 

\sssec{Step 1}

Assume first that $X$ is a smooth classical scheme. In this case the assertion is
obvious from the fact that
$$\ind^r_X(\CO_X)$$
is a compact projective generator for both categories. 

\sssec{Step 2}

Let us show that the functor
$$D\left(\Crys^r(X)^\heartsuit\right)^+\to \Crys^r(X)^+$$
is an equivalence. For this, it suffices to show that every object $\CM\in \Crys^r(X)^\heartsuit$
can be embedded in an \emph{injective} object, i.e., an object $\CI\in \Crys^r(X)^\heartsuit$ such that
$$H^0(\CN)=0 \, \Rightarrow \Hom_{\Crys^r(X)}(\CN,\CI)=0.$$

Let $i:X\hookrightarrow Z$ be a closed embedding, where $Z$ is a smooth classical scheme. 
Choose an embedding $i_{\dr,*}(\CM)\hookrightarrow \CJ$, where $\CJ$ is an injective object
(in the same sense) in $\Crys^r(Z)$; it exists by Step 1.  

\medskip

Since the functor 
$\ind^r_Z$ is t-exact, we obtain that $\oblv^r_Z(\CJ)$ is an injective object of $\QCoh(Z)^\heartsuit$.
This implies that $\CI:=i^{\dagger,r}(\CJ)$ belongs to $\Crys^r(X)^\heartsuit$ and has the 
required property.

\sssec{Step 3}

We note that by \corref{c:left complete}, the category $\Crys^r(X)$ identifies with the
left completion of $\Crys^r(X)^+$. Hence, it is enough to show that the canonical
embedding
$$D\left(\Crys^r(X)^\heartsuit\right)^+\hookrightarrow D\left(\Crys^r(X)^\heartsuit\right)$$
identifies $D\left(\Crys^r(X)^\heartsuit\right)$ with the left completion of 
$D\left(\Crys^r(X)^\heartsuit\right)^+$.

\medskip

For that it suffices to exhibit a generator $\CP$ of $\Crys^r(X)^\heartsuit$ of \emph{bounded Ext dimension}.

\medskip

Consider the object 
$$\CP:=\ind^r_X(\CO_X).$$ 
It has the required property by \propref{p:t and crys sm}(b). 

\qed

\begin{rem}
A standard argument allows us to extend the statement of \propref{p:relation to abelian}
to the case when $X$ is a quasi-compact DG scheme with an affine diagonal.
\end{rem}

\begin{rem} \label{r:fin c d}
Once we identify crystals with D-modules on smooth affine classical schemes,
we will obtain many other properties of $\Crys^r(X)$ on quasi-compact DG schemes:
e.g., the fact that the abelian category $\Crys^r(X)^\heartsuit$ is locally Noetherian\footnote{By this we mean
that $\Crys^r(X)^\heartsuit$ is generated by its compact objects, and a subobject of a compact one
is compact.} and that
$\Crys^r(X)$ has finite cohomological dimension with respect to its t-structure. 
\footnote{By this we mean that there exists $N\in \BN$
such that for $n>N$, $\Hom_{\Crys^r(X)}(\CM_1,\CM_2[n])=0$ for $\CM_1,\CM_2\in \Crys^r(X)^\heartsuit$.}
Note that by \propref{p:cl embed exact}, in order to establish both these properties, it suffices to show them 
for affine smooth classical schemes. 
\end{rem}

\section{Relation to D-modules}

In this section we will relate the monads $\oblv^r_X\circ \ind^r_X$ and 
$\oblv^l_X\circ \ind^l_X$ to the sheaf of differential operators. As a result we 
relate the category $\Crys^r$ over a DG scheme to the 
(derived) category of D-modules. 

\ssec{Crystals via an integral transform}

In this subsection we let $\CX$ be a DG indscheme locally almost of finite type. 

\sssec{}  \label{sss:diff op r}

Recall that for $\CX\in \dgindSch_{\on{laft}}$, the category $\IndCoh(\CX)$ is dualizable
and canonically self-dual, see \cite[Sect. 2.6]{IndSch}.

\medskip

Hence, for $\CX,\CY\in \dgindSch$, the category $\on{Funct}_{\on{cont}}(\IndCoh(\CX),\IndCoh(\CY))$
identifies with
$$\IndCoh(\CX)\otimes \IndCoh(\CY)\simeq \IndCoh(\CX\times \CY).$$
Expilcitly, an object $\CQ\in \IndCoh(\CX\times \CY)$ defines a functor $\sF_\CQ:\IndCoh(\CX)\to \IndCoh(\CY)$
by
\begin{equation} \label{e:functor via kernel}
\CF\mapsto (p_2)^{\IndCoh}_*\circ (\Delta_\CX\times \on{id}_\CY)^!(\CF\boxtimes \CQ),
\end{equation}
where $p_2:\CX\times \CY\to \CY$ is the projection map and $\Delta_\CX$ is 
the diagonal map $\CX\to \CX\times \CX$, 

\medskip

In particular, the endo-functor  $\oblv_\CX^r\circ \ind_\CX^r$ defines an object, denoted 
$$\cD^r_\CX\in \IndCoh(\CX\times\CX).$$ 
We will identify this object.

\sssec{}

Let
$\wh\Delta_\CX$ denote the map 
$$\CX\underset{\CX_\dr}\times \CX\simeq (\CX\times \CX)^\wedge_\CX\to \CX\times \CX.$$

\begin{prop}  \label{p:identify D}
There is a canonical isomorphism in $\IndCoh(\CX\times\CX)$
$$\cD^r_\CX\simeq(\wh\Delta_\CX)_*^{\IndCoh}(\omega_{\CX\underset{\CX_\dr}\times \CX}).$$
\end{prop}

\begin{proof}

We begin with the following general observation. 

\medskip

Suppose that we have a functor $\sF\in \on{Funct}_{\on{cont}}(\IndCoh(\CX),\IndCoh(\CY))$ given by a correspondence, i.e. 
we have a diagram
$$ \xymatrix{ & \CZ\ar[dr]^{q_2}\ar[dl]_{q_{1}} & \\ \CX && \CY }$$
of DG indschemes, and $\sF:=(q_2)_*^{\IndCoh} \circ q_1^!$.  Let
$$i: \CZ \rightarrow \CX\times \CY $$
be the induced product map.

\begin{lem}\label{l:kernel of correspondence}
In the above situation, the functor
\[ (q_2)_*^{\IndCoh} \circ q_1^!: \IndCoh(\CX) \rightarrow \IndCoh(\CY) \]
is given by the kernel $\CQ = i^{\IndCoh}_{*} (\omega_{\CZ})$.
\end{lem}
\begin{proof}
We have a diagram, whose inner square is Cartesian
$$
\CD
\CZ @>{q_1\times \on{id}_{\CZ}}>> \CX\times \CZ @>>>  \CX \\
@V{i}VV    @VV{\on{id}_{X}\times i}V   \\
\CX\times \CY   @>{\Delta_{X}\times \on{id}_{\CY}}>>  \CX\times \CX\times \CY \\
@V{p_2}VV  \\
\CY.
\endCD
$$

For $\CF \in \IndCoh(\CX)$, we have
$$ (q_{2})^{\IndCoh}_{*}\circ q_{1}^{!}(\CF) \simeq (p_{2})_{*}^{\IndCoh}\circ i^{\IndCoh}_{*} 
\circ (q_1\times \on{id}_{\CZ})^{!}(\CF\boxtimes \omega_{Z}) $$

\medskip

By \cite[Proposition 2.9.2]{IndSch},\footnote{Strictly speaking, the base change isomorphism was stated in 
\cite[Proposition 2.9.2]{IndSch} only in the case when the vertical arrow is ind-proper, which translates into $i$
being proper. For the proof of \propref{p:identify D} we will apply it in such a situation.}
$$ i^{\IndCoh}_{*}\circ(q_{1}\times \on{id}_{\CZ})^{!}(\CF\boxtimes \omega_{Z}) \simeq 
(\Delta_{\CX}\times \on{id}_{\CY})^{!} (\CF \boxtimes i^{\IndCoh}_{*} (\omega_{\CZ})) .$$
\end{proof}

We apply this lemma to prove  \propref{p:identify D} as follows:

\medskip

By \propref{p:right ind}(b), we have that the functor $\oblv_\CX^r\circ \ind_\CX^r$ is given by the correspondence
$$ \xymatrix{ & (\CX\times \CX)^\wedge_\CX\ar[dr]^{p_t}\ar[dl]_{p_{s}} & \\ \CX && \CX. }$$
The assertion now follows from \lemref{l:kernel of correspondence}.

\end{proof}

\sssec{}

As a corollary of \propref{p:identify D} we obtain:

\begin{cor}
There exists a canonical isomorphism $\sigma(\cD^r_\CX)\simeq \cD^r_\CX$, where $\sigma$ is the
transpoistion of factors acting on $\CX\times \CX$.
\end{cor}

\ssec{Explicit formulas for other functors}

In this subsection we let $X$ be an eventually coconnective quasi-compact DG scheme almost of finite type.

\sssec{} 

Recall now that the category $\QCoh(X)$ is also compactly generated and self-dual.
Under the identifications
$$\QCoh(X)^\vee\simeq \QCoh(X) \text{ and } \IndCoh(X)^\vee\simeq \IndCoh(X),$$
the dual of the functor $\Upsilon_X$ is the functor $\Psi_X$ of \cite[Sect. 1.1.5]{IndCoh}
(see \cite[Proposition 9.3.3]{IndCoh} for the duality statement). 

\medskip

In particular, for $\bC'\in \StinftyCat_{\on{cont}}$, we have
$$\on{Funct}_{\on{cont}}(\QCoh(X),\bC')\simeq \QCoh(X)\otimes \bC',$$
by a formula similar to \eqref{e:functor via kernel}.

\sssec{}  \label{sss:supp on diag}

Let $\bC$ be any of the categories
$$\QCoh(X\times X)\simeq \QCoh(X)\otimes \QCoh(X),\,\, 
\IndCoh(X\times X)\simeq \IndCoh(X)\otimes \IndCoh(X),$$\
$$\QCoh(X)\otimes \IndCoh(X) \text{ or } \IndCoh(X)\otimes \QCoh(X).$$

\medskip

Then $\bC$ is a module over $\QCoh(X\times X)$, and we define an endo-functor of $\bC$,
denoted
$$\CF\mapsto \CF_{\{X\}}$$
given by tensor product with the object
$$\on{Cone}(\CO_{X\times X}\to j_*\circ j^*(\CO_{X\times X}))[-1],$$
where $j$ is the open embedding $X\times X-X\hookrightarrow X\times X$. 

\medskip

Note that by \cite[Proposition 4.1.7 and Corollary 4.4.3]{IndCoh}, for $\bC=\IndCoh(X\times X)$ this functor identifies with
$$(\wh\Delta_X)^{\IndCoh}_*\circ (\wh\Delta_X)^!,$$
where we recall that $\wh\Delta_X$ denotes the map
$$X\underset{X_\dr}\times X\simeq (X\times X)^\wedge_X\to X\times X.$$

\sssec{}

We claim:

\begin{prop} \label{p:who is who} \hfill

\smallskip

\noindent{\em(a)} The object of $$\QCoh(X)\otimes \QCoh(X)\simeq \on{Funct}_{\on{cont}}(\QCoh(X),\QCoh(X)),$$
corresponding to $\oblv^l_X\circ \ind^l_X$, is canonically identified with 
$$(\Psi_X(\omega_X)\boxtimes \CO_X)_{\{X\}}.$$

\smallskip

\noindent{\em(b)} The object of 
$$\QCoh(X)\otimes \IndCoh(X)\simeq \on{Funct}_{\on{cont}}(\QCoh(X),\IndCoh(X)),$$
corresponding to $\oblv^r_X\circ \Upsilon_{X_\dr} \circ \ind^l_X$, is canonically identified with
$$(\Psi_X(\omega_X)\boxtimes \omega_X)_{\{X\}}.$$

\smallskip

\noindent{\em(c)} The object of 
$$\IndCoh(X)\otimes \QCoh(X)\simeq \on{Funct}_{\on{cont}}(\IndCoh(X),\QCoh(X)),$$
corresponding to $\oblv^l_X\circ (\Upsilon_{X_\dr})^{-1}  \circ \ind^r_X$, is canonically identified with
$$(\omega_X\boxtimes \CO_X)_{\{X\}}.$$

\smallskip

\noindent{\em(d)} The object of $$\QCoh(X)\boxtimes \IndCoh(X)\simeq \on{Funct}_{\on{cont}}(\QCoh(X),\IndCoh(X)),$$
corresponding to $\oblv^r_X\circ {}'\ind^r_X$, is canonically identified with 
$$(\CO_X\boxtimes \omega_X)_{\{X\}}.$$

\smallskip

\noindent{\em(e)} The object of $$\IndCoh(X)\boxtimes \QCoh(X)\simeq \on{Funct}_{\on{cont}}(\IndCoh(X),\QCoh(X)),$$
corresponding to $'\oblv^r_X\circ \ind^r_X$, is canonically identified with 
$$(\omega_X\boxtimes \Psi_X(\omega_X))_{\{X\}}.$$

\smallskip

\noindent{\em(f)} The object of $$\QCoh(X\times X)\simeq \on{Funct}_{\on{cont}}(\QCoh(X),\QCoh(X)),$$
corresponding to $'\oblv^r_X\circ {}'\ind^r_X$, is canonically identified with
$$(\CO_X\boxtimes \Psi_X(\omega_X))_{\{X\}}.$$

\smallskip

\noindent{\em(g)} The object of $$\QCoh(X\times X)\simeq \on{Funct}_{\on{cont}}(\QCoh(X),\QCoh(X)),$$
corresponding to $'\oblv^r_X\circ \Upsilon_{X_\dr} \circ \ind^l_X$, is canonically identified with
$$(\Psi_X(\omega_X)\boxtimes \Psi_X(\omega_X))_{\{X\}}.$$

\smallskip

\noindent{\em(h)} The object of $$\QCoh(X\times X)\simeq \on{Funct}_{\on{cont}}(\QCoh(X),\QCoh(X)),$$
corresponding to $\oblv^l_X\circ (\Upsilon_{X_\dr})^{-1}  \circ {}'\ind^r_X$, is canonically identified with
$$(\CO_X\boxtimes \CO_X)_{\{X\}}.$$

\end{prop}

\begin{proof}

Let $\bC$ and $\bD$ be objects of $\StinftyCat_{\on{cont}}$ with $\bC$ dualizable, so that
$$\on{Funct}_{\on{cont}}(\bC,\bD)\simeq \bC^\vee\otimes \bD.$$
Let $\sF:\bC_1\to \bC$ and $\sG:\bD\to \bD_1$ be continuous functors. Then the resulting functor
$$\on{Funct}_{\on{cont}}(\bC,\bD)\to \on{Funct}_{\on{cont}}(\bC_1,\bD_1)$$
is given by
$$(\sF^\vee\otimes \sG):\bC^\vee\otimes \bD\to \bC^\vee_1\otimes \bD_1.$$

With this in mind, we have:

\medskip

\noindent Points (a) and (c) follow by combining \propref{p:identify D}, \lemref{l:oblv l},  
and the following assertion:

\begin{lem} \label{l:Xi vee of omega}
The unit of the adjunction $\on{Id}_{\QCoh(X)}\to \Xi^\vee_X\circ \Upsilon_X$
defines an isomorphism
$$\CO_X\to \Xi^\vee(\omega_X).$$
\end{lem}

\medskip

\noindent Point (b) follows from \propref{p:identify D} using $\ind^l_X\simeq (\Upsilon_{X_\dr})^{-1}\circ
\ind^r_X\circ \Upsilon_X$. 

\medskip

\noindent Point (d) follows from \propref{p:identify D} using the isomorphism $'\ind^r_X\simeq \ind^r_X\circ \Xi^\vee_X$
and \lemref{l:Xi vee of omega}.

\medskip

\noindent Point (e) follows from \propref{p:identify D}. Point (f) follows from point (d). Point (g) follows from point (b).

\medskip

\noindent Point (h) follows from point (d) using \lemref{l:oblv l} and \lemref{l:Xi vee of omega}.

\end{proof}

\sssec{}

Let $\cD^l_X$,  $\cD^{l\to r'}_X$, $\cD^{r'\to l}_X$ and $\cD^{r'}_X$ denote the objects of 
$$\QCoh(X)\otimes \QCoh(X)\simeq \QCoh(X\times X)$$
corresponding to the functors
$$\oblv^l_X\circ \ind^l_X,\,\, {}'\oblv^r_X\circ \Upsilon_{X_\dr}\circ \ind^l_X,\,\,
\oblv^l_X \circ (\Upsilon_{X_\dr})^{-1}  \circ {}'\ind^r_X \text{ and } '{}\oblv^r_X\circ {}'\ind^r_X,$$
respectively. 

\medskip

We have:

\begin{prop}  \label{p:other D}
The objects $\cD^r_X$, $\cD^l_X$, $\cD^{l\to r'}_X$, $\cD^{r'\to l}_X$ and $\cD_X^{r'}$ are related by

\medskip

\noindent{\em(i)} $\cD^l_X\simeq (\Psi_X\boxtimes \Xi^\vee_X)(\cD^r_X)\in \QCoh(X\times X)$;

\medskip

\noindent{\em(ii)} $(\Psi_X\boxtimes \on{Id}_{\IndCoh(X)})(\cD^r_X)\simeq (\on{Id}_{\QCoh(X)}\boxtimes \Upsilon_X)(\cD^l_X) 
\in \QCoh(X)\otimes \IndCoh(X)$;

\medskip

\noindent{\em(iii)} $\cD^{l\to r'}_X\simeq  (\Psi_X\boxtimes \Psi_X)(\cD^r_X)\in \QCoh(X\times X)$;

\medskip

\noindent{\em(iii')} $\cD^{l\to r'}_X\simeq 
(\on{Id}_{\QCoh(X)}\boxtimes \Psi_X\circ \Upsilon_X)(\cD^l_X)\simeq 
(\CO_X\boxtimes \Psi_X(\omega_X))\underset{\CO_{X\times X}}\otimes \cD^l_X \in \QCoh(X\times X)$;

\medskip

\noindent{\em(iv)} $\cD^{r'\to l}_X\simeq (\Xi^\vee_X\boxtimes \Xi^\vee_X)(\cD^r_X)$.

\medskip

\noindent{\em(v)} $\cD^{r'}\simeq (\Xi^\vee_X\boxtimes \Psi_X)(\cD^r_X)\in \QCoh(X\times X)$.

\end{prop}

\begin{proof}

Point (i) follows from \lemref{l:oblv l}. 

\medskip

\noindent Point (ii) follows from the (tautological) isomorphism of functors
$$\oblv^r_X\circ \ind^r_X\circ \Upsilon_X\simeq \oblv^r_X\circ \Upsilon_{X_\dr}\circ \ind^l_X\simeq
\Upsilon_X\circ \oblv^l_X\circ \ind^l_X.$$

\noindent Point (iii) is tautological.

\medskip

\noindent Point (iii') follows by combining points (ii) and (iii).

\medskip

\noindent Point (iv) follows from \lemref{l:oblv l}. 

\medskip

\noindent Point (v) is tautological.

\end{proof}

\ssec{Behavior with respect to the t-structure}

We continue to assume that $X$ is a quasi-compact DG scheme almost of finite type.

\sssec{}

We note:
\begin{lem}
The object $\cD^r_X$ is bounded below, i.e., belongs to $\IndCoh(X\times X)^+$. 
\end{lem}

\begin{proof}
Follows from \propref{p:identify D}, using the fact that $\omega_X\in \IndCoh(X)^+$, and the fact that the functor
$$\CF\mapsto \CF_{\{X\}},\quad \IndCoh(X\times X)\to \IndCoh(X\times X)$$
is right t-exact.
\end{proof}

\sssec{}

Assume now that $X$ is eventually coconnective. We claim:

\begin{prop} The objects $\cD^l_X$, $\cD^{l\to r'}_X$, $\cD^{r'\to l}_X$ and $\cD^{r'}_X$ of $\QCoh(X\times X)$
are all eventually coconnective, i.e., belong to $\QCoh(X\times X)^+$.
\end{prop}

\begin{proof}
Follows from \propref{p:who is who}, using the fact that $\Psi_X(\omega_X),\CO_X\in \QCoh(X)^+$
and the fact that the functor
$$\CF\mapsto \CF_{\{X\}},\quad \QCoh(X\times X)\to \QCoh(X\times X)$$
is right t-exact.
\end{proof}

\sssec{}

Finally, let us assume that $X$ is a smooth classical scheme. We claim:

\begin{prop} \label{p:diff in heart}
The object $\cD^l_X\in \QCoh(X\times X)$ lies in the heart of the t-structure.
\end{prop}

\begin{proof}
The assertion is Zariski-local, hence, we can assume that $X$ is affine. It is sufficient
to show that
$$(p_2)_*(\cD^l_X)\in \QCoh(X)$$
lies in the heart of the t-structure. We have, 
$$(p_2)_*(\cD^l_X)\simeq \oblv^l_X\circ \ind^l_X(\CO_X)\simeq \Xi^\vee_X \circ (\oblv^r_X\circ \ind^r_X) \circ \Upsilon_X(\CO_X).$$
Now, the functor $\oblv^r_X\circ \ind^r_X$ is t-exact (see \propref{p:t and crys sm}), the functor 
$\Upsilon_X$ is an equivalence that shifts degrees by $[n]$, and $\Xi^\vee_X$ is the inverse of 
$\Upsilon_X$.
\end{proof}

\ssec{Relation to the sheaf of differential operators}  \label{ss:rel to diff op}

In this subsection we shall take $X$ to be a smooth classical scheme.
We are going to identify $\cD^l_X$ with the object of $\QCoh(X\times X)$ 
underlying the classical sheaf of differential operators $\on{Diff}_X$. 

\sssec{}

For any $\CQ\in \QCoh(X\times X)^\heartsuit$, which is set-theoretically supported on the diagonal, and
$\CF_1,\CF_2\in \QCoh(X)^\heartsuit$, a datum of a map
$$(p_2)_*(p_1^*(\CF_1)\otimes \CQ)\to \CF_2$$
is equivalent to that of a map 
$$\CQ\to \on{Diff}_X(\CF_1,\CF_2).$$

\medskip

Furthermore, this assignment is compatible with the monoidal structure on $\QCoh(X\times X)^\heartsuit$,
given by convolution and composition of differential operators.

\sssec{}

Taking $\CQ=\cD^l_X$ and $\CF_1=\CF_2=\CO_X$, from the action of the monad $\oblv^l_X\circ \ind^l_X$ 
on $\CO_X$, we obtain the desired map
\begin{equation} \label{e:left diff}
\cD^l_X\to \on{Diff}_X,
\end{equation}
compatible with the algebra structure.

\medskip

We claim:

\begin{lem} \label{l:left diff}
The map \eqref{e:left diff} is an isomorphism of algebras. 
\end{lem}

\begin{proof}
It suffices to show that \eqref{e:left diff} is an isomorphism at the level of the underlying
objects of $\QCoh(X\times X)$. The latter follows, e.g., from the description of
$\cD^l_X$ as a quasi-coherent sheaf given by \propref{p:who is who}. 
\end{proof}

\ssec{Relation between crystals and D-modules}  \label{ss:crys and Dmod}

Let $X$ be a classical scheme of finite type.  We will show that the category $\Crys^{r}(X)$ can be canonically identified with the 
(derived) category $\Dmod^{r}(X)$ of right D-modules on $X$.

\begin{rem}
The category $\Dmod^{r}(X)$ satisfies Zariski descent.  Therefore, in what follows, by \propref{p:descent for crystals}, it will suffice to establish a canonical equivalence for affine schemes.
\end{rem}

\sssec{}

Let $Z$ be a smooth classical affine scheme, and let $i:X\hookrightarrow Z$ be a closed embedding. By the 
classical Kashiwara's lemma and \propref{p:Kashiwara}, in order to construct an equivalence 
$$\Crys^r(X)\simeq \Dmod^r(X),$$
it suffices to do so for $Z$. 

\medskip

Hence, we can assume that $X$ itself is a smooth classical affine scheme.
 We shall construct the equivalence
in question together with the commutative diagram of functors
$$
\CD
\Crys^r(X)   @>>>  \Dmod^r(X) \\
@V{\oblv^r_X}VV       @VVV  \\
\IndCoh(X)   @>{\Psi_X}>>  \QCoh(X),
\endCD
$$
where the right vertical arrow is the natural forgetful functor, and the functor $\Psi_X$
is the equivalence of \cite[Lemma 1.1.6]{IndCoh}.

\medskip

By \propref{p:left to right}, constructing an equivalence $\Crys^r(X)\simeq \Dmod^r(X)$ as above is 
the same as constructing an equivalence between {\em left} crystals and {\em left} D-modules together 
with the commutative diagram of functors

\begin{equation}\label{e:crys vs Dmod on smooth}
\begin{split}
\xymatrix{ \Crys^{l}(X) \ar[rr]\ar[dr]_{\oblv^{l}_{X}} && \Dmod^{l}(X)\ar[dl] \\ & \QCoh(X) & }.
\end{split}
\end{equation}

\sssec{}   \label{sss:with Dmod}

By Propositions \ref{p:relation to abelian}, \ref{p:left to right} and \ref{p:left to right t-struct}, 
the category $\Crys^l(X)$ identifies with the derived
category of the heart of its t-structure. The category $\Dmod^l(X)$ is by definition the derived category
of $\Dmod^l(X)^\heartsuit$. Moreover, the vertical arrows in Diagram \eqref{e:crys vs Dmod on smooth}
are t-exact.

\medskip

Hence, it suffices to construct the desired equivalence at the level of the corresponding abelian categories
\begin{equation} \label{e:crys vs Dmod on smooth abelian}
\CD
\Crys^l(X)^\heartsuit   @>>>  \Dmod^l(X)^\heartsuit \\
@V{\oblv^l_X}VV       @VVV  \\
\QCoh(X)^\heartsuit  @>{\on{Id}}>>  \QCoh(X)^\heartsuit
\endCD
\end{equation}

\sssec{}  \label{sss:Groth}

The latter is a classical calculation, due to Grothendieck:

\medskip

Namely, one interprets $\Crys^{l}(X)^{\heartsuit}$
as the heart of the category of quasi-coherent sheaves on the truncated simplicial object
$$\xymatrix{(X\times X\times X)^{\wedge}_{X}\ar@<2.4ex>[r]^-{p_{12}} \ar[r]^-{p_{13}} \ar@<-2.4ex>[r]^-{p_{23}}& 
(X \times X)^{\wedge}_{X}\ar@<.7ex>[r]^-{p_{1}} \ar@<-.7ex>[r]_-{p_{2}} & X}.$$
I.e., explicitly, an object of $\Crys^{l}(X)^{\heartsuit}$ is a quasi-coherent sheaf $\CF \in \QCoh(X)^{\heartsuit}$ 
together with an isomorphism
$$ \phi: p_{2}^{*}(\CF) \overset{\sim}{\rightarrow} p_{1}^{*}(\CF) $$
which restricts to the identity on the diagonal and satisfies the cocycle condition
$$p_{13}^{*}(\phi) = p_{12}^{*}(\phi)\circ p_{23}^{*}(\phi).$$

\medskip

Below we give an alternative approach to establishing the equivalence in
\eqref{e:crys vs Dmod on smooth abelian}.

\sssec{}
The abelian categories $\Crys^l(X)^\heartsuit$ and $\Dmod^l(X)^\heartsuit$ are given as modules over the monads
$\sM_{\Crys^l(X)}$ and $\sM_{\Dmod(X)}$, respectively, acting on the category $\QCoh(X)^\heartsuit$.  

\medskip

By definition, $\sM_{\Dmod(X)}$ is given by the algebra of differential operators $\on{Diff}_X$. The monad
$\sM_{\Crys^l(X)}$ is given by $\oblv^l_X\circ \ind^l_X$. Now, the desired equivalence follows from
\lemref{l:left diff}.

\begin{rem}
It follows from the construction that the equivalence
$$ \Crys^{l}(X) \rightarrow \Dmod^{l}(X) $$
is compatible with pull-back for maps $f:Y\rightarrow X$ between smooth classical schemes.
\end{rem}

\section{Twistings}

In this section, we do {\em not} assume that the prestacks and DG schemes that we consider
are locally almost of finite type. We will reinstate this assumption in \secref{ss:tw on indsch}.

\ssec{Gerbes}

\sssec{}   \label{sss:gerbes}

Let $\on{pt}/\BG_m$ be the classifying stack of the group $\BG_m$. In other words, $\on{pt}/\BG_m$ is
the algebraic stack that represents the functor which assigns to an affine DG scheme
$S$, the $\infty$-groupoid of line bundles on $S$. 

\medskip

In fact, since $\BG_m$ is an abelian
group, the stack $\on{pt}/\BG_m$ has a natural abelian group structure. The multiplication map
on $\on{pt}/\BG_m$ represents tensor product of line bundles. This structure upgrades $\on{pt}/\BG_m$
to a functor from affine DG schemes to $\infty$-Picard groupoids, i.e. connective
spectra.

\medskip

For our purposes, a $\BG_m$-gerbe will be a presheaf $\CG$ of 
$\on{pt}/\BG_m$-torsors, which satisfies any of the following three (non-equivalent)
conditions:

\medskip

\noindent(i)  $\CG$ is locally non-empty in the \'etale topology \footnote{By To\"en's theorem, this is equivalent
to local non-emptyness in the fppf topology.}.

\medskip

\noindent(ii) $\CG$ is locally non-empty in the Zariski topology.

\medskip

\noindent(iii) $\CG$ is globally non-empty.

\medskip

Specifically, let
$B^{\on{naive}}(\on{pt}/\BG_m)$ be the classifying prestack of $\on{pt}/\BG_m$. It is given by the
geometric realization of the simplicial prestack
$$ B^{\on{naive}}(\on{pt}/\BG_m) := \left|
\xymatrix{\cdots\ \on{pt}/\BG_m \times \on{pt}/\BG_m \ar@<-1.4ex>[r]\ar[r]\ar@<1.4ex>[r] & \on{pt}/\BG_m
\ar@<.7ex>[r] \ar@<-.7ex>[r] & \on{pt}} \right| . $$
Let $B^{\on{Zar}}(\on{pt}/\BG_m)$ (resp. $B^{\on{et}}(\on{pt}/\BG_m)$) 
be the Zariski (resp. \'etale) sheafification of the prestack $B^{\on{naive}}(\on{pt}/\BG_m)$.  

\medskip

The prestacks $B^{\on{et}}(\on{pt}/\BG_m)$, $B^{\on{Zar}}(\on{pt}/\BG_m)$ and $B^{\on{naive}}(\on{pt}/\BG_m)$ represent 
$\BG_m$-gerbes satisfying the above conditions (i), (ii), and (iii) respectively.

\medskip

Let $(\Ge_{\BG_m})_{\affdgSch}$  be the functor
$$(\affdgSch)^{\on{op}}\to \inftypic$$
that associates to an affine DG scheme $S$, the groupoid of $\BG_m$-gerbes, where we consider any of the three notions of gerbe 
defined above.  

\begin{rem}
While these three versions do not give equivalent notions of $\BG_m$-gerbe, we will see shortly that they do lead to 
the same definition of twisting, since the relevant gerbes will be those whose restrictions to $^{cl,red}S$ are trivialized.
\end{rem}

\sssec{}

We define the functor
$$(\Ge_{\BG_m})_{\on{PreStk}}:(\on{PreStk})^{\on{op}}\to \inftypic$$
as the right Kan extension of $(\Ge_{\BG_m})_{\affdgSch}$ along
$$(\affdgSch)^{\on{op}}\hookrightarrow (\on{PreStk})^{\on{op}}.$$

I.e., for $\CY\in \on{PreStk}$,
$$\Ge_{\BG_m}(\CY):=\underset{S\in (\affdgSch_{/\CY})^{\on{op}}}{lim}\, \Ge_{\BG_m}(S).$$

\medskip

Equivalently,
$$\Ge_{\BG_m}(\CY)=\Maps_{\on{PreStk}}(\CY,B^?(\on{pt}/\BG_m))$$
for $?=\on{naive},\,\, \on{Zar}$ or $\on{et}$. 

\medskip

Thus, informally, a $\BG_m$-gerbe on $\CY$ is an assignment of a $\BG_m$-gerbe
on every $S\in \affdgSch$ mapping to $\CY$, functorial in $S$.

\medskip

For a subcategory $\bC\subset \on{PreStk}$, let $(\Ge_{\BG_m})_\bC$ denote the restriction
$(\Ge_{\BG_m})_{\on{PreStk}}|_\bC$.

\ssec{The notion of twisting}

\sssec{}

Let $\CY$ be a prestack. The Picard groupoid of twistings on $\CY$ defined as 
$$\Tw(\CY):=\on{\ker}\left(p_{\dr,\CY}^*:\Ge_{\BG_m}(\CY_{\dr})\to \Ge_{\BG_m}(\CY)\right),$$
where $\Ge_{\BG_m}$ is understood in any of the three versions: $\on{naive}$, $\on{Zar}$ or
$\on{et}$. As we shall see shortly (see \secref{ss:twist reformulation}), all three versions are equivalent.

\medskip

Informally, a twisting $T$ on $\CY$ is the following data: for every $S\in \affdgSch$
equipped with a map $^{cl,red}S\to \CY$ we specify an object 
$\CG_S\in \Ge_{\BG_m}(S)$, which behaves compatibly under the maps $S_1\to S_2$.
Additionally, for every extension of the above map to a map $S\to \CY$
we specify a trivialization of $\CG_S$, which also behaves functorially with
respect to maps $S_1\to S_2$.

\begin{rem} When we write $\on{ker}(\CA_1\to \CA_2)$, where $\CA_1\to \CA_2$ is a map
in $\inftypic$, we mean 
$$\CA_1\underset{\CA_2}\times \{*\},$$
where the fiber product is taken in $\inftypic$. I.e., this the same as the connective truncation
of the fiber product taken in the category of all (i.e., not necessarily connective) spectra.
\end{rem}

\sssec{Example}  \label{sss:twisting by line bundle}

Let $\CL$ be a line bundle on $\CY$. We define a twisting $T(\CL)$ on $\CY$ as follows:
it assigns to every $S\in \affdgSch$ with a map $^{cl,red}S\to \CY$ the trivial $\BG_m$-gerbe.
For a map $S\to \CY$, we trivialize the above gerbe by multiplying
the tautological trivialization by $\CL$.

\sssec{}

It is clear that twistings form a functor
$$\Tw_{\on{PreStk}}:\on{PreStk}^{\on{op}}\to \inftypic.$$

For a morphism $f:\CY_1\to \CY_2$ we
let $f^*$ denote the corresponding functor
$$\Tw(\CY_2)\to \Tw(\CY_1).$$

If $\bC$ is a subcategory of $\on{PreStk}$  (e.g., $\bC=\affdgSch$ or $\dgSch$), we let 
$\Tw_\bC$ denote the restriction of $\Tw_{\on{PreStk}}$ to $\bC^{\on{op}}$.

\sssec{}

By construction, the functor $\Tw_{\on{PreStk}}$ takes colimits in $\on{PreStk}$
to limits in $\inftypic$. Hence, from \corref{c:dr from Sch}, we obtain:

\begin{lem}  \label{c:twistings as RKE}
The functor $\Tw_{\on{PreStk}}$ maps isomorphically to the right Kan extension of
$\Tw_\bC$ along $$\bC^{\on{op}}\hookrightarrow \on{PreStk}^{\on{op}}$$
for $\bC$ being one of the categories
$$\affdgSch,\,\, \dgSch_{\on{qs-qc}},\,\, \dgSch.$$
\end{lem}

Concretely, this lemma says that the map
$$\Tw(\CY)\to \underset{S\in (\affdgSch_{/\CY})^{\on{op}}}{lim}\, \Tw(S)$$
is an isomorphism (and that $\affdgSch$ can be replaced by $\dgSch_{\on{qs-qc}}$ or
$\dgSch$.)

\medskip

Informally, this means that to specify a twisting on a prestack $\CY$
is equivalent to specifying a compatible family of twistings on affine DG schemes $S$
mapping to $\CY$.

\ssec{Variant: other structure groups}\label{ss:other groups}

\sssec{} 
Let $S$ be an affine DG scheme. Consider the Picard groupoid
$$\Ge_{\BG_m}^{/red}(S):=\on{ker}\left(\Ge_{\BG_m}(S)\to \Ge_{\BG_m}({}^{cl,red}S)\right).$$

\medskip

Let $(\Ge_{\BG_m}^{/red})_{\affdgSch}$ denote the resulting functor
$$(\dgSch)^{\on{op}}\to \inftypic.$$

\sssec{}

By definition, we can think of $\Ge_{\BG_m}^{/red}(S)$ as gerbes (in any of the three versions of \secref{sss:gerbes})
with respect to the presheaf of abelian groups
$$(\CO^{\times})_S^{/red}:\on{ker}(\CO^\times_S\to \CO^\times_{^{cl,red}S}).$$

\sssec{}

In addition to $\BG_m$-gerbes, we can also consider $\BG_a$-gerbes.  We have the functor
$$ (\Ge_{\BG_a})_{\affdgSch}: (\affdgSch)^{\on{op}}\to \inftypic$$
which assigns to an affine DG scheme $S$ the groupoid of $\BG_a$-gerbes on $S$.  

\medskip

Note that unlike the case of $\BG_m$-gerbes, the three notions of gerbes discussed in \secref{sss:gerbes} are equivalent for 
$\BG_a$-gerbes.  This is due to the fact that for an affine DG scheme $S$, 
$$H^2_{\on{Zar}}(S,\BG_a)=H^2_{\on{et}}(S,\BG_a)=0.$$

\medskip

Thus, we have that $(\Ge_{\BG_a})_{\affdgSch}$ is represented by $B^2(\BG_a)$, which is the geometric 
realization of the corresponding simplicial prestack.

\sssec{}

By definition, for an affine DG scheme $S$, we have

$$ \Ge_{\BG_a}(S) = B^2(\Maps(S,\BG_a))\simeq B^2(\Gamma(S,\CO_S)).$$
In particular, viewed as a connective spectrum, $\Ge_{\BG_a}(S)$ has a natural structure of a module over the ground field $k$.  
This upgrades $(\Ge_{\BG_a})_{\affdgSch}$ to a functor
$$(\dgSch)^{\on{op}}\to \inftypic_k,$$
where $\inftypic_k$ denotes the category of $k$-modules
in connective spectra. Note that by the Dold-Kan correspondence, we have
$$\inftypic_k\simeq \Vect^{\leq 0}.$$

\medskip

We define the functor
$$(\Ge_{\BG_a})_{\on{PreStk}}:\on{PreStk}^{\on{op}}\to \inftypic_k$$
as the right Kan extension of the functor $(\Ge_{\BG_a})_{\affdgSch}$ along
$$(\affdgSch)^{\on{op}}\hookrightarrow \on{PreStk}^{\on{op}}.$$

\sssec{}

As with $\BG_m$-gerbes, we can consider the Picard groupoid
$$\Ge_{\BG_a}^{/red}(S):=\on{ker}\left(\Ge_{\BG_a}(S)\to \Ge_{\BG_a}({}^{cl,red}S)\right),$$
and let $(\Ge_{\BG_a}^{/red})_{\affdgSch}$ denote the resulting functor
$$(\dgSch)^{\on{op}}\to \inftypic_k.$$

\medskip

By definition, for an affine DG scheme $S$, $\Ge_{\BG_a}^{/red}(S)$ is given by gerbes for the presheaf of connective spectra
$$\CO_S^{/red}:=\on{ker}(\CO_S\to \CO_{^{cl,red}S}).$$
Explicitly, 
$$\Ge_{\BG_a}^{/red}(S) \simeq B^2(\Gamma(S,\CO_S^{/red})).$$

\sssec{} \label{sss:exp}
Recall from \cite[Sect. 6.8.8]{IndSch} that the exponential map defines an isomorphism
$$\on{exp}:\CO_S^{/red}\to (\CO^{\times})_S^{/red}.$$

Hence, we obtain:

\begin{cor} \label{c:gerbes as B2}
The exponential map defines an isomorphism of functors
\begin{equation}\label{e:gerbes as B2}
\on{exp}: (\Ge_{\BG_a}^{/red})_{\affdgSch} \to (\Ge_{\BG_m}^{/red})_{\affdgSch}
\end{equation}
for any of the three versions \emph{(}$\on{naive}$, $\on{Zar}$ or $\on{et}$\emph{)} of $(\Ge_{\BG_m}^{/red})_{\affdgSch}$.
\end{cor}

Thus, if we realize $\Ge_{\BG_m}^{/red}(S)$ as gerbes in the \'etale or Zariski topology,
this category has trivial $\pi_0$ and $\pi_1$. In other words, any such gerbe on an affine DG scheme is globally
non-empty, and any automorphism is non-canonically isomorphic to
identity.  

\sssec{}

The isomorphism \eqref{e:gerbes as B2}
endows $\Ge_{\BG_m}^{/red}(S)$, viewed as a connective spectrum,
with a structure of module over the ground field $k$.
This upgrades $(\Ge_{\BG_m}^{/red})_{\affdgSch}$ to a functor
$$(\dgSch)^{\on{op}}\to \inftypic_k.$$

\medskip

We define the functor
$$(\Ge_{\BG_m}^{/red})_{\on{PreStk}}:\on{PreStk}^{\on{op}}\to \inftypic_k$$
as the right Kan extension of the functor $(\Ge_{\BG_m}^{/red})_{\affdgSch}$ along
$$(\affdgSch)^{\on{op}}\hookrightarrow \on{PreStk}^{\on{op}}.$$

\sssec{}

By definition, for $\CY\in \on{PreStk}$
$$\Ge_{\BG_m}^{/red}(\CY):=\underset{S\in (\affdgSch_{/\CY})^{\on{op}}}{lim}\, \Ge_{\BG_m}^{/red}(S).$$

Informally, for $\CY\in \on{PreStk}$, an object $\CG\in \Ge_{\BG_m}^{/red}(\CY)$
is an assignment for every $S\in \affdgSch_{/\CY}$ 
of an object $\CG_S\in \Ge_{\BG_m}^{/red}(S)$, and for every $S'\to S$ of an isomorphism
$$f^*(\CG_S)\simeq \CG_{S'}.$$

The following results from the definitions:

\begin{lem} \label{l:mod red via red}
For $\CY\in \on{PreStk}$, the natural map
$$\Ge_{\BG_m}^{/red}(\CY)\to \on{ker}\left(\Ge_{\BG_m}(\CY)\to \Ge_{\BG_m}({}^{cl,red}\CY)\right)$$
is an isomorphism, where 
$$^{cl,red}\CY:=\on{LKE}_{({}^{red}\!\affSch)^{\on{op}}\hookrightarrow (\on{PreStk})^{\on{op}}}
(\CY|_{{}^{red}\!\affSch}).$$
\end{lem}

\ssec{Twistings: reformulations}\label{ss:twist reformulation}

We are going to show that the notion of twisting can be formulated in terms of
$$ (\Ge_{\BG_m}^{/red})_{\on{PreStk}}, (\Ge_{\BG_a})_{\on{PreStk}} \mbox{ or } (\Ge_{\BG_a}^{/red})_{\on{PreStk}},$$
instead of
$(\Ge_{\BG_m})|_{\on{PreStk}}$.

\sssec{}

Consider the functors 
$$\Tw^{/red}, \Tw_a, \Tw_a^{/red}:\on{PreStk}^{\on{op}}\to \inftypic$$ given by
$$\Tw^{/red}(\CY):=\on{\ker}\left(p_{\dr,\CY}^*:\Ge_{\BG_m}^{/red}(\CY_{\dr})\to \Ge_{\BG_m}^{/red}(\CY)\right),$$
$$\Tw_a(\CY):=\on{\ker}\left(p_{\dr,\CY}^*:\Ge_{\BG_a}(\CY_{\dr})\to \Ge_{\BG_a}(\CY)\right)$$
and
$$\Tw_a^{/red}(\CY):=\on{\ker}\left(p_{\dr,\CY}^*:\Ge_{\BG_a}^{/red}(\CY_{\dr})\to \Ge_{\BG_a}^{/red}(\CY)\right).$$

\medskip

We have the following diagram of functors given by the exponential map and the evident forgetful functors.
\begin{equation} \label{e:var to usual}
\begin{split}
\xymatrix{\Tw_a^{/red} \ar[r]^{exp}\ar[d] & \Tw^{/red}\ar[d] \\ \Tw_a & \Tw }
\end{split}
\end{equation}

\begin{prop} \label{p:var to usual}
The functors in \eqref{e:var to usual} are equivalences.
\end{prop}

\begin{proof}
The functor given by the exponential map is an equivalence by \secref{sss:exp}. 

\medskip

Let us show that the right vertical map in \eqref{e:var to usual} is an equivalence. This
is in fact tautological:

\medskip

Both functors are right Kan extensions under
$$(\affdgSch)^{\on{op}}\to (\on{PreStk})^{\on{op}},$$
so it is enough to show that the map in question is an isomorphism when evaluated on objects
$S\in \affdgSch$.

\medskip

We have:
\begin{multline*}
\Tw^{/red}(S)=\Ge^{/red}_{\BG_m}(S_\dr)\underset{\Ge^{/red} _{\BG_m}(S)}\times\{*\}\overset{\text{\lemref{l:mod red via red}}}\simeq \\
\simeq \on{ker}\left(\Ge_{\BG_m}(S_\dr)\to \Ge_{\BG_m}({}^{cl,red}(S_\dr))\right)
\underset{\on{ker}\left(\Ge_{\BG_m}(S)\to \Ge_{\BG_m}({}^{cl,red}S)\right)}\times \{*\}=\\
=\on{ker}\left(\Ge_{\BG_m}(S_\dr)\to \Ge_{\BG_m}({}^{cl,red}S)\right)
\underset{\on{ker}\left(\Ge_{\BG_m}(S)\to \Ge_{\BG_m}({}^{cl,red}S)\right)}\times \{*\}\simeq \\
\simeq 
\Ge_{\BG_m}(S_\dr) \underset{\Ge_{\BG_m}(S)}\times \{*\}=\Tw(S).
\end{multline*} 

The fact that the left vertical arrow in \eqref{e:var to usual} is an equivalence is proved similarly.

\end{proof}

\sssec{}

As a consequence of \propref{p:var to usual}, we obtain:

\begin{cor}
The notions of twisting in all three versions: $\on{naive}$, $\on{Zar}$ and $\on{et}$ are equivalent.
\end{cor}

In addition:

\begin{cor}
The functor $\Tw:(\on{PreStk})^{\on{op}}\to \inftygroup$ canonically upgrades to a functor
$$(\on{PreStk})^{\on{op}}\to \inftypic_k.$$
\end{cor}

\sssec{Example}  \label{e:twist by line bundle}

We can use the natural $k$-module structure on $\Tw$ to produce additional examples of twistings. Let $\CL$ be a line
bundle on $\CY$, and let $T(\CL)$ be the twisting of \secref{sss:twisting by line bundle}.
Now, for $a\in k$, the $k$-module structure on $\Tw(\CY)$ gives us
a new twisting $T(\CL^{\otimes a})$.

\begin{rem} \label{r:non loc triv}
Note that it is \emph{not} true that any twisting $T$ on an affine DG scheme $X$ is trivial, even locally in the Zariski
or \'etale topology. It is true that for any $S\in \affdgSch$ with a map $S\to X_\dr$, the corresponding $\BG_m$-gerbe
on $S$ can be non-canonically trivialized; but such a trivialization can not necessarily be made compatible
for the different choices of $S$. 

\medskip

An example of such a gerbe for a smooth classical $X$ can be given by a choice 
of a closed $2$-form (see \secref{sss:TDO}) which is not \'etale-locally exact. 

\medskip

Note, however, that the gerbes described in Example \ref{e:twist by line bundle} \emph{are}
Zariski-locally trivial, because of the corresponding property of line bundles.

\end{rem}

\sssec{Convergence}

We now claim:

\begin{prop} \label{p:twistings convergent}
The functor $\Tw:(\affdgSch)^{\on{op}}\to \inftypic$ is \emph{convergent}.\footnote{See \secref{sss:convergence},
where the notion of convergence is recalled.}
\end{prop}

\begin{proof}

We will show that the functor $\Tw^{/red}$ is convergent. For this, it is enough to show that
the functors
$$S\mapsto \Ge^{/red}_{\BG_m}(S_\dr) \text{ and } \Ge^{/red}_{\BG_m}(S)$$
are convergent. 

\medskip

The convergence of $\Ge^{/red}_{\BG_m}((-)_\dr)$ is obvious, as this functor only depends on the
underlying reduced classical scheme.  Thus, it remains to prove the convergence of $\Ge^{/red}_{\BG_m}(-)$.


\medskip

We have:
$$\Ge^{/red}_{\BG_m}(S)=\Ge_{\BG_m}(S)\underset{\Ge_{\BG_m}({}^{cl,red}S)}\times \{*\}.$$
Hence, it is sufficient to show that the functor $\Ge_{\BG_m}(-)$ is convergent. The latter follows
from the fact that 
$$\Ge_{\BG_m}(-)=B^2(\Maps(S,\BG_m)),$$
while $\BG_m$ is convergent, being a DG scheme.

\end{proof}

We can reformulate \propref{p:twistings convergent} tautologically as follows:

\begin{cor} \label{c:twistings convergent}
The functor $\Tw_{\on{PreStk}}$ maps isomorphically to the right Kan extension of
$$\Tw_{^{<\infty}\!\affdgSch}:=\Tw_{\affdgSch}|_{^{<\infty}\!\affdgSch}$$
along 
$$({}^{<\infty}\!\affdgSch)^{\on{op}}\hookrightarrow (\affdgSch)^{\on{op}} \hookrightarrow (\on{PreStk})^{\on{op}}.$$
\end{cor}

\begin{rem} \label{r:gerbes conv}
We can use \propref{p:twistings convergent} to show that the functor $\Ge_{\BG_m}$ 
is also convergent (in any of the three versions).
\end{rem}

\sssec{Twistings in the locally almost of finite type case}

\corref{c:twistings convergent} implies that we ``do not need to know"
about DG schemes that are not locally almost of finite type in order to know what twistings
on $\CY\in \on{PreStk}$ if $\CY$ is locally almost of finite type. 

\begin{cor}  \label{c:twistings laft} \hfill

\smallskip

\noindent{\em(a)}
For $\CY\in \on{PreStk}_{\on{laft}}$, the naturally defined map
$$\Tw(\CY)\to \underset{S\in (({}^{<\infty}\!\affdgSch_{\on{aft}})_{/\CY})^{\on{op}}}{lim}\, \Tw(S)$$
is an equivalence.

\medskip

\noindent{\em(b)}
The functor $\Tw_{\on{PreStk}_{\on{laft}}}$ 
maps isomorphically to the right Kan extension of $\Tw_{^{<\infty}\!\affdgSch_{\on{aft}}}$ 
along the inlcusions
$$({}^{<\infty}\!\affdgSch_{\on{aft}})^{\on{op}}\hookrightarrow (\affdgSch_{\on{aft}})^{\on{op}}\hookrightarrow (\on{PreStk}_{\on{laft}})^{\on{op}}.$$

\end{cor}

\begin{proof}
This is true for $\Tw$ replaced by any convergent prestack $(\affdgSch)^{\on{op}}\to \inftygroup$.
\end{proof}

\begin{rem}
It follows from Remark \ref{r:gerbes conv} that the functor $\Ge_{\BG_m}$ (in any of the three versions), 
viewed as a presheaf, belongs to $\on{PreStk}_{\on{laft}}$. Indeed, this is evident in the $\on{naive}$ 
version, since $\on{pt}/\BG_m$ belongs to $\on{PreStk}_{\on{laft}}$. For the $\on{Zar}$ and $\on{et}$
versions, this follows from \cite[Corollary 2.5.7]{Stacks} that says that the condition of being
locally of finite type in the context of $n$-connective prestacks survives sheafification, once we restrict
ourselves to \emph{truncated} prestacks.
\end{rem}

\ssec{Identification of the Picard groupoid of twistings}  

We can use the description of twistings in terms of $\BG_a$-gerbes to give a cohomological description of the groupoid of twistings.

\sssec{De Rham cohomology}
Let $\CY$ be a prestack.  Recall that the coherent cohomology of $\CY$ is defined as
$$ H(\CY) := \Gamma(\CY, \CO_\CY) = \CMaps_{\QCoh(\CY)}(\CO_\CY,\CO_\CY) .$$
We define the de Rham cohomology of $\CY$ to be the coherent cohomology of $\CY_{\dr}$; i.e.,
$$ H_{\dr}(\CY) := H(\CY_{\dr}) = \CMaps_{\QCoh(\CY_{\dr})}(\CO_{\CY_{\dr}}, \CO_{\CY_{\dr}}) .$$
Note that since $\QCoh(\CY_{\dr})$ is a stable $\infty$-category, the ${\mathcal Maps}$ above gives a (not necessarily connective) 
spectrum.

\medskip

Let $X$ be a smooth classical scheme.  In this case, by \secref{ss:crys and Dmod}, we have
$$ H_{\dr}(X) = {\mathcal Maps}_{\Dmod^l(X)}(\CO_X,\CO_X) .$$
In particular, our definition of de Rham cohomology agrees with the usual one for smooth classical schemes.

\sssec{}

Consider the functor $B^2(\BG_a)$, which represents $\BG_a$-gerbes.  By definition, for a prestack $\CY$, we have
an isomorphism of connective spectra:
$$ \Maps(\CY, B^2(\BG_a)) \simeq 
\tau^{\leq 0}\left(\CMaps_{\QCoh(\CY)}(\CO_{\CY}, \CO_{\CY})[2]\right) \simeq \tau^{\leq 0}(H(\CY)[2]).$$

\medskip

Thus by \propref{p:var to usual}, we obtain:

\begin{cor}
For a prestack $\CY$, groupoid of twistings is given by
$$\Tw(\CY) \simeq \tau^{\leq 2}\left(H_{\dr}(\CY)\underset{H(\CY)}\times \{*\}\right)[2].$$
\end{cor}

\sssec{}  \label{sss:TDO}
Now, suppose that $X$ is a smooth classical scheme.  In this case, we have 
$$ H_{\dr}(X) \simeq \Gamma(X, \Omega^\bullet) $$
where $\Omega^\bullet$ is the complex of de Rham differentials on $X$.  The natural map
$$H_{\dr}(X) \rightarrow H(X)$$
is given by global sections of the projection map $\Omega^\bullet \rightarrow \CO_X$.  Therefore, we have
\begin{multline*}
\Tw(X) \simeq \tau^{\leq 2}\left(H_{\dr}(X)\underset{H(X)}\times \{*\}\right)[2]\simeq  \tau^{\leq 2} 
\left(\Gamma(X, \Omega^\bullet \underset{\CO_X}\times 0)\right)[2] \simeq \\
\simeq \tau^{\leq 2}\left(\Gamma(X,\tau^{\leq 2} (\Omega^\bullet \underset{\CO_X}\times 0))\right)[2].
\end{multline*}

\medskip

The complex $\tau^{\leq 2} (\Omega^\bullet \underset{\CO_X}\times 0)$ identifies with the complex
$$ \Omega^1 \to \Omega^{2,cl} $$
where $\Omega^{2,cl}$ is the sheaf of closed 2-forms (placed in cohomological degree $2$) and the map is the de Rham differential. 

\medskip

Thus, we have that the Picard groupoid of twistings on $X$ is given by
$$ \Tw(X) \simeq \tau^{\leq 2} \left(\Gamma(X, \Omega^1 \to \Omega^{2,cl})\right)[2].$$
In particular, our definition of twistings agrees with the notion of TDO of \cite{BB} for smooth classical schemes.

\ssec{Twisting and the infinitesimal groupoid}

\sssec{}  \label{sss:central ext}

Let $$\CY^1\rightrightarrows \CY^0$$ be a groupoid object in $\on{PreStk}$,
and let $\CY^\bullet$ be the corresponding simplicial object. Let us recall
the notion of central extension of this groupoid object by $\BG_m$. 
(Here $\BG_m$ can be replaced by any commutative
group-object $H\in \on{PreStk}$).

\medskip

By definition, a central extension of $\CY^1\rightrightarrows \CY^0$
by $\BG_m$ is an object of $\Ge_{\BG_m}(|\CY^\bullet|)$, equipped 
with a trivialization of its restriction under 
$$\CY^0\to |\CY^\bullet|.$$

\sssec{}

Informally, the data of such a central extension is a line bundle $\CL$ on $\CY^1$, 
whose pullback under the degeneracy map $\CY^0\to \CY^1$ is trivialized, and 
such that for the three maps 
$$p_{1,2},p_{2,3},p_{1,2}:\CY^2\to \CY^1,$$
we are given an isomorphism
$$p_{1,2}^*(\CL)\otimes p_{2,3}^*(\CL)\simeq p_{1,3}^*(\CL),$$
and such that the further coherence conditions are satisfied. 

\sssec{}

For $\CY\in \on{PreStk}$ consider its infinitesimal groupoid 
$$\CY\underset{\CY_\dr}\times \CY\rightrightarrows \CY.$$

By definition, a twisting on $\CY$ gives rise to a central extension
of its infinitesimal groupoid by $\BG_m$. 

\medskip

Conversely, from
\lemref{l:cl formally smooth}, we obtain:

\begin{cor}  \label{c:tw via inf}
Assume that $\CY$ is classically formally smooth. Then the above functor
$$\Tw(\CY)\to \{\text{{\rm Central extensions of the infinitesimal groupoid of $\CY$ by $\BG_m$}}\}$$
is an equivalence.
\end{cor}

\ssec{Twistings on indschemes}  \label{ss:tw on indsch}

\sssec{}

Let $\CX$ be an object of $\dgindSch_{\on{laft}}$. 
We will show that the assertion of \corref{c:tw via inf} holds for $\CX$:

\begin{prop} \label{p:tw on indsch}
The functor 
$$\Tw(\CX)\to \{\text{{\rm Central extensions of the infinitesimal groupoid of $\CX$ by $\BG_m$}}\}$$
is an equivalence.
\end{prop}

The rest of this subsection is devoted to the proof of \propref{p:tw on indsch}. 

\sssec{Step 1}

By \corref{c:twistings laft}(b), we have to show the following: 

\medskip

\noindent For every 
$S\in (\affdgSch_{\on{aft}})_{/\CX_\dr}$ a datum of $\BG_m$-gerbe on $S$, equipped
with a trivialization of its pullback to $S\underset{\CX_\dr}\times \CX$, is equivalent to that
of a $\BG_m$-gerbe on the simplicial prestack 
$$S^\bullet:=S\underset{\CX_\dr}\times (\CX^\bullet/\CX_\dr),$$
equipped with a trivialization over $0$-simplices, i.e., $S\underset{\CX_\dr}\times \CX$. 

\sssec{Step 2}

Note that the simplicial prestack $^{cl,red}(S^\bullet)$ is constant with value $^{red,cl}S$.
Hence, by \lemref{c:gerbes as B2}, we can consider $\BG_a$-gerbes instead of $\BG_m$-gerbes.

\medskip

Hence, it is enough to show that the map
$$\CMaps_{\QCoh(S)}(\CO_S,\CO_S)\to \on{Tot}\left(\CMaps_{\QCoh(S^\bullet)}(\CO_{S^\bullet},
\CO_{S^\bullet})\right)$$
is an isomorphism in $\Vect$. 

\sssec{Step 3}

Note that for any $\CX'\in \dgindSch_{\on{laft}}$, the canonical map
$$\CMaps_{\QCoh(\CX')}(\CO_{\CX'},\CO_{\CX'})\to 
\CMaps_{\IndCoh(\CX')}(\omega_{\CX'},\omega_{\CX'})$$
is an isomorphism. This follows, e.g., from the corresponding assertion
for DG schemes, i.e., \lemref{l:Xi vee of omega}.

\medskip

Hence, it is enough to show that the map
$$\CMaps_{\IndCoh(S)}(\omega_S,\omega_S)\to \on{Tot}\left(\CMaps_{\IndCoh(S^\bullet)}(\omega_{S^\bullet},
\omega_{S^\bullet})\right)$$
is an isomorphism.

\sssec{Step 4}

Note that $S^\bullet$ identifies with the \v{C}ech nerve of
the map 
\begin{equation} \label{e:base change of inf}
S\underset{\CX_\dr}\times \CX\to S.
\end{equation}

As in the proof of \propref{p:inf descent for right}, all $S^i$ belong to $\dgindSch$,
and the morphism \eqref{e:base change of inf} is ind-proper and surjective. 

\medskip

Now, the desired assertion follows from the descent for $\IndCoh$ under ind-proper and surjective
maps of DG indschemes, see \cite[Lemma 2.10.3]{IndSch}.

\section{Twisted crystals}

In this section we will show how a data of a twisting allows to modify the categories
of left and right crystals. The main results say that ``not much really changes."

\ssec{Twisted left crystals}

In this subsection we do not assume that our DG schemes and prestacks are locally
almost of finite type.

\sssec{}

Let $\CY$ be a prestack. Consider the category
$\on{PreStk}_{/\CY}$, and the functor
$$\QCoh_{\affdgSch_{/\CY}}:(\affdgSch_{/\CY})^{\on{op}}\to \StinftyCat_{\on{cont}}.$$
The group-stack $\on{pt}/\BG_m$ acts on $\QCoh$ via tensoring by line bundles.

\medskip

Let $\CG$ be a $\BG_m$-gerbe on $\CY$.  Then $\CG$ gives a twist of the functor $\QCoh_{\affdgSch_{/\CY}}$ via the action of 
$\on{pt}/\BG_m$ on $\QCoh$.  This defines a functor
$$\QCoh^\CG_{\affdgSch_{/\CY}}:(\affdgSch_{/\CY})^{\on{op}}\to \StinftyCat_{\on{cont}}.$$

\sssec{}

In particular, if $T$ is a twisting on $\CY$, we obtain a functor
$$\QCoh^T_{\affdgSch_{/\CY_\dr}}:(\affdgSch_{/\CY_\dr})^{\on{op}}\to \StinftyCat_{\on{cont}}.$$

Let $\QCoh^T_{\affdgSch_{/\CY}}$ be its restriction
along the map
$$(\affdgSch_{/\CY})^{\on{op}}\to (\affdgSch_{/\CY_\dr})^{\on{op}}.$$

By construction, $\QCoh^T_{\affdgSch_{/\CY}}$
is canonically isomorphic to $\QCoh_{\affdgSch_{/\CY}}$.

\sssec{}

More generally, we can consider the functor
$$\QCoh^T_{\on{PreStk}_{/\CY_\dr}}:(\on{PreStk}_{/\CY_\dr})^{\on{op}}\to \StinftyCat_{\on{cont}},$$
which is the right Kan extension of $\QCoh^T_{\affdgSch_{/\CY_\dr}}$
along
$$(\affdgSch_{/\CY_\dr})^{\on{op}}\hookrightarrow (\on{PreStk}_{/\CY_\dr})^{\on{op}}.$$

The restriction $\QCoh^T_{\on{PreStk}_{/\CY}}$ of $\QCoh^T_{\on{PreStk}_{/\CY_\dr}}$ along
$$\on{PreStk}_{/\CY}\to \on{PreStk}_{/\CY_\dr}$$
is canonically isomorphic to $\QCoh_{\on{PreStk}_{/\CY}}$.

\sssec{}

For a twisting $T$ on a prestack $\CY$, the category
of $T$-twisted left crystals on $\CY$ is defined as
$$\Crys^{T,l}(\CY):= \QCoh^T(\CY_\dr).$$
Explicitly, we have
$$\Crys^{T,l}(\CY)=\underset{S\in (\affdgSch_{/\CY_\dr})^{\on{op}}}{lim}\, \QCoh^T(S).$$

\sssec{}

More generally, we define the functor
$$\Crys^{T,l}_{\on{PreStk}_{/\CY_\dr}}:(\on{PreStk}_{/\CY_\dr})^{\on{op}}\to \StinftyCat_{\on{cont}},$$
as the composite $\QCoh^T_{\on{PreStk}_{/\CY_\dr}} \circ\  \dr$.
The analogue of \corref{c:RKE for left} holds for this functor.

\sssec{}

We have a canonical natural transformation
$$\oblv^{T,l}:\Crys^{T,l}_{\on{PreStk}_{/\CY_\dr}}\to \QCoh^T_{\on{PreStk}_{/\CY_\dr}}.$$

For an individual $\CY'\in \on{PreStk}_{/\CY_\dr}$, we denote the resulting functor
$$\Crys^{T,l}(\CY')\to \QCoh^T(\CY')$$
by $\oblv^{T,l}(\CY')$.

\sssec{}

Let $\Crys^{T,l}_{\on{PreStk}_{/\CY}}$ denote the restriction of $\Crys^{T,l}_{\on{PreStk}_{/\CY_\dr}}$
along $\on{PreStk}_{/\CY}\to \on{PreStk}_{/\CY_\dr}$. 

\medskip

By a slight abuse of notation we shall use the same symbol $\oblv^{T,l}$ to denote the resulting
natural transformation
$$\Crys^{T,l}_{\on{PreStk}_{/\CY}}\to \QCoh_{\on{PreStk}_{/\CY}}.$$

\ssec{Twisted right crystals}

At this point, we reinstate the assumption that all DG schemes and prestacks
are locally almost of finite type for the rest of the paper.

\sssec{}

Let $\CY$ be an object of $\on{PreStk}_{\on{laft}}$, and let $\CG$ be a $\BG_m$-gerbe on $\CY$.

\medskip

The action of $\QCoh_{\on{PreStk}_{\on{laft}}}$ on $\IndCoh_{\on{PreStk}_{\on{laft}}}$ (see \cite[Sect. 10.3]{IndCoh})
allows to define the functor
$$\IndCoh^\CG_{(\on{PreStk}_{\on{laft}})_{/\CY}}:((\on{PreStk}_{\on{laft}})_{/\CY})^{\on{op}}\to \StinftyCat_{\on{cont}},$$
with properties analogous to those of 
$$\IndCoh_{(\on{PreStk}_{\on{laft}})_{/\CY}}:=\IndCoh_{\on{PreStk}_{\on{laft}}}|_{(\on{PreStk}_{\on{laft}})_{/\CY}}.$$

\sssec{}

In particular, for $T\in \Tw(\CY)$, we have the functor
$$\Crys^{T,r}_{(\on{PreStk}_{\on{laft}})_{/\CY_\dr}}:((\on{PreStk}_{\on{laft}})_{/\CY_\dr})^{\on{op}}\to  \StinftyCat_{\on{cont}},$$
and the natural transformations $\oblv^{T,l}$
$$\Crys^{T,r}_{(\on{PreStk}_{\on{laft}})_{/\CY_\dr}}\to \IndCoh^T_{(\on{PreStk}_{\on{laft}})_{/\CY_\dr}} \text{ and }
\Crys^{T,r}_{(\on{PreStk}_{\on{laft}})_{/\CY}}\to \IndCoh_{(\on{PreStk}_{\on{laft}})_{/\CY}}.$$  

The analogues of Corollaries \ref{c:only aft right} and \ref{c:RKE for right aft} and
Lemmas \ref{l:right on cl formally smooth} and \ref{l:cons for r} hold
for $\CY'\in (\on{PreStk}_{\on{laft}})_{/\CY_\dr}$, with the same proofs. 

\ssec{Properties of twisted crystals}

As was mentioned above, all DG schemes and prestacks are assumed locally almost of finite type.

\medskip

Let $\CY$ be a fixed object of $\on{PreStk}_{\on{laft}}$, and $T\in \Tw(\CY)$.

\begin{rem}  \label{r:loc triv}
In general, results about crystals do not automatically hold for twisted crystals.  In some of our proofs, we needed to embed a given affine DG scheme $X$ into a smooth classical scheme $Z$.  In the case of twisted crystals, the problem is that we might not
be able to find such a $Z$ which also maps to $\CY$ (or even $\CY_\dr$).

\medskip

However, there is a large family of examples (which covers all the cases that have appeared in applications so far), 
where the extension of the results is automatic: namely, when $T$ is such that
its restriction to any $S\in (\affdgSch_{\on{aft}})_{/\CY}$ is locally trivial in the Zariski or \'etale topology
(see also Remark \ref{r:non loc triv}). This is the case for twistings of the form 
$\CL^{\otimes a}$ for $\CL\in \Pic(\CY)$ and $a\in k$, and tensor products thereof.

\end{rem}

\sssec{}

The analogues of Corollaries \ref{c:dr laft} and \ref{c:RKE for left aft} and Lemma
\ref{l:left on cl formally smooth} hold for twisted left crystals,
with the same proofs.

\medskip

Furthermore, Kashiwara's lemma holds for both left and right twisted crystals,
also with the same proof.

\medskip

Finally, note that there exists a canonical natural transformation
\begin{equation} \label{e:twisted upsilon}
\Upsilon:
\Crys^{T,l}_{(\on{PreStk}_{\on{laft}})_{/\CY_\dr}}\to \Crys^{T,r}_{(\on{PreStk}_{\on{laft}})_{/\CY_\dr}}.
\end{equation}

\begin{prop}
The natural transformation \eqref{e:twisted upsilon} is an equivalence.
\end{prop}

\begin{proof}

The argument is the same as that of \propref{p:left to right}: 

\medskip

We do not need the smooth classical scheme $Z$ to map to $\CY$. 
Rather, we use the fact that if $Y$ is the completion of a smooth
classical scheme $Z$ along a Zariski-closed subset, and $\CG$ is a 
$\BG_m$-gerbe on $Y$, which is trivial over $^{cl,red}Y$, then the functor
$$\Upsilon_Y:\QCoh^\CG(Y)\to \IndCoh^\CG(Y)$$
is an equivalence. The latter follows from the corresponding fact
in the non-twisted situation (proved in the course of the proof of
\propref{p:left to right}), since $\CG$ is (non-canonically) trivial.

\end{proof} 

As a corollary, we obtain that the analog of \lemref{l:cons} holds 
in the twisted case as well.  

\sssec{}

Hence, for $\CY'\in (\on{PreStk}_{\on{laft}})_{/\CY_\dr}$ we can regard crystals on $\CY'$ as a single category, 
$\Crys^T(\CY')$, endowed with two forgetful
functors

\begin{gather} 
\xy
(-20,0)*+{\QCoh^T(\CY')}="A";
(20,0)*+{\IndCoh^T(\CY')}="B";
(0,20)*+{\Crys^T(\CY').}="C";
{\ar@{->}_{\oblv^{T,l}_{\CY'}} "C";"A"};
{\ar@{->}^{\oblv^{T,r}_{\CY'}} "C";"B"};
{\ar@{->}_{\Upsilon_{\CY'}} "A";"B"};
\endxy
\end{gather}

For $\CY'\in (\on{PreStk}_{\on{laft}})_{/\CY}$, the above forgetful functors map to non-twisted sheaves:

\begin{gather} 
\xy
(-20,0)*+{\QCoh(\CY')}="A";
(20,0)*+{\IndCoh(\CY')}="B";
(0,20)*+{\Crys^T(\CY').}="C";
{\ar@{->}_{\oblv^{T,l}_{\CY'}} "C";"A"};
{\ar@{->}^{\oblv^{T,r}_{\CY'}} "C";"B"};
{\ar@{->}_{\Upsilon_{\CY'}} "A";"B"};
\endxy
\end{gather}

\sssec{}

Let $\CX\in (\dgindSch_{\on{laft}})_{/\CY_\dr}$. The analogue of \propref{p:inf descent for right}
holds with no change. In particular, we obtain a functor
$$\ind^{T,r}_\CX:\IndCoh^{T}(\CX)\to \Crys^{T,r}(\CX)$$
left adjoint to $\oblv^{T,r}_\CX$. 

\medskip

Similarly, the analogue of \propref{p:groupoud for left} holds in the present context as well.

\sssec{}

The following observation will be useful in the sequel:

\medskip

Let $X$ be an affine DG scheme (or an ind-affine DG indscheme) over $\CY_\dr$. Choose
a trivialization of the resulting $\BG_m$-gerbe on $X$. This choice defines an identification
$$\IndCoh^T(X)\overset{\alpha}\simeq \IndCoh(X).$$

\begin{lem}  \label{l:twisted monad}
The monad $\oblv^{T,r}_X\circ \ind^{T,r}_X$, regarded as a functor (without the monad structure) 
$$\IndCoh(X)\overset{\alpha^{-1}}\simeq \IndCoh^T(X)\to \IndCoh^T(X)\overset{\alpha}\simeq \IndCoh(X),$$
is non-canonically isomorphic to $\oblv^{r}_X\circ \ind^{r}_X$.
\end{lem}

\begin{proof}

First, we observe that the analogue of \propref{p:identify D} holds; namely,  
the object of $\IndCoh(X\times X)$ that defines the functor 
$\oblv^{T,r}_X\circ \ind^{T,r}_X$ is given by
$$(\wh\Delta_X)^{\IndCoh}_*\left(\CL\otimes (\omega_{X\underset{X_\dr}\times X})\right),$$
where $\CL$ is the line bundle on $X\underset{X_\dr}\times X$ corresponding to $T$ and $\alpha$
as in \secref{sss:central ext}.

\medskip

By construction, $\CL$ is trivial when restricted to $X\hookrightarrow X\underset{X_\dr}\times X$.
Now, since $X$ is affine, this implies that $\CL$ can be trivialized on all of 
$X\underset{X_\dr}\times X$.

\end{proof}

\ssec{t-structures on twisted crystals}
%
%

As in the previous subsection, let $\CY$ be a fixed object of $\on{PreStk}_{\on{laft}}$, and $T\in \Tw(\CY)$.

\sssec{}

If $X$ is a DG scheme and $\CG$ is a $\BG_m$-gerbe on it, the twisted categories
$\QCoh^\CG(X)$ and $\IndCoh^\CG(X)$ have natural t-structures with properties 
analogous to those of their usual counterparts $\QCoh(X)$ and $\IndCoh(X)$.

\medskip
In particular, we have the ``left'' t-structure on $\Crys^{T,l}(\CY')$ for any
$\CY'\in (\on{PreStk}_{\on{laft}})_{/\CY_\dr}$. (This t-structure can be defined without the locally almost of 
finite type assumption on 
either $\CY$ or $\CY'$.)

\medskip
The t-structure on twisted $\IndCoh$ on DG schemes allows us to define
a t-structure
on $\IndCoh^\CG(\CX)$, where $\CX$ is a DG indscheme.  We can then define the ``right'' t-structure on the category $\Crys^{T,r}(\CX)$.

\sssec{} 

We observe that \propref{p:cl embed exact} renders to the twisted context with no change. We now
claim:

\begin{prop} \label{p:t and tw crys} Let $X$ be a quasi-compact DG scheme mapping to $\CY_\dr$. 

\smallskip

\noindent{\em(a)} The functor $\ind^{T,r}_X$ is t-exact.

\smallskip

\noindent{\em(b)}  For a quasi-compact scheme $X$, the functor $\oblv^{T,r}_X$ is of
bounded cohomological amplitude. 

\end{prop}

\begin{proof}

The functor $\ind^{T,r}_X$ is right t-exact, since its right adjoint $\oblv^{T,r}_X$
is left t-exact. By the definition of the ``right" t-structure, the 
left t-exactness of $\ind^{T,r}_X$ is equivalent to the same property of
the composition $\oblv^{T,r}_X\circ \ind^{T,r}_X$. 

\medskip

The assertion is Zariski-local, so we can assume that $X$ is affine. Now, the fact that
the functor $\oblv^{T,r}_X\circ \ind^{T,r}_X$ is left t-exact follows
from \lemref{l:twisted monad} and the fact that the analogous assertion
holds in the non-twisted case.

\medskip

Since $\IndCoh^T(X)^{\leq 0}$ generates $\Crys^{T,l}(X)^{\leq 0}$
via the functor $\ind^{T,r}_X$, in order to show that the cohomological
amplitude of $\oblv^{T,r}_X$ is bounded from above, it suffices to show
the same for $\oblv^{T,r}_X\circ \ind^{T,r}_X$. Again, the assertion follows
from \lemref{l:twisted monad}.

\end{proof}

\sssec{}

We now claim:

\begin{prop} \label{p:tw bounded}
Let $X$ be a quasi-compact DG scheme mapping to $\CY_\dr$. 

\smallskip

\noindent{\em(a)} The ``left" and ``right" t-structures on $\Crys^T(X)$ differ by finite cohomological amplitude.

\smallskip

\noindent{\em(b)} The functor $\oblv_X^{T,l}:\Crys^T(X)\to \QCoh(X)$ is of bounded
cohomological amplitude.\footnote{By point (a) this statement does not depend on which
of the two t-structures we consider on $\Crys^T(X)$.}

\end{prop}

\sssec{}

We shall first prove the following:

\medskip

Let $i:X\to Z$ be a closed embedding, where $Z$ is a smooth classical scheme. 
Let $Y$ be the formal completion of $Z$ along $X$. 

\begin{lem}  \label{l:tw on formally smooth}
The functor
$$\oblv^{T,r}_Y:\Crys^{T,r}(Y)\to \IndCoh^T(Y)$$
is t-exact.
\end{lem}

\begin{proof}

As in the proof of \propref{p:t and tw crys}, it suffices to show that the functor 
$$\oblv^{T,r}_Y\circ \ind^{T,r}_Y:\IndCoh^T(Y)\to \IndCoh^T(Y)$$
is t-exact. The assertion is Zariski-local, so we can assume that $X$ is affine.
Now, as in the proof of \propref{p:t and tw crys}, the
functor in question is non-canonically isomorphic to the non-twisted version:
$\oblv^{r}_Y\circ \ind^{r}_Y$, and the latter is known to be t-exact by
\propref{p:t and crys sm}(a).

\end{proof}

\sssec{Proof of \propref{p:tw bounded}}

The assertion is Zariski-local, so we can assume that $X$ is affine and embed it into a smooth classical 
scheme $Z$. Let $Y$ denote the formal completion of $X$ in $Z$. By definition, $T$ defines
a $\BG_m$-gerbe $\CG$ on $Y$. Let $'i$ denote the corresponding map $X\to Y$.

\medskip

To prove point (a), by \lemref{l:t structure on form compl} (whose 2nd proof is applicable in the twisted case),
we can replace $X$ by $Y$, and it suffices to show that the discrepancy
between the two t-structures on $\Crys^T(Y)$ is finite. By \propref{p:left t-struct}
(applied in the twisted case) and \lemref{l:tw on formally smooth}, it suffices to show
that the functor
$$\Psi_Y:\QCoh^T(Y)\to \IndCoh^T(Y)$$
is of bounded cohomological amplidude. This is equivalent to the corresponding fact for 
$$\Psi_Y:\QCoh(Y)\to \IndCoh(Y),$$
which in turn follows from the corresponding fact for $Z$.

\medskip

Point (b) follows from the fact that the functor
$$'i^*:\QCoh^T(Y)\to \QCoh^T(X)$$
is of bounded amplitude, which is again equivalent to the corresponding fact for
$$'i^*:\QCoh(Y)\to \QCoh(X),$$
and the latter follows from the corresponding fact for $Z$.

\qed

\sssec{}

The results concerning the ``coarse" forgeftul and induction functors, 
established in \secref{ss:coarse} for untwisted crystals, render automatically to the twisted situation.

\sssec{}

Our current goal is to show:

\begin{prop} \label{p:left complete twisted} Let $X$ be a quasi-compact DG scheme mapping to $\CY_\dr$. 

\smallskip

\noindent{\em(a)} 
The ``right" t-structure on $\Crys^{T,r}(X)$ is left-complete. 

\smallskip

\noindent{\em(b)} For $X$ affine, the 
natural functor $D(\Crys^{T,r}(X)^\heartsuit)^+\to \Crys^{T,r}(X)^+$,  
where the heart is taken with respect to the ``right" t-structure, 
uniquely extends to an an equivalence
$$D(\Crys^{T,r}(X)^\heartsuit)\to \Crys^{T,r}(X).$$ 

\end{prop}

\sssec{Proof of \propref{p:left complete twisted}(a)}

Again, the assertion is Zariski-local, and we retain the setting of the proof
of \propref{p:tw bounded}.

\medskip

It suffices to exhibit a collection of objects 
$$\CP_\alpha\in \Crys^{T,r}(X)$$ that generate $\Crys^{T,r}(X)$ and that are of
bounded Ext dimension, i.e., if for each $\alpha$ there exists an integer $k_\alpha$ such that
$$\Hom_{\Crys^{T}(X)}(\CP_\alpha,\CM)=0 \text{ if } \CM\in \Crys^T(X)^{<-k_\alpha}.$$

\medskip

We realize $\Crys^{T,r}(X)$ as $\Crys^{T,r}(Y)$. By \lemref{l:t structure on form compl}
and \lemref{l:tw on formally smooth},
the t-structure on $\Crys^{T,r}(X)\simeq \Crys^{T,r}(Y)$ is characterized by the property that
$$\CM\in \Crys^{T,r}(Y)^{\geq 0}\, \Leftrightarrow \, \oblv_Y^{T,r}(\CM)\in \IndCoh^T(\CY)^{\geq 0}.$$

\medskip

We take $\CP_\alpha$ to be of the form $\ind^{T,r}(\CF)$ for $\CF\in \Coh^T(Y)^\heartsuit$. 
To prove the required vanishing of Exts, we need to show that for $\CM\in \Crys^{T,r}(Y)^{\ll 0}$,
$$\Hom_{\IndCoh^T(Y)}(\CF,\oblv_Y^{T,r}(\CM))=0.$$

\medskip

However, this follows from the fact that the category $\IndCoh^T(Y)$
has finite cohomological dimension with respect to its t-structure:\footnote{We refer the reader to the footnone in
Remark \ref{r:fin c d} where we explain what we mean by this.} indeed, the category in question in non-canonically
equivalent to $\IndCoh(Y)$, and the cohomological dimension of the latter is bounded
by that of $\IndCoh(Z)$. 

\sssec{Proof of \propref{p:left complete twisted}(b)}

We keep the notations from the proof of point (a).

\medskip

As in the proof of \propref{p:relation to abelian}, given what we have
shown in point (a), we only have to verify that for $\CM_1,\CM_2\in \Crys^T(X)^\heartsuit$ and any $k\geq 0$, the map
$$\on{Ext}^k_{\Crys^T(X)^\heartsuit}(\CM_1,\CM_2)\to \Hom_{\Crys^T(X)}(\CM_1,\CM_2[k])$$
is an isomorphism. 

\medskip

For that it suffices to show that the category $\Crys^{T,r}(X)^\heartsuit$ contains
a pro-projective generator of $\Crys^{T,r}(X)$, i.e., that there exists a filtered inverse family with surjective
maps $\CP_\alpha\in \Crys^{T,r}(X)^\heartsuit$, such that the functor 
$$\underset{\alpha}{colim}\, \CMaps_{\Crys^{T,r}(X)}(\CP_\alpha,-)$$
is t-exact and conservative on $\Crys^{T,r}(X)$. 

\medskip

We take $\CP_\alpha$ to be 
$$\ind^{T,r}_Y(\CO_{X_n})\in \Crys^{T,r}(Y)^\heartsuit\simeq \Crys^{T,r}(X)^\heartsuit,$$
where $X_n$ is the $n$-th infinitesimal neighborhood of $^{cl,red}X$ in $Z$. 

\qed

\ssec{Other results}

\sssec{Twisted crystals and twisted D-modules}

Let $X$ be a smooth classical scheme. We have seen in \secref{sss:TDO} that the Picard category of
twistings on $X$ is equivalent to that of TDO's on $X$.

\medskip

Given a twisting $T$, and the corresponding TDO, denoted $\on{Diff}^T_X$,
there exists a canonical equivalence 
$$\Crys^{T,l}(X)\simeq \Dmod^{T,l}(X),$$
which commutes with the forgetful functors to $\QCoh(X)$, and similarly for twisted right crystals.
The proof is either an elaboration of the strategy indicated in \secref{sss:Groth}, or one using
\secref{ss:rel to diff op}.  

\sssec{}

The relation between twisted D-modules and modules over a TDO can be extended to the
case when instead of a smooth classical scheme $X$, we are dealing with a formal completion $Y$ 
of a DG scheme $X$ inside a smooth classical scheme $Z$. 

\medskip

This allows to prove:

\begin{prop}  \label{p:behavior twisted crystals} Let $X$ be a quasi-compact DG scheme over $\CY_\dr$.

\smallskip

\noindent{\em(a)}
The abelian category $\Crys^{T,r}(X)^\heartsuit$ is locally Noetherian.

\smallskip

\noindent{\em(b)}
$\Crys^{T,r}(X)$ has finite cohomological dimension with respect to its t-structure. 
\end{prop}

(We refer the reader to the footnotes in Remark \ref{r:fin c d} where we explain what we mean
by the properties asserted in points (a) and (b) of the proposition.)

\end{document}